\documentclass{amsart}

\usepackage{graphicx}
\usepackage[centertags]{amsmath}
\usepackage{amsfonts}
\usepackage{amssymb}
\usepackage{amsthm}
\usepackage{newlfont}
\usepackage{hyperref}
\usepackage{tikz}
\usepackage{tikz-cd}
\usepackage{subcaption}
\usepackage{enumitem}
\usepackage{multirow}
\usepackage{MnSymbol}
\usepackage{rotating}

\newcommand{\movapp}[6]{
\fill[gray!10,opacity=.5] (#4,#5,#6)-- (#1+#4,#5,#6) -- (#1+#4,#3+#5,#6)  -- (#4,#3+#5,#6) --cycle;
\fill[gray!10,opacity=.5] (#4,#5,#2+#6)-- (#1+#4,#5,#2+#6) -- (#1+#4,#3+#5,#2+#6)  -- (#4,#3+#5,#2+#6) --cycle;
\fill[gray!10,opacity=.5] (#4,#3+#5,#6)-- (#4,#3+#5,#2+#6) -- (#1+#4,#3+#5,#2+#6) -- (#1+#4,#3+#5,#6)--cycle;
\fill[gray!10,opacity=.5] (#4,#5,#6)-- (#4,#5,#2+#6) -- (#1+#4,#5,#2+#6) -- (#1+#4,#5,#6)--cycle; 
\draw[] (#4,#5,#2+#6) -- (#1+#4,#5,#2+#6) -- (#1+#4,#3+#5,#2+#6) --(#4,#3+#5,#2+#6) --(#4,#5,#2+#6)
        (#1+#4,#5,#2+#6) -- (#1+#4,#5,#6)  -- (#1+#4,#3+#5,#6) --(#4,#3+#5,#6) -- (#4,#3+#5,#2+#6)    
        (#1+#4,#3+#5,#2+#6) -- (#1+#4,#3+#5,#6);
\draw[dashed] (#4,#5,#6) -- (#4,#5,#2+#6) (#4,#5,#6)-- (#1+#4,#5,#6) (#4,#5,#6)-- (#4,#3+#5,#6);

}

\newcommand{\oppapp}[6]{
\fill[white,opacity=0] (#4,#5,#6)-- (#1+#4,#5,#6) -- (#1+#4,#3+#5,#6)  -- (#4,#3+#5,#6) --cycle;
\fill[white,opacity=0] (#4,#5,#2+#6)-- (#1+#4,#5,#2+#6) -- (#1+#4,#3+#5,#2+#6)  -- (#4,#3+#5,#2+#6) --cycle;
\fill[white,opacity=0] (#4,#3+#5,#6)-- (#4,#3+#5,#2+#6) -- (#1+#4,#3+#5,#2+#6) -- (#1+#4,#3+#5,#6)--cycle;
\fill[white,opacity=0] (#4,#5,#6)-- (#4,#5,#2+#6) -- (#1+#4,#5,#2+#6) -- (#1+#4,#5,#6)--cycle; 
\draw[] (#4,#5,#2+#6) -- (#1+#4,#5,#2+#6) -- (#1+#4,#3+#5,#2+#6) --(#4,#3+#5,#2+#6) --(#4,#5,#2+#6)
        (#1+#4,#5,#2+#6) -- (#1+#4,#5,#6)  -- (#1+#4,#3+#5,#6) --(#4,#3+#5,#6) -- (#4,#3+#5,#2+#6)    
        (#1+#4,#3+#5,#2+#6) -- (#1+#4,#3+#5,#6);

}    

\newcommand{\bakapp}[6]{
\draw[dashed] (#4,#5,#6) -- (#4,#5,#2) (#4,#5,#6)-- (#1+#4,#5,#6) (#4,#5,#6)-- (#4,#3+#5,#6);

}

\usepackage{mathtools} 

\newtheorem{thm}{Theorem}[section]
\newtheorem{cor}[thm]{Corollary}
\newtheorem{lem}[thm]{Lemma}

\theoremstyle{definition}
\newtheorem{defn}[thm]{Definition}

\theoremstyle{remark}

\newtheorem{example}[thm]{Example}
\numberwithin{equation}{section}

\newtheorem{ques}[thm]{Question}

\newcommand{\R}{\mathbb{R}}

\newcommand{\Z}{\mathbb{Z}}

\newcommand{\bP}{\mathbb{P}}
\newcommand{\dom}{\mathrm{dom}}

\newcommand{\fzn}{F(2^{\Z^n})}

\newcommand{\sQ}{\mathcal{Q}}
\newcommand{\sA}{\mathcal{A}}
\newcommand{\sB}{\mathcal{B}}
\newcommand{\sC}{\mathcal{C}}
\newcommand{\sP}{\mathcal{P}}

\newcommand{\sT}{\mathcal{T}}

\newcommand{\sm}{\setminus}

\DeclareMathOperator{\Int}{int}

\makeatletter
\def\Ddots{\mathinner{\mkern1mu\raise\p@
\vbox{\kern7\p@\hbox{.}}\mkern2mu
\raise4\p@\hbox{.}\mkern2mu\raise7\p@\hbox{.}\mkern1mu}}
\makeatother

\begin{document}

\title{On regulated partitions}
\author{Su Gao}
\address{School of Mathematical Sciences and LPMC, Nankai University, Tianjin 300071, P.R. China}
\email{sgao@nankai.edu.cn}
\thanks{The first author acknowledges the partial support of his research by the National Natural Science Foundation of China (NSFC) grants 12271263 and 12250710128. Research of the second author was partially supported by the U.S. NSF grants DMS-1800323 and DMS-2452099}

\author{Steve Jackson}
\address{Department of Mathematics, University of North Texas, 1155 Union Circle \#311430, Denton, TX 76203, U.S.A.}
\email{jackson@unt.edu}

\begin{abstract}   This paper considers the combinatorics of continuous and Borel rectangular
partitions of free actions of $\mathbb{Z}^n$ on $0$-dimensional Polish spaces, specifically the free
part $F(2^{\mathbb{Z}^n})$ of the shift action of $\mathbb{Z}^n$ on the space $2^{\mathbb{Z}^n}$. This is done through the study of a corresponding notion of regulated partitions of $\mathbb{R}^n$. The
main concepts studied are the 
continuous and Borel {\em regulation} numbers of the partition. This is defined as the maximum number of
rectangles in the corresponding regulated partition that can intersect in a point. The continuous and Borel regulation numbers
$\gamma_c$, $\gamma_B$ are the minimum possible values of these numbers as we range over continuous 
(respectively Borel) rectangular partitions of $F(2^{\mathbb{Z}^n})$. It is shown that 
for $n=2$ that $\gamma_c=\gamma_B=3$, and for $n \geq 3$ that 
$n+2\leq \gamma_B \leq \gamma_c \leq 3\cdot 2^{n-2}$. For $n=3$ we improve this to $\gamma_c=\gamma_B=5$. This shows a striking difference between the Borel combinatorics of dimension $n=2$ and dimensions $n>2$. 
\end{abstract}

\maketitle


\section{Introduction}

This paper can be regarded as a contribution to the study of Borel combinatorics of countable abelian group actions. The study of Borel combinatorics in general began with the seminal paper of Kechris, Solecki and Todorcevic \cite{KST}. Since its inception, the study has repeatedly encountered the phenomenon that there is often a difference between the answers to questions in Borel combinatorics and those in classical combinatorics with a countable domain. For instance, in the context of a marked group $G$ (a group with a specific finite set of generators), combinatorial questions about the Cayley graph of $G$ often have different answers than Borel combinatorial questions about the Schreier graph of a free Borel action of $G$ on a Polish space $X$. In this paper we give one more instance of this kind of difference. 

The specific combinatorial question we consider in this paper comes from another thread of research, namely the study of Borel marker constructions for countable Borel equivalence relations. This is a technique originally designed for hyperfiniteness proofs (see, e.g.\ \cite{JKL} and \cite{GJ2015}). More recently, it has played a crucial role in answering some Borel combinatorial questions about abelian group actions (see \cite{GJKSBorel}).

More specifically, we consider the so-called {\em Borel rectangular partitions} of a Polish space $X$ where the additive group $\mathbb{Z}^n$ freely acts on $X$ in a Borel manner. These are Borel partitions $\mathcal{P}$ of $X$ where each set $P$ in the partition $\mathcal{P}$ is of the form $R\cdot x$ for some $x\in X$ and some $R$ which is a rectangle in $\mathbb{Z}^n$. Since the action is free, the Schreier graph on $X$, when restricted to each equivalence class of the action, is essentially a copy of the Cayley graph on $\mathbb{Z}^n$, and a Borel rectangular partition $\mathcal{P}$ on $X$, when restricted to each equivalence class, gives rise to a rectangular partition $\mathcal{Q}$ on $\mathbb{Z}^n$. 

To study a rectangular partition $\mathcal{Q}$ on $\mathbb{Z}^n$, we replace each element of $\mathbb{Z}^n$ by an $n$-dimensional cube of side length $1$ centered about the element, and thus obtain a division of $\mathbb{R}^n$ into $n$-dimensional rectangles with disjoint interiors, which we call {\em regulated partitions}. Technically, we require each region in a regulated partition to be finite (i.e., of finite side lengths). In dimension $2$, these have been called {\em rectangulations} in the literature and have been an object of research for combinatorial geometry (see, e.g.\ \cite{Re} and \cite{Kl}). For regulated partitions of $\mathbb{R}^n$, we are particularly interested in counting how many different regions in the division can come together to a point (the {\em regulation number} of a point with respect to a regulated partition). We will show in Lemma~\ref{lem:minimalregulated} that for any regulated partition $\mathcal{P}$ of $\mathbb{R}^n$ there exists some point whose regulation number with respect to $\mathcal{P}$ is at least $n+1$. Thus, when every point has a regulation number at most $n+1$, we will call the regulated partition $\mathcal{P}$ {\em minimal}. 

Minimal regulated partitions of $\mathbb{R}^n$ have a remarkable property (Theorem~\ref{union}) that around any point $x$ there is a sufficiently small neighborhood $U$ such that any subcollection of the induced regulated partition within $U$ unions up to a set which is homeomorphic to an $n$-dimensional rectangle. This is a property that would simplify some Borel marker constructions (as in \cite{GJKSBorel}) and yield marker sets with stronger topological regularity properties. In dimension $2$, minimal regulated partitions have been called {\em generic rectangulations} (see \cite{Re}), while the term of {\em minimal rectangulation} has been used (e.g.\ in \cite{Kl}) to refer to a rectangulation of a finite polygonal region with a minimal number of rectangles. 
The above topological regularity property for generic rectangulations in dimension $2$ 
has been noted in \cite{Ri}. The interest in minimal rectangulations was primarily motivated by the observation that finding minimal rectangulations can be performed in polynomial time (see e.g.\ \cite{Kl}, \cite{Epp}, \cite{LLLMP}), which explains why the literature about them is mostly in computer science. 

We thus call a rectangular partition on $\mathbb{Z}^n$ {\em minimal} if its induced regulated partition of $\mathbb{R}^n$ is minimal. Moreover, we also refer to a rectangular partition on $\mathbb{Z}^n$ as a {\em regulated partition}. This terminology is similarly transferred to the context of rectangular partitions of a free Borel action of $\mathbb{Z}^n$. Our main result of this paper, which is somewhat surprising, is the following.

\begin{thm}\label{thm:main1} For any free Borel action of $\mathbb{Z}^2$ on a Polish space $X$, there exists a Borel minimal regulated partition of $X$. In contrast, for any $n\geq 3$, and for any free Borel action of $\mathbb{Z}^n$ on a Polish space $X$, there does not exist a Borel, nondegenerate, minimal regulated partition of $X$.
\end{thm}

Here a regulated partition is {\em nondegenerate} if it does not consist of rectangular regions of lower dimensions all of which are perpendicular to a single standard unit vector.

In view of Theorem~\ref{thm:main1}, it is natural to introduce the Borel and continuous 
{\em regulation numbers} $\gamma_B(n)$, $\gamma_c(n)$ for $\Z^n$. For a given Borel or continuous 
rectangular partition $\sP$ of the free part of the shift action of $\Z^n$ on $X=2^{\Z^n}$, for each $x \in X$
the partiton $\sP$ induces a partition $\sP_x$ of the equivalence class $[x]$ 
which can be viewed as a partition of $\Z^n$. This also gives a
corresponding regulated partition $\sP'_x$ of $\R^n$ as mentioned above. Let $\gamma(\sP)$ be the least
$k$ so that for all $x \in X$, every point of $\R^n$ lies in at most $k$ of the rectangles of $\sP'_x$. 
Then $\gamma_B(n)$ is the minimum value of $\gamma(\sP)$  amongst Borel nondegenerate 
rectangular partitions of $X$. Likewise, $\gamma_c(n)$ is defined in the same way but using 
continuous partitions $\sP$ of the shift space. We clearly have $\gamma_B(n)\leq \gamma_c(n)$
for all $n$. For $n=2$ we actually show that $\gamma_B(2)=\gamma_c(2)=3$, and for $n>2$ Theorem~\ref{thm:main1}
says $\gamma_c(n)\geq \gamma_B(n)>n+1$. We will show an upper-bound for $\gamma_c(n)$ which with 
Theorem~\ref{thm:main1} gives the following results about the continuous and Borel regulation numbers.

\begin{thm} \label{thm:main2}
For $n=2$ we have $\gamma_B(2)=\gamma_c(2)=3$. For $n>2$ we have 
$n+2\leq \gamma_B(n)\leq \gamma_c(n) \leq 3\cdot 2^{n-2}$. 
For $n=3$ we have the improvement $\gamma_B(3)=\gamma_c(3)=5$.
\end{thm}

We do not know if there are any values of $n$ for which $\gamma_B(n)\neq \gamma_c(n)$. We also do not know
the asymptotic growth rates of $\gamma_B(n)$ and $\gamma_c(n)$, in particular whether they 
grow linearly or exponentially.

Our proof of the main theorem uses forcing. Forcing is known to be a useful tool for proving negative results about Borel combinatorics (see \cite{GJKS2022}). In fact we give two forcing proofs of the main theorem. In the first, we prove the theorem in dimension $3$ and then lift up the proof to arbitrary dimension $n\geq 3$. In the second, we give a different forcing construction by working directly in dimension $n\geq 3$. However, both these proofs are based on some observations of minimal regulated partitions of $\mathbb{R}^n$. In particular, we prove the following finitary version of the main theorem for minimal locally regulated partition of a rectangle in $\mathbb{R}^n$.

\begin{thm}\label{thm:main2} For any rectangle $R$ in $\mathbb{R}^2$ and a finite collection $\mathcal{S}$ of pairwise disjoint subrectangles of $R$, there is a finite minimal locally regulated partition $\mathcal{P}$ of $R$ with $\mathcal{S}\subseteq \mathcal{P}$. In contrast, for any $n\geq 3$ and for any $n$-dimensional rectangle $R$ in $\mathbb{R}^n$, there exists a finite collection $\mathcal{S}$ of pairwise disjoint $n$-dimensional subrectangles of $R$ such that there does not exist any finite minimal locally regulated partition $\mathcal{P}$ of $R$ with $\mathcal{S}\subseteq \mathcal{P}$. 
\end{thm}

The rest of the paper is organized as follows. In Section~\ref{sec:prelim} we define the basic notions and prove a general property of all regulated partitions of $\mathbb{R}^n$ which allows us to define minimality for regulated partitions and locally regulated partitions. In Section~\ref{sec:min} we give a canonical description of minimal locally regulated partitions (Theorem~\ref{thm:minimalrep}) and use it to prove the topological regularity property of minimal locally regulated partitions (Theorem~\ref{union}). In Section~\ref{sec:minR} we define and characterize minimal regulated partitions of $\mathbb{R}^n$, and show their existence. In Section~\ref{sec:4} we analyze minimal (locally) regulated partitions and prove the second part of Theorem~\ref{thm:main2}. In Section~\ref{sec:6} we show that for any $n\geq 2$, considering the continuous action of $\mathbb{Z}^n$ on the free part of the Bernoulli shift space $2^{\mathbb{Z}^n}$, there is a clopen regulated partition where the regulation number of all points are bounded by $3\cdot 2^{n-2}$ (Theorem~\ref{thm:gamma}). In Section~\ref{sec:7} we improve this bound to $5$ for dimension $3$. Section~\ref{sec:8} is the main section in which we prove the second part of Theorem~\ref{thm:main1}. Finally, in Section~\ref{sec:9} we make some more observations about minimal regulated partitions of $\mathbb{R}^n$ in general, and sketch an alternative forcing proof of the main result. 

{\it Acknowledgments.} The authors thank Andrew Marks for helpful discussions on the topic of this paper.

\section{Regulated Partitions of $\mathbb{R}^n$\label{sec:prelim}}

Throught the paper we fix $n\geq 1$, let $\vec{0}$ be the origin of $\R^n$, and let $e_1,\dots, e_n$ denote the standard unit vectors in $\R^n$:
$$ e_i=(0,\dots, 0, 1, 0,\dots, 0), $$
where the only $1$ appears in the $i$-th coordinate. For any $x\in\R^n$ and $1\leq i\leq n$, let $\pi_i(x)=\pi_{e_i}(x)=\langle e_i, x\rangle$ be the $i$-th coordinate of $x$. For any subset $A\subseteq \{1,\dots, n\}$, let $V_A=\{e_i\colon i\in A\}$ and let 
$$\pi_A=\pi_{V_A}\colon \R^n\to\R^{A}\cong \R^{|A|}$$ be the projection map given by
$$ \pi_A(x)=\pi_{V_A}(x)=(\pi_i(x))_{i\in A}=(\pi_v(x))_{v\in V_A}. $$

For $0\leq k\leq n$, a {\em $k$-dimensional rectangle} in $\R^n$ is a set of the form
$$ R=I_1\times \cdots \times I_n, $$
where
\begin{enumerate}
\item[(i)] each $I_i$ is of the form $[a_i,b_i]$ for real numbers $a_i\leq b_i$;
\item[(ii)] for exactly $k$ many $i$, $I_i$ is not a singleton.
\end{enumerate}
Thus in the degenerate case, a $0$-dimensional rectangle is just a point.

If $R$ is an $n$-dimensional rectangle in $\R^n$, we denote its interior as $\Int(R)$. Thus for
$$ R=[a_1,b_1]\times \cdots \times[a_n,b_n],$$
where $a_i<b_i$ for all $i$,
$$ \Int(R)=(a_1,b_1)\times\cdots \times(a_n,b_n). $$
We also let $\partial R$ denote the boundary of $R$, i.e., $\partial R=R\setminus \Int(R)$. $\partial R$ can be written in a unique way as a union of $2n$ many distinct $(n-1)$-dimensional rectangles in $\R^n$, which we call {\em faces} of $R$. For each face $F$ of $R$, the {\em normal vector} of $F$ is the unique unit vector $e_i$ perpendicular to all vectors in $F$; the {\em direction} of $F$ is the unique vector of the form $+e_i$ or $-e_i$, whichever is pointing to the interior of $R$ from $F$, where $e_i$ is the normal vector of $F$. If two distinct faces of $R$ have the same normal vector, then they are  parallel and they have the opposite directions; otherwise they are perpendicular. Any set of pairwise perpendicular faces of $R$ has a nonempty intersection. If $\mathcal{F}$ is a nonempty set of pairwise perpendicular faces of $R$ and $|\mathcal{F}|=k$, then $1\leq k\leq n$ and the intersection $\bigcap \mathcal{F}$ is an $(n-k)$-dimensional rectangle in $\R^n$. For each $x\in R$, we define the {\em boundary rank} of $x$ with respect to $R$, denoted $\beta_R(x)$, to be the largest $0\leq k\leq n$ such that there is a set $\mathcal{F}$ of pairwise perpendicular faces of $R$ with $|\mathcal{F}|=k$ and $x\in\bigcap\mathcal{F}$. Thus $\beta_R(x)=0$ iff $x\in\Int(R)$, and if $x\in\partial R$, then $1\leq \beta_R(x)\leq n$. For $0\leq m\leq n-1$, the {\em $m$-boundary} of $R$ is the set of all $x\in R$ with $\beta_R(x)=n-m$. These geometric concepts can be easily described algebraically as follows. If
$$ R=[a_1,b_1]\times\cdots \times [a_n,b_n]$$
where $a_i<b_i$ for all $1\leq i\leq n$, then $R$ is an $n$-dimensional rectangle in $\R^n$. A face of $R$ is of the form
$$ \{x\in R\colon \pi_i(x)=a_i\} \mbox{ or } \{x\in R\colon \pi_i(x)=b_i\} $$
for some $1\leq i\leq n$; the direction of the former face is $+e_i$ and that of the latter face is $-e_i$. For each $x\in R$, define the {\em boundary incidence sequence} of $x$ with respect to $R$ to be
$$ B_R(x)=(\beta_1,\dots, \beta_n) $$
where for each $1\leq i\leq n$,
$$ \beta_i=\left\{\begin{array}{ll} +1, &\mbox{ if $\pi_i(x)=a_i$,} \\ -1, & \mbox{ if $\pi_i(x)=b_i$,}\\
0, &\mbox{ otherwise.}\end{array}\right.
$$
Then $\beta_R(x)$ is the number of nonzero entries of $B_R(x)$.

\begin{defn}\label{def:regulatedpartition} Let $n\geq 1$. A {\em regulated cover} of $\R^n$ is a set $\sP$ of $n$-dimensional rectangles in $\R^n$ such that
\begin{enumerate}
\item there is a discrete subset $D\subseteq \R$ and a constant $d>0$ (called a {\em discreteness constant} for $\sP$) such that
\begin{itemize}
\item for any distinct $r_1,r_2\in D$, $|r_1-r_2|>d$, and
\item each $R\in \sP$ is of the form
$$ [a_1,b_1]\times\cdots\times[a_n,b_n] $$
where $a_i,b_i\in D$ and $a_i<b_i$ for all $1\leq i\leq n$;
\end{itemize}
\item for every $x\in \R^n$ there is $R\in\sP$ such that $x\in R$.
\end{enumerate}
A regulated cover $\sP$ is a {\em regulated partition} if it satisfies additionally
\begin{enumerate}
\item[(3)] if $R_1, R_2\in\sP$ are distinct then $\Int(R_1)\cap \Int(R_2)=\varnothing$.
\end{enumerate}
\end{defn}

A collection which satisfies (2) and (3) is called a {\em rectangulation} of $\mathbb{R}^n$. Clause (3) in the definition of regulated partition is equivalent to
\begin{enumerate}
\item[(3')] if $R_1, R_2\in\sP$ then either $R_1\cap R_2=\varnothing$ or $R_1\cap R_2$ is a $k$-dimensional rectangle in $\R^n$ for some $k<n$.
\end{enumerate}

\begin{defn} Let $n\geq 1$ and let $R$ be an $n$-dimensional rectangle in $\R^n$. A {\em locally regulated cover} of $R$ is a finite set $\sP$ of $n$-dimensional rectangles in $\R^n$ such that $\bigcup \sP=R$. A locally regulated cover $\sP$ of $R$ is a {\em locally regulated partition} of $R$ if for any distinct $R_1,R_2\in\sP$, $\Int(R_1)\cap\Int(R_2)=\varnothing$.
\end{defn}

The following lemma is immediate from the definitions.

\begin{lem}\label{lem:regulatedvslocal} If $\sP$ is a regulated cover (partition) of $\R^n$ and $R$ is an $n$-dimensional rectangle, then
$$ \sQ=\{P\cap R\colon P\in\sP \mbox{ and $P\cap R$ is an $n$-dimensional rectangle}\} $$
is a locally regulated cover (partition) of $R$.
\end{lem}

\begin{defn} Let $R$ be an $n$-dimensional rectangle in $\R^n$ and let $x\in \Int(R)$.
A locally regulated cover (partition) $\sP$ of $R$ is said to be {\em about $x$} if for all $P\in\sP$ and for all faces $F$ of $P$, either $x\in F$ or $F$ is contained in a face of $R$.
\end{defn}

It is easy to see that if $\sP$ is a locally regulated cover of $R$ about $x$ then $x\in \bigcap\sP$. Moreover, if $\sP$ is a locally regulated partition of $R$, then it is about $x$ iff $x\in \bigcap\sP$. 

\begin{lem} \label{lem:localvsabout} If $R$ is an $n$-dimensional rectangle in $\R^n$, $x\in \Int(R)$, and $\sP$ is a locally regulated cover (partition) of $R$, then there is an $n$-dimensional rectangle $R_x$ such that $R_x\subseteq R$ and
$$ \sQ=\{P\cap R_x\colon P\in\sP\} $$
is a locally regulated cover (partion) of $R_x$ about $x$.
\end{lem}

\begin{proof} Note that
$$ U=\bigcup\{P\in\sP\colon x\in P\} $$
is a neighborhood of $x$. If $\{P\in\sP\colon x\in\Int(P)\}$ is nonempty, then 
$$ V=\bigcap\{P\in\sP\colon x\in \Int(P)\} $$
is also a neighborhood of $x$; otherwise we can let $V=\R^n$. 
Then we can let $R_x$ be an $n$-dimensional rectangle such that
$$ x\in\Int(R_x)\subseteq R_x\subseteq U\cap V. $$
\end{proof}

\begin{example} Let $R_0=[-1,1]^n$ and $x_0=\vec{0}$. Let $\mathcal{P}_0$ be the set of all $n$-dimensional rectangles of the form
$$ I_1\times\cdots \times I_n $$
where each $I_i$ is either $[-1,0]$ or $[0,1]$. Then $|\mathcal{P}_0|=2^n$ and $\mathcal{P}_0$ is a locally regulated partition of $R_0$ about $x_0$. On the other hand, if $\mathcal{P}$ is any locally regulated cover of $R_0$ about $x_0$, then any element of $\mathcal{P}$ has the form
$$ J_1\times\cdots\times J_n$$
where each $J_i$ is either $[-1,0]$, $[0,1]$ or $[-1,1]$. More generally, for any $n$-dimensional rectangle $R$ and $x\in\Int(R)$, there is an obvious one-one correspondence between locally regulated covers of $R$ about $x$ and those of $R_0$ about $x_0$ via a translation (from $x$ to $x_0$) and rescaling (from the side lengths $R$ to those of $R_0$).
\end{example}

\begin{defn} Let $R$ be an $n$-dimensional rectangle in $\R^n$, let $\mathcal{P}$ be a locally regulated cover (partition) of $R$, and let $x\in \Int(R)$.
\begin{enumerate}
\item[(i)] A standard unit vector $e_i$ is called a {\em boundary vector} of $x$ with respect to $\sP$ if there is $P\in\mathcal{P}$ such that $x$ is contained in a face of $P$ whose normal vector is $e_i$, or equivalently, $x\in P$ and $\pi_i(B_P(x))\neq 0$.
\item[(ii)] The {\em boundary rank} of $x$ with respect to $\sP$, denoted $\beta_{\sP}(x)$, is defined to be the number of distinct boundary vectors of $x$ with respect to $\sP$.
\item[(iii)] The {\em partition number} of $x$ with respect to $\sP$, denoted $\nu_{\sP}(x)$, is the number of distinct elements of $\sP$ containing $x$.
\end{enumerate}
Similarly define these notions for regulated covers (partitions) $\sP$ of $\R^n$ and $x\in \R^n$.
\end{defn}

For a locally regulated cover (partition) about a point, we define the following algebraic representation for the information of interest.

\begin{defn} Let $R$ be an $n$-dimensional rectangle in $\R^n$, $x\in \Int(R)$, and $\sP$ be a locally regulated cover of $R$ about $x$. Suppose $|\sP|=m$. Let $P_1,\dots, P_m$ be an enumeration of $\sP$. The $m\times n$ matrix $M=(a_{i,j})_{1\leq i\leq m, 1\leq j\leq n}$ where
$$ a_{i,j}=\pi_{e_j}(B_{P_i}(x))$$
is called a {\em principal matrix} of $x$ with respect to $\sP$. For a principal matrix of $x$ with respect to $\sP$, we view the columns as indexed by the vectors $e_1,\dots, e_n$ and the rows as indexed by the elements of $\mathcal{P}$.
\end{defn}

The following observation is immediate.

\begin{lem}\label{lem:boundaryvector} Let $R$ be an $n$-dimensional rectangle in $\R^n$, $x\in \Int(R)$, and $\mathcal{P}$ be a locally regulated cover of $R$ about $x$. Let $M=(a_{i,j})_{m\times n}$ be a principal matrix of $x$ with respect to $\sP$. Then the following hold:
\begin{enumerate}
\item[(i)] For any $1\leq j\leq n$, $e_j$ is a boundary vector of $x$ iff there is $1\leq i\leq m$ such that $a_{i,j}\neq 0$.
\item[(ii)] The boundary rank of $x$ with respect to $\sP$, $\beta_{\sP}(x)$, is the number of columns in $M$ that are not constantly zero.
\end{enumerate}
\end{lem}

\begin{defn} Let $M=(a_{i,j})_{m\times k}$ be a matrix whose entries are elements of $\{0,\pm1\}$. For each $1\leq j\leq k$, let $M_{j,+}$ and $M_{j,-}$ be submatrices of $M$ defined as follows. $M_{j,+}$ is obtained from $M$ by removing the $j$-th column as well as all the $i$-th rows where $a_{i,j}=-1$. Thus $M_{j,+}$ is an $m'\times (k-1)$ matrix with $m'\leq m$. Similarly, $M_{j,-}$ is obtained from $M$ by removing the $j$-th column as well as all the $i$-th rows where $a_{i,j}=+1$.
\end{defn}

\begin{defn}  Let $R$ be an $n$-dimensional rectangle in $\R^n$, $x\in \Int(R)$, and $\mathcal{P}$ be a locally regulated cover of $R$ about $x$.
For each $1\leq j\leq n$, let $\sP_{j,+}=\{P\in\sP\colon \pi_j(B_P(x))\neq -1\}$ and $\sP_{j,-}=\{P\in\sP\colon \pi_j(B_P(x))\neq +1\}$.
\end{defn}

\begin{lem}\label{lem:principal0}  Let $R$ be an $n$-dimensional rectangle in $\R^n$, $x\in \Int(R)$, and $\mathcal{P}$ be a locally regulated cover (partition) of $R$ about $x$. Let $M=(a_{i,j})_{m\times n}$ be a principal matrix of $x$ with respect to $\sP$. Then for any $1\leq j\leq n$, letting $A=\{1,\dots, n\}\setminus \{j\}$, we have the following:
\begin{enumerate}
\item[(i)] Both $\pi_A(\sP_{j,+})$ and $\pi_A(\sP_{j,-})$ are locally regulated covers (partitions) of $\pi_A(R)$ about $\pi_A(x)$;
\item[(ii)] $M_{j,+}$ and $M_{j,-}$ are principal matrices of $\pi_A(x)$ with respect to $\pi_A(\sP_{j,+})$ and $\pi_A(\sP_{j,-})$ respectively.
\end{enumerate}
\end{lem}

\begin{proof} Immediate from the definitions.
\end{proof}

For a principal matrix of a point $x$ with respect to a locally regulated partition about $x$, we have the following easy observations without proof.

\begin{lem}\label{lem:principal1} Let $R$ be an $n$-dimensional rectangle in $\R^n$, $x\in \Int(R)$, and $\mathcal{P}$ be a locally regulated partition of $R$ about $x$. Let $M=(a_{i,j})_{m\times n}$ be a principal matrix of $x$ with respect to $\sP$. Then the following hold:
\begin{enumerate}
\item[(i)] For any $1\leq i\leq m$ and $1\leq j\leq k$, $a_{i,j}\in\{0,\pm 1\}$;
\item[(ii)] For any $1\leq j\leq k$, if there is $1\leq i\leq m$ such that $a_{i,j}\neq 0$, then there is $1\leq i'\leq m$ such that $a_{i',j}=-a_{i,j}$.
\end{enumerate}
\end{lem}

Lemmas~\ref{lem:principal0} and \ref{lem:principal1} motivate the following definition.

\begin{defn} Let $M=(a_{i,j})_{m\times k}$ be a matrix whose entries are elements of $\{0,\pm1\}$. Recursively define $M$ to be {\em separative} as follows. If $k=1$, then $M$ is separative iff $M$ is constantly zero or there are $1\leq i,i'\leq m$ such that $a_{i,1}a_{i',1}=-1$. If $k>1$, then $M$ is separative iff for every $1\leq j\leq k$, both $M_{j,+}$ and $M_{j,-}$ are separative.
\end{defn}

\begin{lem}\label{lem:separative} Let $R$ be an $n$-dimensional rectangle in $\R^n$, $x\in \Int(R)$, and $\mathcal{P}$ be a locally regulated partition of $R$ about $x$. Let $M$ be a principal matrix of $x$ with respect to $\sP$. Then $M$ is separative.
\end{lem}

\begin{proof} By induction on $n$. For $n=1$ the lemma holds by Lemma~\ref{lem:principal1}. The general inductive case follows from Lemma~\ref{lem:principal0}.
\end{proof}

\begin{lem}\label{lem:separativecover} Let $M=(a_{i,j})_{m\times n}$ be a separative matrix whose entries are elements of $\{0,\pm1\}$. For each $1\leq i\leq m$, let
$$ P_i=I_1\times\cdots\times I_n $$
where 
$$ I_j=\left\{\begin{array}{ll} [0,1], & \mbox{ if $a_{i,j}=+1$,} \\ \mbox{$[-1,0]$}, & \mbox{ if $a_{i,j}=-1$,} \\ \mbox{$[-1,1]$}, & \mbox{ if $a_{i,j}=0$.}
\end{array}\right. $$
Then $\sP=\{P_i\colon 1\leq i\leq m\}$ is a locally regulated cover of $R_0=[-1,1]^n$ about $x_0=\vec{0}$, and $M$ is a principal matrix of $x$ with respect to $\sP$.
\end{lem}

\begin{proof} By induction on $n$. If $n=1$ the lemma is immediate from the definition of separativeness. Suppose $n>1$. Then $M_{1,+}$ and $M_{1,-}$ are both separative. Let $A=\{2,\dots, n\}$. Let $\sP_{1,+}=\{P\in\sP\colon a_{i,j}\neq -1\}$ and $\sP_{1,-}=\{P\in\sP\colon a_{i,j}\neq +1\}$. It is straightforward to check that $\pi_A(\sP_{1,+})$ is the collection of $(n-1)$-dimensional rectangles in $\mathbb{R}^A$ defined from $M_{1,+}$ as in the assumption of the lemma. By the inductive hypothesis, $\pi_A(\sP_{1,+})$ is a locally regulated cover of $\pi_A(R_0)$ about $\pi_A(x_0)$, and $M_{1,+}$ is a principal matrix of $\pi_A(x_0)$ with respect to $\pi_A(\sP_{1,+})$. Similarly, $\pi_A(\sP_{1,-})$ is a locally regulated cover of $\pi_A(R_0)$ about $\pi_A(x_0)$, and $M_{1,-}$ is a principal matrix of $\pi_A(x_0)$ with respect to $\pi_A(\sP_{1,-})$ . It follows that $\sP$ is a locally regulated cover of $R_0$ about $x_0$, and $M$ is a principal matrix of $x_0$ with respect to $\sP$.
\end{proof}

\begin{defn} Let $R$ be an $n$-dimensional rectangle in $\R^n$, $x\in \Int(R)$, and $\mathcal{P}$ be a locally regulated cover of $R$ about $x$. We say that $\sP$ is {\em separative} if a principal matrix of $x$ with respect to $\sP$ is separative.
\end{defn}

By Lemma~\ref{lem:separative}, any locally regulated partition about a point $x$ is a separative locally regulated cover about $x$.

\begin{thm}\label{thm:separative} Let $R$ be an $n$-dimensional rectangle in $\R^n$, let $x\in \Int(R)$, and let $\sP$ be a separative locally regulated cover of $R$ about $x$. Then $\nu_{\sP}(x)\geq \beta_{\sP}(x)+1$.
\end{thm}

\begin{proof} By induction on $|\sP|=m$ and $\beta_{\sP}(x)=k$. If $\beta_{\sP}(x)=0$ then the theorem holds since $|\sP|\geq 1$. If $\beta_{\sP}(x)=1$ then the theorem holds by Lemma~\ref{lem:boundaryvector} and the definition of separativeness. Suppose $\beta_{\sP}(x)=k>1$. Without loss of generality assume that $e_1$ is a boundary vector of $x$ with respect to $\sP$. By Lemma~\ref{lem:principal0}, $M_{1,+}$ is a principal matrix of $\pi_A(x)$ with respect to  a locally regulated cover $\sP_{1,+}$ about $\pi_A(x)$, where $A=\{2,\dots, n\}$. Let $m_+$ be the number of rows of $M_{1,+}$ and $k_+$ be the number of nonzero columns of $M_{1,+}$. Then $m_+=|\sP_{1,+}|$ and $k_+=\beta_{\sP_{1,+}}(\pi_A(x))$. By Lemma~\ref{lem:principal1} (ii), $m_+<m$. By the inductive hypothesis, we have $m_+\geq k_++1$. 

Let $m_-=m-m_+$ and $k_-=k-1-k_+$. Let $B$ be the set of all $1<j\leq n$ such that for all $1\leq \iota\leq m$ with $\pi_1(P_{\iota})=0$ or $+1$, $a_{\iota,j}=0$. Consider the $m_-\times k_-$ matrix $N$ which is obtained from $M=(a_{i,j})_{m\times n}$ by retaining all entries $a_{i,j}$ where $a_{i,1}=-1$ and $j\in B$. Figure~\ref{fig:1} illustrates the position of $N$ relative to $M_{1,+}$. In the figure, $M_{1,+}^*$ denotes the matrix $M_{1,+}$ without the all-zero columns.

\begin{figure}[h]
\begin{tikzpicture}[scale=0.1]

\draw (5,0) to (5, 20);
\draw (19,0) to (19,20);
\draw (-2,10) to (30,10);
\node[left] at (5, 23) {$e_1$};
\node[left] at (18,23) {$A-B$};
\node[left] at (27, 23) {$B$};
\node[left] at (5, 13) {$0$};
\node[left] at (5, 17) {$+1$};
\node[left] at (5, 4) {$-1$};
\node[left] at (-2,4) {$m_-$};
\node[left] at (-2,15) {$m_+$};
\node[left] at (17,15) {$M_{1,+}^*$};
\node[left] at (27,4) {$N$};
\node[left] at (27,15) {$0$};
\node[left] at (16,-4) {$k_+$};
\node[left] at (28,-4) {$k_{-}$};

\end{tikzpicture}
\caption{A decomposition of matrix $M$. \label{fig:1}}
\end{figure}

We summarize the information contained in Figure~\ref{fig:1} by writing
$$ M=\left(\begin{array}{ccc} +1/0 & M^*_{1,+} & 0 \\ -1& * & N\end{array}\right). $$

If $N$ is empty then $k_-=0$ and $m_-\geq 1$, and 
$$ m=m_++m_-\geq m_++1\geq k_++2=k+1 $$
as desired. In the rest of the proof, assume $N$ is nonempty, i.e., $k_->0$.

We claim that $N$ is separative. We prove this by induction on $k_-$. If $k_-=1$, then $N$ is separative because $M$ is. Suppose $k_->1$. Let $j\in B$. Then in the same sense of Figure~\ref{fig:1}, we have
$$ M_{j,+}=\left(\begin{array}{ccc} +1/0 & M_{1,+}^* & 0 \\ -1 & * & N_{j,+}\end{array}\right). $$
By the inductive hypothesis, $N_{j,+}$ is separative. Similarly, $N_{j,-}$ is separative. This shows that $N$ is separative.

Let $\sQ=\{P\in\sP\colon \pi_1(B_P(x))=-1\}$. Then $\pi_B(\sQ)$ is a collection of rectangles in $\R^B$ which correspond to $N$ in the sense of Lemma~\ref{lem:separativecover}. By Lemma~\ref{lem:separativecover}, $\pi_B(\sQ)$ is a locally regulated cover of $\pi_B(R)$ about $\pi_B(x)$, and $N$ is a principal matrix of $\pi_B(x)$ with respect to $\pi_B(\sQ)$. By the inductive hypothesis, we have $m_-\geq k_-+1$.

Now 
$$m=m_++m_-\geq (k_++1)+(k_-+1)=k_++k_-+1+1=k+1$$ 
as desired. This finishes the proof of the inductive case. The proof of the theorem is complete.
\end{proof}

\begin{cor}\label{cor:regulated} The following hold:
\begin{enumerate}
\item[(1)] Let $R$ be an $n$-dimensional rectangle in $\R^n$, $x\in \Int(R)$ and let $\sP$ be a locally regulated partition of $R$ about $x$. Then $\nu_{\sP}(x)\geq\beta_{\sP}(x)+1$.
\item[(2)] Let $\sP$ be a regulated partition of $\R^n$ and let $x\in \R^n$. Then $\nu_{\sP}(x)\geq \beta_{\sP}(x)+1$.
\end{enumerate}
\end{cor}

\begin{proof} (1) follows from Theorem~\ref{thm:separative} and Lemma~\ref{lem:separative}, i.e., every locally regulated partition about a point is separative. (2) follows immediately from (1) and Lemmas~\ref{lem:regulatedvslocal} and \ref{lem:localvsabout}.
\end{proof}

Before closing this section we prove a strengthening of Corollary~\ref{cor:regulated} (i) by a direct induction.

\begin{thm} Let $R$ be an $n$-dimensional rectangle in $\R^n$, let $x\in\Int(R)$, and let $\sP$ be a locally regulated partition of $R$ about $x$. For any nonempty set $V$ of boundary vectors of $x$ with respect to $\sP$, let $\sP_V$ be the set of all $P\in \sP$ such that some $v\in V$ is a boundary vector of $x$ with respect to $P$. Define $\sP_\varnothing=\sP$. Then $|\sP_V|\geq |V|+1$.
\end{thm}

\begin{proof} By induction on $|V|$. If $|V|=1$ then the theorem holds by Lemma~\ref{lem:principal1} (ii). Suppose $|V|>1$. Without loss of generality assume $e_1\in V$. Then $(\sP_{1,-})_{e_1}$ is nonempty, and disjoint from $\sP_{1,+}$. Let $A=\{2,\dots, n\}$. Then by Lemma~\ref{lem:principal0} (i), $\sQ=\pi_A(\sP_{1,+})$ and $\sQ'=\pi_A(\sP_{1,-})$ are both locally regulated partitions of $\pi_A(R)$ about $\pi_A(x)$.  Let $V'=V\setminus\{e_1\}$. Then $V'$ is nonempty. Let $U\subseteq V'$ be the set of all boundary vectors $v\in V'$ of $\pi_A(x)$ with respect to $\sQ$. Then by the inductive hypothesis, $|\sQ_U|\geq |U|+1$. It follows  that $|(\sP_{1,+})_U|=|\sQ_U|\geq |U|+1$. Let $W=V'\setminus U$. If $W=\varnothing$, then
$$ |\sP_V|\geq |(\sP_{1,+})_{U}|+|(\sP_{1,-})_{e_1}|\geq |\sQ_{U}|+1\geq |V|+1. $$
Assume $W\neq \varnothing$.  Then every $v\in W$ is a boundary vector of $\pi_A(x)$ with respect to $\sQ'$. Hence by the inductive hypothesis, $|\sQ'_W|\geq |W|+1$. This implies that $|\sP_W|\geq |\sQ'_W|\geq |W|+1$. Since every $v\in W$ is not a boundary vector of $x$ with respect to $\sP_{1,+}$, we have that $\sP_W\subseteq \sP\setminus\sP_{1,+}$. Thus
$$ |\sP_V|\geq |\sP_W|+|(\sP_{1,+})_U|\geq |W|+1+|U|+1=|V|+1, $$
as desired.
\end{proof}

\section{Minimal Locally Regulated Partitions\label{sec:min}}

\begin{defn}  Let $R$ be an $n$-dimensional rectangle in $\R^n$, $x\in \Int(R)$, and $\sP$ be a locally regulated partition of $R$ about $x$. We say that $\sP$ is {\em minimal} if $\nu_{\sP}(x)=\beta_{\sP}(x)+1$.
\end{defn}

\begin{thm}\label{thm:minimal}  Let $R$ be an $n$-dimensional rectangle in $\R^n$, $x\in \Int(R)$, and $\sP$ be a locally regulated partition of $R$ about $x$. Let $1\leq j\leq n$ and $A=\{1,\dots, n\}\setminus\{j\}$. Assume $\sP$ is minimal. Then both $\pi_A(\sP_{j,+})$ and $\pi_A(\sP_{j,-})$ are minimal locally regulated partitions of $\pi_A(R)$ about $\pi_A(x)$.
\end{thm}

\begin{proof} Without loss of generality assume $j=1$. Then $A=\{2, \dots, n\}$. We show that both $\pi_A(\sP_{1,+})$ and $\pi_A(\sP_{1,-})$ are minimal locally regulated partitions of $\pi_A(R)$ about $\pi_A(x)$. Let $M=(a_{i,j})_{m\times n}$ be a principal matrix of $x$ with respect to $\sP$. The lemma is trivial if the first column of $M$ is all-zero, or equivalently, $e_1$ is not a boundary vector of $x$. So we assume the first column of $M$ is nonzero. Let $\{A_0, A_1, A_2, A_3\}$ be a partition of $A$ defined as follows.
$$ \begin{array}{l}
A_0=\{1<j\leq n\colon \mbox{ the $j$-th column of $M$ is all-zero}\}, \\
A_1=\{1<j\leq n\colon j\notin A_0 \mbox{ and the $j$-th column of $M_{1,+}$ is all-zero}\}, \\
A_2=\{1<j\leq n\colon j\notin A_0 \mbox{ and the $j$-th column of $M_{1,-}$ is all-zero}\}, \\
A_3=A\setminus (A_0\cup A_1\cup A_2).
\end{array}
$$
Note that $A_1\cap A_2=\varnothing$. 

Let 
$$\begin{array}{l} C_+=\{1\leq i\leq m\colon a_{i,j}=+1\}, \\
C_-=\{1\leq i\leq m\colon a_{i,j}=-1\}, \\
C_0=\{1\leq i\leq m\colon a_{i,j}=0\}. 
\end{array}
$$
Then $\{C_+,C_-,C_0\}$ is a partition of $\{1,\dots, m\}$. The submatrix of $M$ corresponding to rows indexed by $C_-$ and columns indexed by $A_0\cup A_1$ is the matrix $N$ considered in the proof of Theorem~\ref{thm:separative}, and by that proof, we have 
$$ |C_{-}|\geq |A_1|+1. $$
The submatrix of $M$ corresponding to rows indexed by $C_{+}\cup C_0$ and columns indexed by $A_2\cup A_3$ is the matrix $M_{1,+}^*$ considered in the proof of Theorem~\ref{thm:separative}, and by that proof, we have
$$ |C_+|+|C_0|= |C_+\cup C_0|\geq |A_2\cup A_3|+1=|A_2|+|A_3|+1. $$
By symmetry, we also have 
$$ |C_+|\geq |A_2|+1. $$
It follows that
$$ \nu_{\sP}(x)=|C_+\cup C_0\cup C_-|\geq |C_+\cup C_0|+|C_-|\geq |A_2|+|A_3|+1+|A_1|+1= \beta_{\sP}(x)+1. $$
By minimality, we have that the equality holds, and it follows that $|C_-|=|A_1|+1$ and $|C_+|+|C_0|=|A_2|+|A_3|+1$. By symmetry, we also have $|C_+|=|A_2|+1$. It follows that $|C_0|=|A_3|$. 

Now it is easy to see that $\pi_A(\sP_{1,+})$ is a locally regulated partition of $\pi_A(R)$ about $\pi_A(x)$. Thus 
$$ \nu_{\pi_A(\sP_{1,+})}(\pi_A(x))= |\pi_A(\sP_{1,+})|=|C_+\cup C_0|=|A_2|+|A_3|+1=\beta_{\pi_A(\sP_{1,+})}(x)+1, $$
and hence $\pi_A(\sP_{1,+})$ is minimal. Similarly, $\pi_A(\sP_{1,-})$ is also a minimal locally regulated partition of $\pi_A(R)$ about $\pi_A(x)$.
\end{proof}

\begin{defn}  Let $R$ be an $n$-dimensional rectangle in $\R^n$, $x\in \Int(R)$, and $\sP, \sQ$ be locally regulated partitions of $R$ about $x$. We say that $\sP$ and $\sQ$ are {\em orthogonal} if the set of all boundary vectors of $x$ with respect to $\sP$ is disjoint from the set of all boundary vectors of $x$ with respect to $\sQ$.
\end{defn}

The following is a corollary of the proof of Theorem~\ref{thm:minimal}.

\begin{cor}  Let $R$ be an $n$-dimensional rectangle in $\R^n$, $x\in \Int(R)$, and $\mathcal{P}$ be a minimal locally regulated partition of $R$ about $x$. Let $1\leq j\leq n$ and $A=\{1,\dots, n\}\setminus\{j\}$. Suppose for all $P\in \mathcal{P}$, $\pi_j(B_P(x))\neq 0$. Then $\pi_A(\sP_{j,+})$ and $\pi_A(\sP_{j,-})$ are orthogonal.
\end{cor}

\begin{proof} Without loss of generality assume $j=1$. In the notation of the proof of Theorem~\ref{thm:minimal}, our assumption is $C_0=\varnothing$. Thus $|A_3|=|C_0|=0$. It follows that $M_{1,+}^*$ is contained exactly in columns of $M$ indexed by $A_2$, and $M_{1,-}^*$ (consist of nonzero columns of $M_{1,-}$) is contained exactly in columns of $M$ indexed by $A_1$. Note that $\{e_i\colon i\in A_2\}$ is the set of all boundary vectors of $\pi_A(x)$ with respect to $\pi_A(\sP_{1,+})$, and $\{e_i\colon i\in A_1\}$ is the set of all boundary vectors of $\pi_A(x)$ with respect to $\pi_A(\sP_{1,-})$. Hence they are disjoint.
\end{proof}

In the rest of this section we give a complete description of all minimal locally regulated partitions about a point.

\begin{defn} Let $R$ be an $n$-dimensional rectangle in $\R^n$, $x\in \Int(R)$, let $\sP$ be a locally regulated partition of $R$ about $x$, let $e_j$ be a boundary vector of $x$ with respect to $\sP$, and let $P, Q\in \sP$. We call $P$ and $Q$ an {\em $e_j$-pair} if $\pi_{j'}(B_P(x))=\pi_{j'}(B_Q(x))$ for all $j'\neq j$ and $\pi_j(B_P(x))\pi_j(B_Q(x))=-1$.
\end{defn}

Equivalently,  if $M=(a_{i,j})_{m\times n}$ is a principal matrix of $x$ with respect to $\sP$, then $(P_i, P_{i'})$ is an $e_j$-pair iff the $i$-th and $i'$-th rows of $M$ are identical except $a_{i,j}$ and $a_{i',j}$ take opposite nonzero values. In this case we also say that $(i,i')$ is an $e_j$-pair in $M$.

\begin{lem} Let $R$ be an $n$-dimensional rectangle in $\R^n$, $x\in \Int(R)$, let $\sP$ be a locally regulated partition of $R$ about $x$, let $e_j$ be a boundary vector of $x$ with respect to $\sP$, and let $P, Q\in \sP$ be an $e_j$-pair. Then
$$ \{P\cup Q\}\cup\sP\setminus \{P, Q\} $$
is still a locally regulated partition of $R$ about $x$.
\end{lem}

\begin{proof} Only note that if $P, Q$ form an $e_j$-pair, then $P\cup Q$ is a rectangle.
\end{proof}

\begin{example} The following example of a locally regulated partition $\sP$  in $\mathbb{R}^3$ does not contain an $e_j$-pair for any $1\leq j\leq 3$. First partition $\mathbb{R}^3$ into eight octants, and denote them as octants I through VIII in the usual order. Then 
$$\sP=\{\mbox{I}+\mbox{IV}, \mbox{II}+\mbox{VI}, \mbox{III}, \mbox{V}, \mbox{VII}+\mbox{VIII}\} $$
is such a partition. Its principal matrix is
$$ \left[\begin{array}{ccc} +1 & 0 & +1 \\ -1 & +1 & 0 \\ -1 & -1 & +1 \\ +1 & +1 & -1 \\ 0 & -1 & -1\end{array}\right]. $$
In this matrix every pair of distinct rows differ in at least two entries.
\end{example}

\begin{lem}\label{lem:ejpair} Let $R$ be an $n$-dimensional rectangle in $\R^n$, $x\in \Int(R)$, let $\sP$ be a minimal locally regulated partition of $R$ about $x$. Then there exist a $1\leq j\leq n$ and an $e_j$-pair $P, Q\in\sP$.
\end{lem}

\begin{proof} We prove this by induction on $|\sP|$. Let $P_1,\dots, P_m$ enumerate $\sP$. Let $M=(a_{i,j})_{m\times n}$ be a principal matrix of $x$ with respect to $\sP$. Consider $j=1$ and $A=\{2,\dots, n\}$. Using the notation in the proof of Theorem~\ref{thm:minimal}, $A=A_0\cup A_1\cup A_2\cup A_3$. Consider the rows of $M$ correspondent to $\sP_{1,+}$. These are the rows indexed by $P_i\in\sP$ where $a_{i,1}=0$ or $+1$. In the notation of the proof of Theorem~\ref{thm:minimal}, the indices of these rows are from $C_+\cup C_0$.

In the rest of the proof we use induction on $|C_0|=|\{P\in\sP\colon \pi_1(B_P(x))=0\}|$.

First consider the case $C_0=\varnothing$. Then by Theorem~\ref{thm:minimal}, $\pi_A(\sP_{1,+})$ is a minimal locally regulated partition of $\pi_A(R)$ about $\pi_A(x)$, and by the inductive hypothesis there is an $\eta>1$ and an $e_\eta$-pair $(\pi_A(P_i),\pi_A(P_{i'}))$ in $\pi_A(\sP_{1,+})$. This means that the $i$-th and $i'$-th rows of $M_{1,+}$ are identical except that $a_{i,\eta}a_{i',\eta}=-1$. It follows that the $i$-th and $i'$-th rows of $M$ are identical except that $a_{i,\eta}a_{i',\eta}=-1$. Thus $(P_i,P_{i'})$ is an $e_\eta$-pair in $\sP$.

Next consider the case $C_0\neq\varnothing$. Again by Theorem~\ref{thm:minimal}, $\pi_A(\sP_{1,+})$ is a minimal locally regulated partition of $\pi_A(R)$ about $\pi_A(x)$. By the indutive hypothesis, $\pi_A(\sP_{1,+})$ contains an $e_\eta$-pair for some $\eta\geq 2$. By Lemma~\ref{lem:principal0}, $M_{1,+}$ is a principal matrix of $\pi_A(x)$ with respect to $\pi_A(\sP_{1,+})$, and hence there are distinct $i,i'\in C_+\cup C_0$ such that the $i$-th and $i'$-th rows of $M_{1,+}$ are identical except that $a_{i,\eta}a_{i',\eta}=-1$. If $i,i'\in C_+$ or $i,i'\in C_0$, then the $i$-th and $i'$-th rows of $M$ are identical except that $a_{i,\eta}a_{i',\eta}=-1$, and hence $(P_i, P_{i'})$ is an $e_\eta$-pair in $\sP$. For the rest of the proof, assume $i\in C_+$ and $i'\in C_0$. To simplify notation, assume without loss of generality that $\eta=2$, and assume the $i$-th row of $M$ is $(+1,+1, \vec{k})$ and the $i'$-th row of $M$ is $(0,-1,\vec{k})$, where $\vec{k}$ is some vector of length $n-2$ with values in $\{0,\pm 1\}$. 

We are now going to define a new locally regulated partition $\tilde{\sP}$ about $x$. $\tilde{\sP}$ is defined in two steps. In the first step, we take $P_{i'}\in\sP$ and divide it by the hyperplane with normal vector $e_1$ into two parts $\tilde{P}_+$ and $\tilde{P}_-$, with $\pi_1(B_{\tilde{P}_+}(x))=+1$ and $\pi_1(B_{\tilde{P}_-}(x))=-1$. In the second step, we combine $P_i$ and $\tilde{P}_+$ into a new rectangle $\tilde{Q}$. Then let $\tilde{\sP}=\{\tilde{Q}, \tilde{P}_-\}\cup\sP\setminus \{P_i, P_{i'}\}$. In the enumeration of $\tilde{\sP}$, give $\tilde{Q}$ index $i$ and $\tilde{P}_-$ index $i'$. Then in the principal matrix $\tilde{M}$ of $\tilde{\sP}$, all rows except the $i$-th and $i'$-th rows are identical to the corresponding rows in $M$, and the new $i$-th row is $(+1,0,\vec{k})$ and the $i'$-th row is $(-1,-1,\vec{k})$. Since $\tilde{\sP}$ has the same boundary vectors as $\sP$ and $|\sP|=|\tilde{\sP}|$, $\tilde{\sP}$ is minimal. 

\begin{center}
\begin{tikzcd}[row sep=tiny]
(+1,+1, \vec{k}) \arrow[r] & (+1,+1,\vec{k})\arrow[r] & (+1,0,\vec{k}) \\
 & (+1,-1,\vec{k})\arrow[ur] &  \\
(0,-1,\vec{k})\arrow[r]\arrow[ur] & (-1,-1,\vec{k})\arrow[r] & (-1,-1,\vec{k}) \\
M & & \tilde{M}
\end{tikzcd}
\end{center}

Now note that $|\{P\in\tilde{\sP}\colon \pi_1(B_P(x))=0\}|<|\{P\in \sP\colon \pi_1(B_P(x))=0\}|$. Thus by the inductive hypothesis, there is an $e_\eta$-pair $(\iota, \iota')$ in $\tilde{M}$. If $\{\iota,\iota'\}\cap \{i,i'\}=\varnothing$, then $(\iota,\iota')$ is also an $e_\eta$-pair in $M$, and we are done. 

We claim that in $\tilde{M}$ there is no $e_\eta$-pair $(i',\iota)$ for $\iota\in C_-$ and $2<\eta\leq n$. Assume the claim fails, that $(i',\iota)$ is an $e_\eta$-pair for $\iota\in C_-$ and $2<\eta\leq n$. Then we define a new locally regulated partition $\hat{\sP}$ from $\tilde{\sP}$ by combining $\tilde{P}_{i'}$ and $\tilde{P}_{\iota}$ into $\hat{P}_{i'}$. We have $\hat{\sP}$ is a locally regulated partition of $x$ and $|\hat{\sP}|=|\tilde{\sP}|-1$. Note that $\hat{\sP}$ has the same boundary vectors as $\tilde{\sP}$ (as well as $\sP$). This is a contradiction. 

Similarly, we also have that in $\tilde{M}$ there is no $e_\eta$-pair $(i, \iota)$ for $\iota\in C_+$ and $2<\eta\leq n$. 

Suppose in $\tilde{M}$ that $(i',\iota)$ is an $e_2$-pair for some $\iota\in C_-$. Then note that in $M$ we have that $(i,\iota)$ is indeed an $e_1$-pair, and we are done. Suppose in $\tilde{M}$ that $(i',\iota)$ is an $e_1$-pair for some $\iota\in C_+$. Then the $\iota$-th row of $\tilde{M}$ is $(+1, -1, \vec{k})$. We then have $\tilde{P}_{\iota}\subseteq \tilde{P}_{i}$, contradicting that $\tilde{P}$ is a locally regulated partition.

Finally, suppose $(i,\iota)$ is an $e_1$-pair in $\tilde{M}$ for some $\iota\in C_-$. Then the $\iota$-th row of $\tilde{M}$ is $(-1,0,\vec{k})$. We then have $\tilde{P}_{i'}\subseteq \tilde{P}_{\iota}$, contradicting that $\tilde{P}$ is a locally regulated partition. 

Thus we have gone through all cases, and in each possible case we have found an $e_\eta$-pair in $M$.
\end{proof}

In the following we describe a canonical way to construct minimal locally regulated partitions about a point. We first define a notion of a labelled simple tree as follows. Let $k\geq 0$ be an integer. A {\em simple tree} of $k+1$ levels is a rooted binary tree (i.e. every node has at most two children) such that the root is on level $0$, all terminal nodes are on level $k$, and at each level $\ell<k$, at most one node on level $\ell$ has two children. A simple tree is {\em full} if at each level above the terminal level, exactly one node on that level has two children. Note that in a full simple tree $T$ of more than one level the root always has two children.

We assign labels $+1, -1$ and $0$ to the nodes of a simple tree as follows. If a node does not have a sibling, then it is labelled $0$; otherwise, we label any pair of sibling nodes by $+1$ for one and $-1$ for the other. If $T$ is a labelled simple tree and $t\in T$, then define a labelled simple tree $T_t$ to be the substree of $T$ with $t$ as the root and containing all descendants of $t$ in $T$ as nodes; the label of $t$ in $T_t$ is (re)assigned as $0$ (if necessary), and the labels of all other nodes are the same as their labels in $T$. Given a simple tree $T$ of $k+1$ levels, let $T'$ be the full simple tree obtained from $T$ by removing all levels $\ell>0$ of $T$ on which every node has no siblings; we call $T'$ the {\em derived simple tree} of $T$. Let $J(T)$ denote the set of all indices of removed levels $0<\ell\leq k$ in the construction of $T'$ from $T$. If $T$ is a labelled simple tree, then in the construction of $T'$ from $T$ only nodes with label $0$ are removed, and the rest of the nodes retain their labels in $T$. If $T$ is a labelled full simple tree, then let $T_{-}$ be $(T_t)'$ where $t$ is the child of the root of $T$ with label $-1$, and let $T_{+}$ be $(T_t)'$ where $t$ is the child of the root of $T$ with label $+1$.

Without loss of generality, let $x=x_0=\vec{0}\in\R^n$ and let $R=R_0=[-1,1]^n$. Let $i_1,\dots, i_k\in\{1,\dots, n\}$ be distinct indices, and let $T$ be a labelled full simple tree of $k+1$ levels. We define a partition $\sP(i_1,\dots, i_k;T)$ by induction on $k$ as follows. If $k=0$, then $T$ is a trivial tree with only the root, and we let $\sP(\varnothing;T)=\{R\}$. Suppose $k>0$. Let $S=\{1,\dots, n\}\setminus\{i_1\}$. Let $t_+$ be the child of the root of $T$ with label $+1$, and let $t_-$ be the child of the root of $T$ with label $-1$. Let $I_{+}\subseteq \{2,\dots, k\}$ be the set $\{j+1\colon j\in J(T_{t_{-}})\}$ and $I_{-}\subseteq\{2,\dots, k\}$ be the set $\{j+1\colon j\in J(T_{t_{+}})\}$. Then by the fullness of $T$, we have $I_{-}\cap I_{+}=\varnothing$ and $I_{-}\cup I_{+}=\{2,\dots, k\}$. By the inductive hypothesis, we have already defined $\sP((i_j\colon j\in I_{-});T_{-})=\sP_{-}$ and $\mathcal{P}((i_j\colon j\in I_{+}); T_{+})=\sP_{+}$, which are locally regulated partitions of $\pi_S(R)$. Let
$$ R_{-}=\{y\in R\colon \pi_{i_1}(y)\in [-1,0]\} $$
and
$$ R_{+}=\{y\in R\colon \pi_{i_1}(y)\in [0,1]\}. $$
Define
$$ \sP(i_1,\dots, i_k; T)=\{ \pi_S^{-1}(P)\cap R_{-}\colon P\in \sP_{-}\}\cup\{\pi_S^{-1}(P)\cap R_{+}\colon P\in \sP_{+}\}. $$
This finishes the inductive definition of $\sP(i_1,\dots, i_k; T)$.

The construction of $\sP(i_1,\dots, i_k;T)$ can be described geometrically as follows. We successively divide $R$ using faces with normal vectors $e_{i_1},\dots, e_{i_k}$. In the first step, divide $R$ into two parts $R_{-}$ and $R_{+}$ with a hyperplane containing $x$ and with normal vector $e_{i_1}$ so that $\pi_{i_1}(B_{R_{-}}(x))=-1$ and $\pi_{i_1}(B_{R_{+}}(x))=+1$. Let $t_{-}\in T$ be the child of the root of $T$ with label $-1$ and $t_{+}\in T$ be the child of the root of $T$ with label $+1$. Then exactly one of $t_{-}$ and $t_{+}$ has two children. For definiteness, assume $t_{-}$ has two children. Then in the second step, further divide $R_{-}$ with a hyperplane containing $x$ and with normal vector $e_{i_2}$. Repeat this until all levels of $T$ have been considered. The result is the locally regulated partition of $R$ about $x$.

Note that $\sP(i_1,\dots, i_k;T)$ is minimal because at each step $0<\ell\leq k$ of the above construction, we divide exactly one of the intermediate $n$-dimensional rectangles into two, and thus have exactly $\ell+1$ many rectangles.

In the following we show that all minimal locally regulated partitions about a point arise in this way.

\begin{thm}\label{thm:minimalrep} Let $R$ be an $n$-dimensional rectangle in $\R^n$, $x\in \Int(R)$, and let $\sP$ be a minimal locally regulated partition of $R$ about $x$. Let $k=\beta_{\sP}(x)$. Then there exist distinct $i_1,\dots, i_k\in\{1,\dots, n\}$ and a labelled full simple tree $T$ of $k+1$ levels such that $\sP=\sP(i_1,\dots, i_k;T)$.
\end{thm}

\begin{proof} By induction on $k$. If $k=0$ then we must have $\sP=\{R\}$ and the theorem holds. Suppose $k>0$. By minimality, we have $|\sP|=k+1$. Let $J=\{j_1,\dots, j_k\}$ be the set of all $j\in\{1,\dots, n\}$ such that $e_j$ is a boundary vector of $x$ with respect to $\sP$. By Lemma~\ref{lem:ejpair} there exist $j\in J$ and an $e_j$-pair $P, Q\in \sP$. Suppose $\pi_j(B_P(x))=+1$ and $\pi_j(B_Q(x))=-1$. Without loss of generality, assume $j=j_k$. Consider
$$ \sQ=\{P\cup Q\}\cup \sP\setminus\{P, Q\}. $$
Then $\sQ$ is a locally regulated partition of $R$ about $x$. Note that each $e_j$ for $j\in J'=\{j_1,\dots, j_{k-1}\}$ is still a boundary vector of $x$ with respect to $\sQ$, thus $\beta_{\sQ}(x)\geq k-1$. Since $\nu_{\sQ}(x)=|\sQ|=k$, it follows from Theorem~\ref{thm:separative} that $\beta_{\sQ}=k-1$. This means that $e_{j_k}$ is no longer a boundary vector of $x$ with respect to $\sQ$. In addition, $\sQ$ is minimal.

By the inductive hypothesis, there is a labelled full simple tree $T'$ of $k$ levels such that $\sQ=\sP(i_1,\dots, i_{k-1}; T')$. One of the terminal nodes of $T'$ corresponds to $P\cup Q\in \sQ$. Denote this node as $t$. Now we let $T$ be an extension of $T'$ by extending all terminal nodes of $T'$ to one or two nodes on the level $k+1$. If terminal node of $T'$ is not $t$, then it has one extension. The node $t$ has two extensions corresponding to $P$ and $Q$ respectively. In $T$, the label of $t$ is $0$, the label of the node corresponding to $P$ is $+1$, and the label of the node corresponding to $Q$ is $-1$. It is easily seen that $\sP=\sP(i_1,\dots, i_k; T)$.
\end{proof}

We note the following strong topological property of a minimal locally regulated partition about a point. This property was the original motivation of our study in this paper.

\begin{thm}  \label{union}
Let $R$ be an $n$-dimensional rectangle in $\R^n$, $x\in \Int(R)$, and $\sP$ be a minimal locally regulated partition of $R$ about $x$. Then for any nonempty subset $\sA\subseteq \sP$, $\bigcup\sA$ is homeomorphic to an $n$-dimensional rectangle in $\R^n$.
\end{thm}

\begin{proof} Without loss of generality, assume $x=x_0=\vec{0}$ and $R=R_0=[-1,1]^n$. By Theorem~\ref{thm:minimalrep} there exist distinct $i_1,\dots, i_k\in\{1,\dots, n\}$ and a labelled full simple tree $T$ of $k+1$ levels such that $\sP=\sP(i_1,\dots, i_k;T)$. Let $t$ be the unique splitting node on level $k$, $j=i_k$, and let $P, Q\in \sP$ be the $e_j$-pair corresponding to the nodes extending $t$. We note that $t$ corresponds to $P\cup Q$.

We prove the theorem by induction on $k$. If $k=0$ or $1$, the theorem is obvious. Suppose $k>1$. Let $\sA\subseteq \sP$ be an arbitrary subset. If $\sA$ contains both $P$ and $Q$ or none of $P$ and $Q$, the theorem holds by the inductive hypothesis. Without loss of generality, assume $P\in \sA$, $Q\not\in \sA$, and $\pi_j(B_P(x))=+1$. Let $\sB=\sA\setminus\{P\}$. If $\sB=\varnothing$, then $\bigcup \sA=P$ and the theorem holds trivially. We assume $\sB\neq\varnothing$.

Let $S=\{1,\dots, n\}\setminus\{j\}$. Then $\pi_S(\sP_{j,+})$ is a locally regulated partition of $\pi_S(R)$ about $\pi_S(x)$. Note that since $e_j$ is only a boundary vector of $x$ with respect to $P$ and $Q$, we have that 
$$\pi_S(\sA)=\pi_S(\sB)\cup \pi_S(P) $$
and
$$ \left(\bigcup\pi_S(\sB)\right)\cap \pi_S(P)=\varnothing.$$

\begin{figure}[h]
\begin{tikzpicture}[scale=0.2]

\draw (0,0) to (40,0) to (40,5) to (0,5) to (0,0);
\draw (0,0) to (0,-5) to (30,-5) to (30,5);
\node[left] at (-1,5) {$+1$};
\node[left] at (-1,0) {$0$};
\node[left] at (-1,-5) {$-1$};
\node[left] at (44,5) {$+1$};
\node[left] at (44,0) {$0$};
\node[left] at (19,-2) {$\bigcup\pi_S(\sB)$};
\node[left] at (38.2,-2) {$\pi_S(P)$};
\node[left] at (36.5,2) {$P$};

\end{tikzpicture}
\caption{A decomposition of $\bigcup\sA$. \label{fig:2}}
\end{figure}

By the inductive hypothesis, the sets $\pi_S(P)$, $\bigcup\pi_S(\sB)$ and $\bigcup\pi_S(\sA)=\bigcup\pi_S(\sB)\cup\pi_S(P)$ are  all homeomorphic to $(n-1)$-dimensional rectangles in $\R^S$. Note that
$$\begin{array}{rcl}
 \bigcup \sA&=& \bigcup\sB\cup P
=\Big(\bigcup\pi_S(\sB)\times [-1,1]\Big)\cup \Big(\pi_S(P)\times [0,1]\Big) \\ \\
&=&\Big(\,\bigcup\pi_S(\sB)\times[-1,0]\,\Big)\cup \Big(\,(\bigcup\pi_S(\sB)\cup  \pi_S(P))\times [0,1]\,\Big). 
\end{array}$$
Thus $\bigcup\sA$ is homeomorphic to the union of two $n$-dimensional rectangles with their intersection being an $(n-1)$-dimensional rectangle. It follows that $\bigcup\sA$ is homeomorphic to an $n$-dimensional rectangle in $\R^n$. Figure~\ref{fig:2} illustrates the sets in consideration.
\end{proof}

\section{Minimal Regulated Partitions of $\R^n$\label{sec:minR}}

\begin{defn} \label{gammareg}
Let $n\geq 1$ and $\gamma\geq 1$. Let $\sP$ be a regulated partition of $\R^n$.
\begin{enumerate}
\item[(1)]  We call $\sP$ a {\em $\gamma$-regulated partition} if each $x\in \R^n$ is contained in at most $\gamma$ many elements of $\mathcal{P}$.
\item[(2)] We call $\sP$ {\em minimal} if $\sP$ is an $(n+1)$-regulated partition.
\end{enumerate}
\end{defn}

The following lemmas justify the terminology of minimality.

\begin{lem}\label{lem:minimalregulated} If $\sP$ is a $\gamma$-regulated partition of $\R^n$, then $n+1\leq \gamma\leq 2^n$.
\end{lem}

\begin{proof} It is clear that $\gamma\leq 2^n$. To see $\gamma\geq n+1$, we show that there is $x\in\R^n$ which is contained in at least $n+1$ many elements of $\sP$. We prove this by induction on $n$. When $n=1$ this is obvious. Assume $n>1$. Let $D\subseteq\R$ be a discrete subset of $\R$ satisfying (1) of Definition~\ref{def:regulatedpartition}. Without loss of generality assume that $0\in D$ and that there is $R\in\sP$ such that $\pi_n(R)=[0, a]$ for some $a>0$. Let $d>0$ be a discreteness constant for $\sP$. Let $S=\{1,\dots, n-1\}$. Define
$$ \sQ=\{\pi_S(R)\colon R\in\sP, [0,d]\subseteq \pi_n(R)\}. $$
Then $\sQ$ is a regulated partition of $\R^{n-1}$. By the inductive hypothesis, there is $y_0\in\R^{n-1}$ such that $y_0$ is contained in at least $n$ many elements of $\sQ$. Let $S_1,\dots, S_n$ enumerate $n$ distinct elements of $\sQ$ containing $y_0$. Let $R_1,\dots, R_n\in\sP$ be such that $S_i=\pi_S(R_i)$ for $1\leq i\leq n$. Then $R_1,\dots, R_n$ are distinct, and for each $1\leq i\leq n$, $R_i=S_i\times [a_i, b_i]$ for some $a_i\leq 0<d \leq b_i$. Let $b=\min\{b_i\colon 1\le i\leq n\}$. Without loss of generality, assume $b=b_1$. Let $x_0=(y_0,b)\in \R^n$. We claim that $x_0$ is contained in at least $n+1$ many elements of $\mathcal{P}$. First note that $x_0$ is contained in each of $R_1,\dots, R_n$. Also, note that $x_0\in S_1\times [b,+\infty)$ but
$$\Int(S_1\times [b,+\infty))\cap \bigcup_{i=1}^n R_n=\varnothing. $$
Hence there is some $R_{n+1}\in\sP$, distinct from $R_1,\dots, R_n$, with $\Int(R_{n+1})\cap \Int(S_1\times [b,+\infty))\neq\varnothing$, such that $x_0\in R_{n+1}$. We thus have at least $n+1$ many elements of $\mathcal{P}$ containing $x_0$.
\end{proof}

\begin{lem}\label{lem:localminimal} 
Let $\sP$ be a minimal regulated partition of $\R^n$. Then for each $x\in\R^n$, $\nu_{\sP}(x)=\beta_{\sP}(x)+1$.
\end{lem}

\begin{proof} By the reverse induction on $\beta_{\sP}(x)\leq n$. If $\beta_{\sP}(x)=n$ then the lemma follows from Corollary~\ref{cor:regulated}. Suppose $\beta_{\sP}(x)=k<n$. Toward a contradiction, assume $\nu_{\sP}(x)>k+1$. Assume without loss of generality that $e_n$ is not a boundary vector of $x$ with respect to $\sP$. Without loss of generality assume $\pi_n(x)=0$. Let $D\subseteq \R$ be a discrete subset satisfying Definition~\ref{def:regulatedpartition} (1). Let $a>0$ be the least element of $D$ greater than $0$. Let $S=\{1,\dots, n-1\}$. Similarly to the proof of Lemma~\ref{lem:minimalregulated}, let
$$\sQ=\{\pi_S(R)\colon R\in\sP, [0,a]\subseteq \pi_n(R)\}. $$
Then $\sQ$ is a regulated partition of $\R^{n-1}$ and $\pi_S(x)\in\R^{n-1}$. Since $e_n$ is not a boundary vector of $x$ with respect to $\sP$, it follows that $\nu_{\sQ}(\pi_S(x))=\nu_{\sP}(x)>k+1$. Let $S_1,\dots, S_{k+2}$ enumerate $k+2$ many distinct elements of $\sQ$ containing $\pi_S(x)$. Let $R_1,\dots, R_{k+2}\in\mathcal{P}$ be such that $S_i=\pi_S(R_i)$ for $1\leq i\leq k+2$. Then $R_1,\dots, R_{k+2}$ are distinct, and for each $1\leq i\leq k+2$, $R_i=S_i\times [a_i, b_i]$ for some $a_i\leq 0<a \leq b_i$. Let $b=\min\{b_i\colon 1\le i\leq n\}$. Without loss of generality, assume $b=b_1$. Let $y=(\pi_S(x),b)\in \R^n$. Then by an argument similar to the proof of Lemma~\ref{lem:minimalregulated}, $\nu_{\sP}(y)\geq k+3$. However, $\beta_{\sP}(y)=\beta_{\sP}(x)+1=k+1$. This contradicts the inductive hypothesis.
\end{proof}

We have thus obtained the following nontrivial characterization of the minimality for regulated partitions.

\begin{thm}\label{thm:regpart} Let $\sP$ be a regulated partition of $\R^n$. Then $\sP$ is minimal iff $\nu_{\sP}(x)=\beta_{\sP}(x)+1$ for all $x\in \R^n$.
\end{thm}

\begin{defn} Let $\mathcal{P},\mathcal{Q}$ be regulated partitions of $\R^n$. We say that $\mathcal{P}$ and $\mathcal{Q}$ are {\em orthogonal} if for any $x\in\R^n$, the set of all boundary vectors of $x$ with respect to $\mathcal{P}$ is disjoint from the set of all boundary vectors of $x$ with respect to $\mathcal{Q}$.
\end{defn}

Next we show the existence of minimal regulated partitions of $\R^n$.

\begin{lem}\label{lem:orthogonal} There exist orthogonal minimal regulated partitions of $\R^n$. In particular, there exists a minimal regulated partition of $\R^n$.
\end{lem}

\begin{proof} We prove by induction on $n\geq 1$ that there is an $(n+1)$-regulated partition $\mathcal{P}$ of $\R^n$ such that any $R\in\mathcal{P}$ is of the form
$$ R=[a_1, b_1]\times\cdots\times[a_n,b_n] $$
where $a_i<b_i$ are all even integers. (For simplicity we call such a regulated partition $\mathcal{P}$ {\em having even coordinates}.) Granting this, we can let $\vec{1}=(1,\dots, 1)\in\R^n$ and
$$ \mathcal{Q}=\{R+\vec{1}\colon R\in\mathcal{P}\}. $$
Then $\mathcal{Q}$ is still an $(n+1)$-regulated partition of $\R^n$, and $\mathcal{Q}$ and $\mathcal{P}$ are orthogonal.

For $n=1$ the claim is obviously witnessed by $\mathcal{P}=\{[2k,2k+2]\colon k\in\mathbb{Z}\}$. Assume $n>1$. By the inductive hypothesis, let $\mathcal{P}_0$ be an $n$-regulated partition of $\R^{n-1}$ having even coordinates. As above, let $\mathcal{Q}_0=\{R+\vec{1}\colon R\in\mathcal{P}_0\}$ where $\vec{1}=(1,\dots,1)\in\R^{n-1}$. Then $\mathcal{P}_0$ and $\mathcal{Q}_0$ are orthogonal $n$-regulated partitions of $\R^{n-1}$.

Now for an arbitrary $k$-dimensional rectangle $S$ in $\R^n$, $0\leq k\leq n$, define a rectangle $2S$ as follows: if $S$ is of the form $[a_1,b_1]\times\cdots\times[a_k,b_k]$, then $2S=[2a_1,2b_1]\times\cdots\times[2a_k,2b_k]$. Let $\mathcal{P}$ be the regulated partition of $\R^n$ whose elements are of the form
$$2R\times [4m, 4m+2] \mbox{ where $R\in \mathcal{P}_0$ and $m\in\mathbb{Z}$}$$
or
$$ 2R\times [4m-2, 4m] \mbox{ where $R\in\mathcal{Q}_0$ and $m\in\mathbb{Z}$}.$$
Then $\mathcal{P}$ clearly has even coordinates. We verify that $\mathcal{P}$ is an $(n+1)$-regulated partition of $\R^n$. Let $x\in\R^n$ be arbitrary. We show that $x$ is contained in at most $n+1$ many elements of $\mathcal{P}$. If $\pi_n(x)$ is not an even integer, then by the $n$-regulatedness of $\mathcal{P}_0$ and $\mathcal{Q}_0$, $x$ is contained in at most $n$ many elements of $\mathcal{P}$. Suppose $\pi_n(x)$ is an even integer. Let $S=\{1,\dots, n-1\}$. Suppose $\beta_{\sP_0}(\pi_S(x))=i$ and $\beta_{\sQ_0}(\pi_S(x))=j$. Then by the orthogonality of $\sP_0$ and $\sQ_0$, we have $i+j\leq n-1$. By Lemma~\ref{lem:localminimal},
$$ \nu_{\sP}(x)=\nu_{\sP_0}(\pi_S(x))+\nu_{\sQ_0}(\pi_S(x))=i+1+j+1\leq n+1. $$
Thus $\sP$ is $(n+1)$-regulated.
\end{proof}

\section{Locally Regulated Partitions in $\R^n$\label{sec:4}}

\begin{defn} Let $R$ be an $n$-dimensional rectangle in $\R^n$ and let $\sP$ be a locally regulated partition of $R$. We say that $\sP$ is {\em minimal} if for every $x\in \Int(R)$, $\nu_{\sP}(x)=\beta_{\sP}(x)+1$.
\end{defn}

In this section we consider the following extension problem about locally regulated partitions. Suppose $R$ is an $n$-dimensional rectangle in $\R^n$ and $\sP_0$ is a collection of pairwise disjoint $n$-dimensional subrectangles of $R$. Does there always exist a minimal locally regulated partition $\sP$ of $R$ with $\sP_0\subseteq \sP$? We prove in this section that the answer is positive for $n=2$ and negative for $n\geq 3$.

\begin{lem} Let $R$ be a $2$-dimensional rectangle in $\R^2$ and let $\sP_0$ be a finite collection of pairwise disjoint $2$-dimensional subrectangles of $R$. Then there exists a minimal locally regulated partition $\sP$ of $R$ with $\sP_0\subseteq \sP$.
\end{lem}

\begin{proof} To obtain $\sP$, extend all vertical sides of $P\in\sP_0$ ``maximally," i.e., until they reach a horizontal side of either $R$ or another $Q\in\sP_0$. This results in a locally regulated partition $\sP$ of $R$ with $\sP_0\subseteq \sP$. To see that $\sP$ is minimal, assume toward a contradiction that $x\in \Int(R)$ and $\nu_{\sP}(x)=4$. Thus $x$ is the intersection of a vertical line segment $V$ with a horizontal line segment $H$ which are on the boundary of some elements of $\sP$. Since our construction of $\sP$ from $\sP_0$ does not involve introducing horizontal lines, $H$ must be on the boundary of an element $P\in \sP_0$. Without loss of generality, assume $H$ is a side of $P$. We consider two cases. 
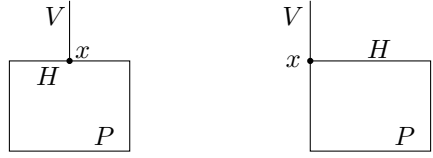
\begin{figure}[h]
\begin{tikzpicture}[scale=0.4]

\draw (2,5) to (2,3) to (4,3) to (4,0) to (0,0) to (0,3) to (2,3);
\node[left] at (3.8,0.5) {$P$};
\node[left] at (2,2.5) {$H$};
\node[left] at (3,3.3) {$x$};
\node[left] at (2.2,4.5) {$V$};
\filldraw[fill=black, very thick] (2,3) circle (0.05);

\draw (10,5) to (10,0) to (14,0) to (14,3) to (10,3);
\node[left] at (13.8,0.5) {$P$};
\node[left] at (13,3.4) {$H$};
\node[left] at (10,3) {$x$};
\node[left] at (10.1,4.5) {$V$};
\filldraw[fill=black, very thick] (10,3) circle (0.05);

\end{tikzpicture}
\end{figure}

Case 1: $x$ is not an endpoint of $H$. Then $P$ lies on one side of $H$ and $V$ lies on the other side. We have $\nu_{\sP}(x)=3$, contradicting our assumption. Case 2: $x$ is an endpoint of $H$. Then $x$ is a corner point of $P$. In this case, $H$ lies on one side of $V$ and does not extend to the other side of $V$. Again we have $\nu_{\sP}(x)\leq 3$, contradiction. 
\end{proof}

Next we consider the extension problem for $n=3$. 


\begin{defn}\label{def:maximalsurface} Let $R=[a_1,b_1]\times [a_2,b_2]\times[a_3,b_3]=S\times [a_3,b_3]$ be a $3$-dimensional rectangle in $\R^3$ and let $\sP$ be a locally regulated partition of $R$. 
\begin{enumerate}
\item[(1)] If $P\in \sP$ and $\pi_3(P)=[c, d]$, then we say that $P$ has {\em lower level} $c$ and {\em upper level} $d$.
\item[(2)]  A {\em level} of $\sP$ is any $d\in (a_3,b_3)$ such that $d$ is the upper level of some $P\in \sP$.
\item[(3)] If $d$ is a level of $\sP$, then define 
$$\sP_{d,-}=\{P\in \sP\colon \mbox{ $P$ has upper level $d$}\}. $$
\item[(4)] If $d$ is a level of $\sP$, then a {\em maximal surface} on level $d$ is a maximal subset $M$ of $S$ such that 
$$ M\times\{d\}\subseteq (S\times\{d\})\cap \bigcup \sP_{d,-} $$
and $\Int(M)$ is connected.
\end{enumerate}
\end{defn}

We state the following easy observations without proof.

\begin{lem} Let $R=S\times [a_3,b_3]$ be a $3$-dimensional rectangle in $\R^3$ and let $\sP$ be a locally regulated partition of $R$. Let $d\in(a_3,b_3)$. Then the following hold:
\begin{enumerate}
\item[(i)] $d$ is a level of $\sP$ iff $d$ is the lower level of some $P\in \sP$.
\item[(ii)] Suppose $d$ is a level of $\sP$. Define
$$ \sP_{d,+}=\{P\in\sP\colon \mbox{ $P$ has lower level $d$}\}. $$
Then $\bigcup \pi_{1,2}(\sP_{d,+})=\bigcup\pi_{1,2}(\sP_{d,-})$.
\end{enumerate}
\end{lem}

The following lemmas are key observations about minimal locally regulated partitions in $\R^3$.

\begin{lem}\label{lem:maxsurface0} 
Let $R$ be a $3$-dimensional rectangle in $\R^3$ and let $\sP$ be a minimal locally regulated partition of $R$. Then for any level $d$ of $\sP$, any maximal surface on level $d$ is a ($2$-dimensional) rectangle.
\end{lem}

\begin{proof} Let $M$ be a maximal surface on level $d$. By definition, $M$ is a rectangular polygon (i.e. a union of finitely many rectangles) with connected interior. To show that $M$ is itself a rectangle, it suffices to show that $M$ is convex. Toward a contradiction, assume $M$ is not convex. Then $M$ contains a concave corner point $x$. This is illustrated in Figure~\ref{fig:3}. 

\begin{figure}[h]
\begin{tikzpicture}[scale=0.1]

\draw[color=gray!20, fill=gray!20] (0,10) -- (10,10) -- (10,0) -- (20,0) -- (20,20) -- (0,20) -- (0,10);
\draw (0,10) to (10,10) to (10,0);
\filldraw[fill=black, very thick] (10,10) circle (0.2);
\node[left] at (10,8) {$x$};
\node[left] at (17,15) {$M$};
\end{tikzpicture}
\caption{A hypothetical concave maximal surface. \label{fig:3}}
\end{figure}
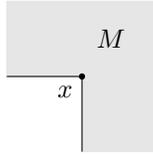

Let $R_x$ be a sufficiently small $3$-dimensional rectangle so that $R_x$ is a neighborhood of $x$ and 
$$\sQ=\{P\cap R_x\colon P\in\sP\}$$
is a locally regulated partition of $R_x$ about $x$. Since $\sP$ is minimal, $\sQ$ is minimal, and we have $|\sQ|=4$. However, it is obvious that $|\sQ_{d,+}|\geq 2$, $|\sQ_{d,-}|\geq 2$, and $\sQ_{d,+}\cup \sQ_{d,-}\neq \sQ$. Thus $|\sQ|\geq 5$, a contradiction.
\end{proof}

\begin{defn} \label{def:vms} Let $R=[a_1, b_1]\times  [a_2, b_2]\times [a_3, b_3]$ be a $3$-dimensional rectangle in $\R^3$, let $\sP$ be a locally regulated partition of $R$, and let $a_3\leq d_1<d_2\leq b_3$. 
A (2-dimensional) rectangle $M$ is a {\em virtual maximal surface} from $d_1$ to $d_2$ if
$$\partial M\times [d_1,d_2]\subseteq \bigcup\partial\sP=\bigcup_{P\in\sP} \partial P. $$

Given a two-dimensional rectangle $M$ we say the partition $\sP$ {\em respects} $M$
from $d_1$ to $d_2$ if 
$\partial M\times [d_1,d_2]\subseteq \bigcup\partial\sP=\bigcup_{P\in\sP} \partial P$. 
Thus, $M$ being a virtual maximal surface from $d_1$ to $d_2$ is just saying that 
$\sP$ respects $M$ from $d_1$ to $d_2$. 
\end{defn}

\begin{lem}\label{lem:vms0} Let $R$ be a $3$-dimensional rectangle in $\R^3$ and let $\sP$ be a minimal locally regulated partition of $R$. If $M$ is a maximal surface on level $d$, then there exist $d_1<d<d_2$ such that $M$ is a virtual maximal surface from $d_1$ to $d_2$.
\end{lem}

\begin{proof} By Lemma~\ref{lem:maxsurface0} $M$ is a (2-dimensional) rectangle. Consider 
$$\sP_{d,+}^M=\{P\in\sP_{d,+}\,\colon\, \pi_{1,2}(P)\subseteq M\}, \mbox{ and} $$ 
$$\sP_{d,-}^M=\{P\in\sP_{d,-}\,\colon\, \pi_{1,2}(P)\subseteq M\}. $$ 
Then $\bigcup\pi_{1,2}(\sP_{d,+}^M)=\bigcup\pi_{1,2}(\sP_{d,-}^M)=M$. Let $d_1$ be the largest of the lower levels of the elements of $\sP_{d,-}^M$, and let $d_2$ be the smallest of the upper levels of the elements of $\sP_{d,+}$. Then $M$ is a virtual maximal surface from $d_1$ to $d_2$.
\end{proof}



\begin{lem}\label{lem:maxsurface1} 
Let $R=S\times[a_3,b_3]$ be a $3$-dimensional rectangle in $\R^3$ and 
let $\sP$ be a minimal locally regulated partition of $R$. Let $M_1$ be a virtual
maximal surface from $d_1$ to $d_2$ with $a_3\leq d_1<d_2<b_3$ and $d_2$ the largest such. Then at least one of the following holds:
\begin{enumerate}
\item[(i)] there is a unique maximal surface $M_2$ on level $d_2$ such that $M_1\subsetneq M_2$;
\item[(ii)] there is $d_3>d_2$ and a virtual maximal surface $M_2$ from $d_2$ to $d_3$ such that 
\begin{itemize}
\item[(iia)] $M_1\subsetneq M_2$, 
\item[(iib)] $\pi_i(M_1)=\pi_i(M_2)$ for exactly one $i\in \{1,2\}$, 
\item[(iic)] for any $(a,b)\in M_2\setminus M_1$, $(a, b, d_2)$ lies on a maximal surface $N$ on the level $d_2$ so that $\Int(N)\cap \Int(M_1)\neq\varnothing$, and
\item[(iid)] for any maximal surface $N$ on the level $d_2$ so that $\Int(N)\cap \Int(M_1)\neq\varnothing$, we have $N\subseteq M_2$.
\end{itemize}
\end{enumerate} 
\end{lem}

\begin{proof} By definition, $M_1\subseteq S$ is a rectangle. Denote $T=M_1\times [d_1,d_2]$. Then $T$ is a $3$-dimensional rectangle. Consider the collections
$$ \sQ=\{P\in\sP\,\colon\, \mbox{$P$ has lower level $d_2$,  $\pi_{1,2}(\Int(P))\cap M_1\neq\varnothing$}\} $$
and 
$$ \sQ^+=\{P\in \sQ\,\colon\, \pi_{1,2}(\Int(P))\setminus M_1\neq\varnothing\}. $$
By the maximality of $d_2$, $\sQ^+\neq\varnothing$. 

Note that, for any side $L$ of $M_1$ which does not intersect $\bigcup\Int(\pi_{1,2}(\sQ^+))$, there is $d_L>d_2$ such that $L\times[d_2,d_L]\subseteq \bigcup \partial\sP$. It follows again from the maximality of $d_2$ that $\bigcup\Int(\pi_{1,2}(\sQ^+)$ intersects at least one side of $M_1$.

\begin{figure}[h]
\begin{tikzpicture}[scale=0.15]

\draw (15,3.2) to (21,3.2) to (21,5.2) to (15,5.2) to (15,3.2);
\draw (21,3.2) to (23,6.2) to (23,8.2) to (17,8.2) to (15,5.2);
\draw (21,5.2) to (23,8.2);
\node[left] at (21,6.7) {$P$};

\draw (0,0) to (40,0) to (40,5) to (22.2,5);
\draw (0,0) to (0,5) to (15,5);
\draw (40,0) to (48,12) to (48,17) to (8,17) to (0,5);
\draw (40,5) to (48,17);
\draw (44.7, 12) to (8,12) to (3.3,5);
\draw (8,12) to (8,17);

\node[left] at (35,8) {$M_1$};
\node[left] at (52,12) {$d_1$};
\node[left] at (52,17) {$d_2$};
\filldraw[fill=black, very thick] (22.2, 5) circle (0.2);
\node[left] at (25,3.1) {$x$};

\end{tikzpicture}
\caption{In search for the next virtual maximal surface. \label{fig:4}}
\end{figure}
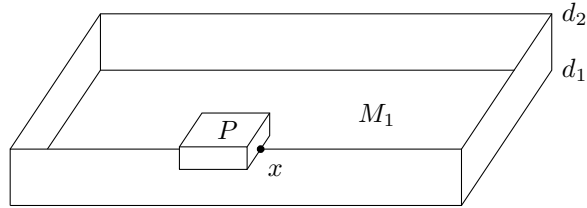

We claim that if $\bigcup\Int(\pi_{1,2}(\sQ^+))$ intersects a side $L$ of $M_1$, then $L\subseteq \bigcup \pi_{1,2}(\sQ^+)$. 
To see the claim, we consider the following situation illustrated by Figure~\ref{fig:4}. Suppose $P\in\sQ^+$, $x$ is an interior point of $L\times \{d_2\}$, and $x$ is on both a vertical face of $P$ and a vertical face of $T$. Since $\sP$ is minimal, we have that $\nu_{\sP}(x)=4$. By considering the possible ways in which $\nu_{\sP}(x)=4$ can be realized, a moment of reflection shows that there is a unique $Q\in\sQ^+$ such that $Q\neq P$ and $x\in \partial Q$. Thus $x\in \bigcup \pi_{1,2}(\sQ^+)$. By repeating this argument for points in both directions of $L$, we conclude that $L\subseteq \bigcup\pi_{1,2}(\sQ^+)$. 


Now if $\bigcup\Int(\pi_{1,2}(\sQ^+))$ intersects two adjacent sides of $M_1$, then by the claim, both these sides lie on a maximal surface $M_2$ on the level $d_2$. By Lemma~\ref{lem:maxsurface0}, $M_2$ is a rectangle and we must have $M_1\subsetneq M_2$, and thus conclusion (i) holds in this case. 

Next, let $M_1$ have sides $L_0, L_1, L_2, L_3$, where $L_0$ and $L_2$ are opposite and $L_1$ and $L_3$ are opposite. Moreover, suppose $\bigcup\Int(\pi_{1,2}(\sQ^+))$ intersects $L_0$ and does not intersect $L_1$ and $L_3$. For notational definiteness, suppose $M_1=[c_1, c_2]\times [e_1, e_2]$ and $L_0=[c_1,c_2]\times \{e_1\}$. By the above claim, $L_0$ lies on a maximal surface $N_0$ on the level $d_2$. Since $N_0$ is a rectangle by Lemma~\ref{lem:maxsurface0} , we may assume $N_0$ is of the form $[\alpha_1,\alpha_2]\times[\beta_1, \beta_2]$, where $\alpha_1\leq c_1$, $\alpha_2\geq c_2$, and $\beta_1<e_1<\beta_2$. Here we may assume $\beta_2<e_2$, since otherwise $N_0$ is a maximal surface on the level $d_2$ with $M_1\subsetneq N_0$, and hence conclusion (i) again holds with $M_2=N_0$. Now we claim that $\alpha_1=c_1$ and $\alpha_2=c_2$. Indeed, if $\alpha_1<c_1$, then consider the point $y=(c_1, \beta_2, d_2)$, which is in the interior of either $L_1\times\{d_2\}$ or $L_3\times\{d_2\}$; a similar argument as in the above claim gives a contradiction to our assumption that $\bigcup\Int(\pi_{1,2}(\sQ^+))$ does not intersect either $L_1$ or $L_3$. Thus $\alpha_1=c_1$. Similarly $\alpha_2=c_2$. This proves the claim.

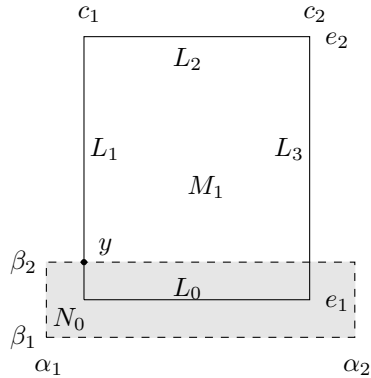
\begin{figure}[h]
\begin{tikzpicture}[scale=0.1]

\draw[color=gray!20, fill=gray!20] (0,0) -- (41,0) -- (41,10) -- (0,10) -- (0,0);
\draw[dashed] (0,0) -- (41,0) -- (41,10) -- (0,10) -- (0,0);
\draw (5,5) -- (35,5) -- (35,40) -- (5,40) -- (5,5);
\node[left] at (3.5,-4) {$\alpha_1$};
\node[left] at (44.5,-4) {$\alpha_2$};
\node[left] at (0, 0) {$\beta_1$};
\node[left] at (0,10) {$\beta_2$};
\node[left] at (8.5, 43) {$c_1$};
\node[left] at (38.5, 43) {$c_2$};
\node[left] at (41.5,4.5) {$e_1$};
\node[left] at (41.5, 39.5) {$e_2$};
\node[left] at (25,20) {$M_1$};
\node[left] at (6.5,2.5) {$N_0$};
\node[left] at (22, 6.5) {$L_0$};
\node[left] at (11, 25) {$L_1$};
\node[left] at (22, 37) {$L_2$};
\node[left] at (35.5, 25) {$L_3$};
\filldraw[fill=black, very thick] (5,10) circle (0.2);
\node[left] at (10,12) {$y$};
\end{tikzpicture}
\caption{The existence of the virtual maximal surface $M_2$. \label{fig:4half}}
\end{figure}

Now if $\bigcup\Int(\pi_{1,2}(\sQ^+))$ does not intersect $L_2$, then by the second claim and Lemma~\ref{lem:vms0} applied to $N_0$, we obtain a virtual maximal surface $M_2=[c_1, c_2]\times[\beta_1, e_2]$ from $d_2$ to some $d_3>d_2$ where $d_3\leq d_{L_1}, d_{L_2}, d_{L_3}$. In this case, $\pi_1(M_2)=\pi_1(M_1)$ and $\pi_2(M_2)\neq \pi_2(M_1)$. Thus conclusions (iia) and (iib) hold. Conclusion (iic) is also clear from the construction of $M_2$. To verify (iid), consider any maximal surface $N$ such that $\Int(N)\cap \Int(M_1)\neq\varnothing$. If $N\cap N_0\neq\varnothing$ then $N=N_0$ and $N\subseteq M_2$ by our construction. If $N\cap N_0=\emptyset$ then in fact we have that $N\subseteq M_1$, and thus $N\subseteq M_2$ as required. Otherwise, $\Int(N)\setminus M_1\neq\varnothing$, and we obtain a contradiction by an argument similar to the first claim of this proof. 

Finally, if $\bigcup\Int(\pi_{1,2}(\sQ^+))$ intersects $L_2$, then a similar claim holds for $L_2$ as for $L_0$, that is, there is a maximal surface $N_2$ on the level $d_2$ on which $L_2$ lies, which is of the form $[c_1,c_2]\times [\gamma_1, \gamma_2]$, where $\gamma_1<e_2<\gamma_2$. In this case we obtain a virtual maximal surface $M_2=[c_1,c_2]\times [\beta_1,\gamma_2]$ from $d_2$ to some $d_3>d_2$ with $d_3\leq d_{L_1}, d_{L_3}$. Again conclusion (ii) holds in this case. 
\end{proof}

\begin{thm} There exist a $3$-dimensional rectangle $R$ and a finite collection $\sP_0$ of pairwise disjoint $3$-dimensional subrectangles of $R$ such that there is no minimal locally regulated partition $\sP$ of $R$ with $\sP_0\subseteq \sP$.
\end{thm}

\begin{proof} Let $R=[0,4]^3$ and let $\sP_0=\{P_0, P_1, P_2\}$ where
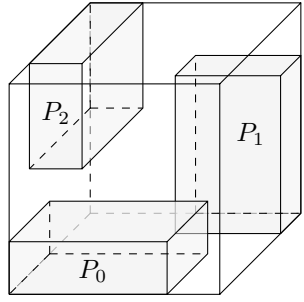
\begin{figure}[h]
\begin{tikzpicture}[scale=0.7]
\bakapp{4}{4}{4}{0}{0}{0};
\movapp{2}{1}{3}{2}{0}{0};
\movapp{3}{2}{1}{0}{0}{2};
\movapp{1}{3}{2}{0}{2}{0};
\oppapp{4}{4}{4}{0}{0}{0};
\node[left] at (0.5,-1.1) {$P_0$};
\node[left] at (3.5,1.5) {$P_1$};
\node[left] at (-0.2, 1.9) {$P_2$};
\end{tikzpicture}
\caption{A $3$-dimensional obstacle to minimal locally regulated partitions. \label{fig:5}}
\end{figure}

$$ \begin{array}{l}
P_0=[0,3]\times [0,2]\times [0,1], \\
P_1=[2,4]\times [3,4]\times [0,3], \\
P_2=[0,1]\times [1,4]\times [2,4].
\end{array}
$$
Figure~\ref{fig:5} illustrates $R$ and $\sP_0$.

Assume toward a contradiction that $\sP$ is a minimal locally regulated partition of $R$ with $\sP_0\subseteq \sP$. Let $d_1=1$ and let $M_1$ be the maximal surface at level $d_1$ containing $\pi_{1,2}(P_0)=[0,3]\times [0,2]$. By Lemma~\ref{lem:maxsurface0}, $M_1$ is a rectangle. However, because of the existence of $P_1$, we have that $M_1\subseteq [0,4]\times [0,3]$. 

By Lemma~\ref{lem:vms0} $M_1$ is a virtual maximal surface from $d_1$ to some $d_2>d_1$. Let $d_2$ be the largest such. By our construction, in particular because of the existence of $P_2$, we have $d_2\leq 2$. Then by Lemma~\ref{lem:maxsurface1}, either there is a maximal surface $M_2$ on the level $d_2$ or there is $d_3>d_2$ and a virtual maximal surface $M_2$ from $d_2$ to $d_3$ with $\pi_i(M_1)=\pi_i(M_2)$ for exactly one $i\in\{1,2\}$. In both cases, we have $M_1\subsetneq M_2\subseteq [0,4]\times [0,3]$. In fact, in the former case, $M_2$ is a maximal surface and therefore a rectangle, and $M_2\subseteq [0,4]\times [0,3]$ because of the existence of $P_1$. In the latter case, by Lemma~\ref{lem:maxsurface1} (iic), $M_2$ is a virtual maximal surface which extends $M_1$ only by maximal surfaces. Because of the spatial relationship between $P_0$ and $P_1$, and because $d_2\leq 2$, we also have $M_2\subseteq [0,4]\times[0,3]$.

Now this construction and argument can be repeated to obtain an increasing sequence 
$$ d_2=\delta_0<\delta_1<\dots<\delta_k\leq 2 $$
and virtual maximal surfaces
$$ M_2=N_0\subsetneq N_1\subsetneq \dots \subsetneq N_k $$
so that for each $0\leq i<k$, $N_{i+1}$ is a virtual maximal surface from $\delta_i$ to $\delta_{i+1}$. Inductively, we have that $N_i\subseteq[0,4]\times[0,3]$ for all $0\leq i\leq k$ by the same argument as above for $M_2$. Now let $k$ be the largest such. Since $\sP$ is finite, we must have $\delta_k=2$ and $N_k\subseteq [0,4]\times[0,3]$. On the other hand, by Lemma~\ref{lem:maxsurface1} (iid), $N_k$ must extend the maximal surface containing $\pi_{1,2}(P_2)=[0,1]\times[0,4]$, which is a contradiction.
\end{proof}

\begin{cor} For each $n\geq 3$ there exist an $n$-dimensional rectangle $R$ and a finite collection $\sP_0$ of pairwise disjoint $n$-dimensional subrectangles of $R$ such that there is no minimal locally regulated partition $\sP$ of $R$ with $\sP_0\subseteq\sP$.
\end{cor}

\begin{proof} By induction on $n\geq 3$. Suppose $n\geq 3$, $R_n$ is an $n$-dimensional rectangle and $\sP^n_0$ is a finite collection of pairwise disjoint $n$-dimensional subrectangles of $R_n$ such that there is no minimal locally regulated partition $\sP_n$ of $R$ with $\sP^n_0\subseteq \sP_n$. Let $R_{n+1}=R_n\times [0,1]$ and 
$$ \sP^{n+1}_0=\{P\times [0,1]\colon P\in \sP^n_0\}. $$
We claim that there is no minimal locally regulated partition $\sP$ of $R_{n+1}$ with $\sP^{n+1}_0\subseteq \sP$.

Assume toward a contradiction that $\sP$ is a minimal locally regulated partition of $R_{n+1}$ with $\sP^{n+1}_0\subseteq \sP$. Let $0<c<1$ that is not a level of $\sP$, i.e., there is no $P\in\sP$ with $\sup\pi_n(P)=c$. Then 
$$\sQ=\{P\cap (R_n\times\{c\})\colon P\in \sP\} $$
is a minimal locally regulated partition of $R_n\times\{c\}$. Let $A=\{1,\dots, n\}$. Then $\pi_A(\sQ)$ is a minimal locally regulated partition of $R_n$ with $\sP^n_0\subseteq \pi_A(\sQ)$, a contradiction.
\end{proof}

\section{Clopen Regulated Partitions of $F(2^{\Z^n})$\label{sec:6}}

\begin{defn} \label{znrn1}
Let $n\geq 1$.
\begin{enumerate}
\item[(1)] For each $x=(a_1,\dots, a_n)\in \Z^n$, let $C_x\subseteq \R^n$ be defined by
$$ C_x=\left[a_1-\frac{1}{2}, a_1+\frac{1}{2}\right]\times\cdots\times \left[a_n-\frac{1}{2}, a_n+\frac{1}{2}\right]. $$
\item[(2)] For each $P\subseteq \Z^n$, let
$$ C_P=\bigcup\{C_x\colon x\in P\}. $$
\end{enumerate}
\end{defn}

We naturally speak of {\em $n$-dimensional rectangles} in $\Z^n$. Specifically, they are sets of the form $R\cap \Z^n$ where $R$ is an $n$-dimensional rectangle in $\R^n$. If $P\subseteq \Z^n$ is an $n$-dimensional rectangle in $\Z^n$, then $C_P$ is an $n$-dimensional rectangle in $\R^n$ such that $C_P\cap \Z^n=P$.

For $a, b\in \Z$ with $a<b$, we denote the set $[a, b]\cap \Z$, called an {\em interval in $\Z$}, simply by $[a, b]$ if there is no danger of confusion. If $P$ is an $n$-dimensional rectangle in $\Z^n$, we may write 
$$ P=[a_1,b_1]\times \dots \times [a_n,b_n] $$
where $[a_1, b_1], \dots, [a_n, b_n]$ are intervals in $\Z$. The {\em side lengths} of $P$ are defined as $b_i-a_i$, $i=1,\dots, n$. Note that the side lengths of the corresponding $C_P$ are in fact $b_i-a_i+1$, $i=1,\dots, n$.

\begin{defn} \label{znrn2}
Let $n\geq 1$.
\begin{enumerate}
\item[(1)] A {\em rectangular partition} of $\Z^n$ is a partition of $\Z^n$ into pairwise disjoint $n$-dimensional rectangles in $\Z^n$.
\item[(2)] For a rectangular partition $\sP$ of $\Z^n$, let
$$ \sC_{\sP}=\{ C_P\colon P\in \sP\}. $$
\end{enumerate}
\end{defn}

If $\sP$ is a rectangular partition of $\Z^n$, then $\sC_{\sP}$ is a regulated partition of $\R^n$ such that 
$$ \sP=\{ Q\cap \Z^n\colon Q\in \sC_{\sP}\}. $$
Thus we also call a rectangular partition of $\Z^n$ a {\em regulated partition} of $\Z^n$.

\begin{defn} \label{gammaregfzn}
Let $n\geq 1$, let $n+1\leq \gamma\leq 2^n$, and let $\sP$ be a rectangular partition of $\Z^n$. We say that $\sP$ is a {\em $\gamma$-regulated partition} of $\Z^n$ if $\sC_{\sP}$ is a $\gamma$-regulated partition of $\R^n$. We say that $\sP$ is {\em minimal} if $\sC_{\sP}$ is minimal.
\end{defn}

\begin{defn} Let $\sP$ and $\sQ$ be regulated partitions of $\Z^n$. We say that $\sP$ and $\sQ$ are {\em orthogonal} if $\sC_{\sP}$ and $\sC_{\sQ}$ are orthogonal regulated partition of $\R^n$.
\end{defn}

\begin{lem} For each $n\geq 1$, there exist orthogonal minimal regulated partitions of $\Z^n$. In particular, there exists a minimal regulated partition of $\Z^n$.
\end{lem}

\begin{proof} By an adaptation of the proof of Lemma~\ref{lem:orthogonal}.
\end{proof}

Next we consider continuous and Borel constructions in the space $F(2^{\Z^n})$. Consider the {\em Bernoulli shift} action of $\Z^n$ on $2^{\Z^n}$, which is defined by
$$ (g\cdot z)(h)=z(h+g) $$
for $g, h\in \Z^n$ and $x\in 2^{\Z^n}$. $F(2^{\Z^n})$ is the {\em free part} of $2^{\Z^n}$, i.e., $F(2^{\Z^n})$ is the set of $z\in 2^{\Z^n}$ such that $g\cdot z\neq z$ for all nonidentity $g\in \Z^n$. Let $E_{\Z^n}$ be the {\em orbit equivalence relation} on $F(2^{\Z^n})$, i.e., 
$$ (u,v)\in E_{\Z^n} \iff \exists g\in \Z^n\ (g\cdot u=v). $$
Then $E_{\Z^n}$ is an $F_\sigma$ subset of $F(2^{\Z^n})\times F(2^{\Z^n})$. For each $z\in F(2^{\Z^n})$, the {\em orbit} of $z$ is the $E_{\Z^n}$-class containing $z$, and we denote it as $[z]_{E_{\Z^n}}$ and abbreviate it as $[z]$.

Suppose $P\subseteq [z]$ for some $z\in F(2^{\Z^n})$. Then $P=S\cdot z$ for some $S\subseteq \Z^n$. We say that $P$ is an {\em $n$-dimensional rectangle} if $S$ is an $n$-dimensional rectangle in $\Z^n$. Note that this does not depend on the choice of $z$. 

\begin{defn}\label{def:fzn} A {\em regulated partition} of $F(2^{\Z^n})$ is a partition $\sP$ where each $P\in \sP$ is an $n$-dimensional rectangle. If $\sP$ is a regulated partition of $F(2^{\Z^n})$, then we say that $\sP$ is {\em clopen} ({\em Borel}) if for every $g\in \Z^n$, the set
$$ \{u\colon \exists P\in \sP\ (u, g\cdot u)\in P\} $$
is a relatively clopen (Borel) subset of $F(2^{\Z^n})$. 
\end{defn}

\begin{thm}[{\cite[Theorem 3.1]{GJ2015}}] \label{crp}
For any integers $n, \ell\geq 1$ there exists a clopen regulated partition $\sP$ of $F(2^{\Z^n})$ such that each $P\in \sP$ is an $n$-dimensional rectangle with side lengths either $\ell$ or $\ell+1$.
\end{thm}

\begin{defn} 
Let $n\geq 1$, let $n+1\leq \gamma\leq 2^n$, and let $\sP$ be a regulated partition of $F(2^{\Z^n})$. We say that $\sP$ is a {\em $\gamma$-regulated partition} of $F(2^{\Z^n})$ if for each $z\in F(2^{\Z^n})$,
$$ \{S\subseteq \Z^n\colon S\cdot z\in \sP\} $$
is a $\gamma$-regulated partition of $\Z^n$. We say that $\sP$ is {\em minimal} if it is an $(n+1)$-regulated partition of $F(2^{\Z^n})$.
\end{defn}

The main result of this section is the following theorem on clopen regulated partitions.

\begin{thm}\label{thm:gamma} For each $n\geq 2$ there is a clopen $3\cdot 2^{n-2}$-regulated partition of $F(2^{\Z^n})$.
\end{thm}

We note that when $n=1$ there is obviously a clopen $2$-regulated partition of $F(2^{\mathbb{Z}})$; such a regulated partition is minimal. Among all the dimensions considered in this theorem, only in the case $n=2$ is the regulated partition obtained minimal. In the rest of this paper we will show that there do not exist Borel minimal regulated partitions of $F(2^{\Z^n})$ for any $n\geq 3$. The rest of this section is devoted to a proof of Theorem~\ref{thm:gamma}. 

Throughout the rest of this section fix $n\geq 2$. Let $\rho$ be the $\ell_\infty$-metric on $\R^n$. Then for any $x\in \R^n$ and any real number $r>0$, the $\rho$-closed ball of radius $r$ centered at $x$, $B_{\rho}(x,r)=\{y\in \R^n\colon \rho(x,y)\leq r\}$, is an $n$-dimensional rectangle in $\R^n$ with side length $2r$.  The distance function $\rho$ can be extended to $F(2^{\Z^n})$:  for $u, v\in F(2^{\Z^n})$, define
$$\rho(u,v)=\left\{\begin{array}{ll} \rho(g,\vec{0}), & \mbox{ if $g\cdot u=v$,} \\ \infty, & \mbox{ if $(u,v)\not\in E_{\Z^n}$.}\end{array}\right. $$

We will use the following lemmas proved in \cite{GJ2015}.

\begin{lem}[Basic clopen marker lemma, {\cite[Lemma 2.1]{GJ2015}}]\label{lem:basicclopen} Let $r>0$ be any real number. Then there is a relatively clopen set $S\subseteq F(2^{\Z^n})$ such that
\begin{enumerate}
\item[(i)] if $u,v\in S$ are distinct, then $\rho(u,v)>r$;
\item[(ii)] for any $u\in F(2^{\Z^n})$, there is $v\in S$ such that $\rho(u,v)\leq r$.
\end{enumerate}
\end{lem}

\begin{lem}[{\cite[Lemma 2.2]{GJ2015}}]\label{lem:basicclopen2} Let $r_2>2r_1>0$ be real numbers, and let $S_1$ be a relatively clopen subset of $F(2^{\Z^n})$ as in Lemma~\ref{lem:basicclopen} for $r_1$. Then there is a clopen set $S_2\subseteq S_1$ such that
\begin{enumerate}
\item[(i)] if $u,v\in S_2$ are distinct, then $\rho(u,v)>r_2-2r_1$;
\item[(ii)] for any $u\in F(2^{\Z^n})$, there is $v\in S_2$ such that $\rho(u,v)\leq r_1+r_2$.
\end{enumerate}
\end{lem}

Toward proving Theorem~\ref{thm:gamma}, let $r_2\gg r_1\gg \ell\gg 0$ and let $S_2\subseteq S_1$  be clopen subsets of $F(2^{\Z^n})$ given by Lemmas~\ref{lem:basicclopen} and \ref{lem:basicclopen2} for $r_1$ and $r_2$. Let $g_1,\dots, g_k$ enumerate all nonidentity elements of the set
$$ D=\{ g\in \Z^n\colon \rho(g,\vec{0})\leq r_1+r_2\}. $$
Inductively define a sequence $T_0, T_1,\dots, T_k$ of subsets of $S_1$ as follows:
$$ \begin{array}{l} T_0=S_2, \\ T_i=S_1\cap (g_i\cdot S_2)\setminus \bigcup_{j<i} T_j, \mbox{ for $1\leq i\leq k$.} 
\end{array}$$
Then $T_0, T_1,\dots, T_k$ form a partition of $S_1$ by clopen sets. By Lemma~\ref{lem:basicclopen2} (i), for any $0\leq i\leq k$ and any distinct $u, v\in T_i$, $\rho(u,v)>r_2-2r_1$.

The rest of our construction will be invariant. This means that, one can view the construction as if we have chosen a base point $z$ in each $E_{\Z^n}$-class, but the result of the construction will not depend on the choice of $z$. Moreover, we will describe a construction of a regulated partition $\sP_z$ of $\R^n$ for each orbit $[z]$. The desired regulated partition of $F(2^{\Z^n})$ consists of $(P\cap\Z^n)\cdot z$ for all $P\in \sP_z$. 

To each $z\in T_i$ we associate an $n$-dimensional rectangle $R_z$ by induction on $0\leq i\leq k$. Without loss of generality assume $r_1\in \Z+\frac{1}{2}$ and $\ell>0$ is an integer. If $i=0$ we let $R_z$ be the $n$-dimensional rectangle $B_\rho(z, r_1)$. Note that if $u, v\in T_0$ are distinct then $R_u\cap R_v=\varnothing$ since $\rho(u,v)>r_2-2r_1\gg 2r_1$ if we ensure $r_2\gg r_1$ appropriately. Suppose $0<i\leq k$ and assume $R_u$ has been defined for all $u\in \bigcup_{j<i}T_j$. Let $z\in T_i$. We first let $R^*_z$ be the $n$-dimensional rectangle $B_\rho(z,r_1)$. Then for each face $F^*$ of $R^*_z$, if there is a $u\in \bigcup_{j<i}T_j$ such that $\rho(u,z)\leq 3r_1$ and there is a face $F$ of $R_u$ such that the perpendicular distance of $F^*$ and $F$ is less than $\ell$, then adjust $F^*$ by moving it outward by no more than $\frac{1}{10}r_1$ so that the resulting face has level in $\Z+\frac{1}{2}$ and has perpendicular distance at least $\ell$ from all parallel faces of $R_u$ for $u\in\bigcup_{j<i}T_j$ with $\rho(u,z)\leq 6r_1$. Note that there is a constant upper bound (depending on $n$ but nothing else) on the number of $u\in \bigcup_{j<i}T_j$ with $\rho(u,z)\leq 6r_1$. Hence if $r_1\gg\ell$ have been appropriately chosen, such an adjustment of $F^*$ can be found. The $n$-dimensional rectangle $R_z$ is the result after such adjustments for all faces of $R^*_z$. Note that if $u,v\in T_i$ are distinct then $\rho(u,v)>r_2-2r_1\gg 6r_1$, and thus the constructions of $R_u$ and $R_v$ described above do not interfere with each other; in particular, any parallel faces of $R_u$ and $R_v$ have perpendicular distance at least $\ell$. We refer to this property as {\em orthogonality}.

By Lemma~\ref{lem:basicclopen} (ii), on each orbit $\bigcup_{0\leq i\leq k}\{R_u\colon u\in T_i\}$ is a regulated cover of $\R^n$.

Next, for each $z\in T_i$ we define a locally regulated partition of 
$$ H_z=R_z\setminus\bigcup_{j<i}\{R_u\colon u\in T_j\} $$
using the following recursive division algorithm. Note that $H_z$ is an {\em $n$-dimensional rectangular polyhedron}, i.e., it is homeomorphic to an $n$-dimensional rectangle and is the union of finitely many $n$-dimensional rectangles. Moreover, any two distinct parallel faces of $H_z$ have perpendicular distance at least $\ell$. The algorithm is as follows. Take all faces of $H_z$ with normal vector $e_n$, extend them into hyperplanes, and use these hyperplanes to divide $H_z$ into finitely many sets of the form $L\times [a, b]$, where $L$ is an $(n-1)$-dimensional rectangular polyhedron and $a, b\in \Z+\frac{1}{2}$. Note here that $L$ satisfies the similar property that any two distinct parallel faces of $L$ have perpendicular distance at least $\ell$. Now apply this algorithm recursively to each $L$ to obtain a regulated partition $\sP_L$ of $L$. Then sets of the form $P\times [a,b]$ for $P\in \sP_L$ give a locally regulated partition $\sP_z$ of $H_z$. 

When we have completely defined the locally regulated partitions $\sP_z$ of $H_z$ for all $z\in T_i$ and $0\leq i\leq k$, we have obtained a regulated partition $\sP_{[z]}$ of $\R^n$ for each orbit. Since this construction is finitary, the resulting regulated partition on $F(2^{\Z^n})$ is clopen.

It remains to verify that on each orbit we have obtained a $3\cdot 2^{n-2}$-regulated partition of $\R^n$. For this we consider a regulated partition $\sP_{[z]}$ of $\R^n$ for an arbitrary orbit $[z]$. Let $x\in \R^n$ and let $\gamma_n=3\cdot 2^{n-2}$. Let $R$ be an $n$-dimensional rectangle in $\R^n$ with $x\in \Int(R)$ such that $\sP=\{P\cap R\colon P\in\sP_{[z]}\}$ is a locally regulated partition of $R$ about $x$.  
We note that if $x$ has no boundary vectors with respect to $\mathcal{P}$, then $|\sP|=\nu_{\sP}(x)=1\leq \gamma_n$ and there is nothing more to prove. Thus we may assume without loss of generality that $x$ has at least one boundary vector with respect to $\sP$. Also note that if there is $u\in S_1$ such that $x\in \Int(H_u)$, then by our construction, $e_1$ is not a boundary vector of $x$; in this case $|\sP|\leq 2^{n-1}\leq \gamma_n$, and we are also done. In the rest of the proof we assume that $x\not\in \Int(H_u)$ for any $u\in S_1$. Let $i_x$ be the largest $i$ with $0\leq i\leq k$ and $x\in\partial H_z$ for some $z\in T_i$. The $z\in T_{i_x}$ witnessing $x\in \partial H_z$ is unique; we fix it in the rest of the proof. 

If $e_j$ is a boundary vector of $x$ with respect to $\mathcal{P}$, we call $e_j$ {\em active} for $x$ if $e_j$ is the normal vector of a face $F$ of $\partial H_z$ with $x\in F$, and call $e_j$ {\em passive} for $x$ otherwise. In particular, if $x$ is on a hyperplane extending a face of $H_z$,  then the normal vector of the face is active for $x$ iff $x$ is on that face of $H_z$. We claim that this statement still holds for all $u\in S_1$ with $x\in \partial H_u$. To see this, suppose $u\in S_1$ is such that $x\in \partial H_u$. Then $u\in T_i$ for some $0\leq i\leq i_x$. Now suppose $x$ is on a hyperplane extending a face $F$ of $H_u$. Suppose first $x\in F$. Then there is a face $F'$ of $H_z$ such that $F'\subseteq F$ and $x\in F'$. Hence the normal vector of $F$ is active for $x$. Suppose next $x\not\in F$. Then by the orthogonality, there is no face $F'$ of $H_z$ with $x\in F'$ and $F'$ lying on the same hyperplane as the one extending $F$. Hence the normal vector of $F$ is passive for $x$. 


We note that, by our construction, $e_1$ is not passive for $x$, since no faces of $R_u$ with normal vector $e_1$ are ever extended in the construction. Since $x\in \partial H_z$, there exists at least one active boundary vector for $x$.  Let $\alpha(x)$ denote the number of active boundary vectors of $x$ and $\pi(x)$ denote the number of passive boundary vectors of $x$. We have $\alpha(x)\geq 1$ and $\alpha(x)+\pi(x)\leq n$. Note that the number of regions in $\sP$ that are defined by the active boundary vectors (namely the hyperplanes with the active boundary vectors as normal vectors) is at most $\alpha(x)+1$. This is because, if $e_j$ and $e_l$ are both active vectors and $1\leq j<l\leq n$, the hyperplane with $e_j$ as the normal vector does not cross the hyperplane with $e_l$ as the normal vector, according to our construction. Thus the number of elements in $\sP$ is at most $(\alpha(x)+1)\cdot 2^{\pi(x)}$. If $\alpha(x)\geq 2$ or $\alpha(x)+\pi(x)<n$, then we have that
$$ |\sP|\leq (\alpha(x)+1)\cdot 2^{\pi(x)}\leq (\alpha(x)+1)\cdot 2^{n-\alpha(x)}\leq 3\cdot 2^{n-2}=\gamma_n. $$

It remains to consider the case $\alpha(x)=1$ and $\pi(x)=n-1$. Since $e_1$ is not passive for $x$, in this case the only active boundary vector of $x$ is $e_1$ and every other $e_j$, $1<j\leq n$, is a  passive boundary vector of $x$. Under our assumption, for each $1\leq j\leq n$, there exists a face $F_j$ of some $P_j\in \sP$ such that $x\in F_j$ and $e_j$ is a normal vector of $F_j$. Let $N_j$ be the hyperplane containing $F_j$. 

For distinct $1\leq j,l\leq n$, we say that $e_j$ {\em crosses} $e_l$ if on each of the two sides of $N_l$ there are faces $F_j^*$ of some $P_j^*\in \sP$ with normal vector $e_j$. Note that, under our assumption, $e_1$ crosses $e_l$ for all $1<l\leq n$.

\begin{lem} \label{lem:crossing} Under our assumption that $\alpha(x)=1$ and $\pi(x)=n-1$, there exist distinct $1\leq j,l\leq n$ such that $e_j$ does not cross $e_l$.
\end{lem}

\begin{proof} Let $u_1=z, u_2, \dots, u_n$ be elements of $\bigcup_{0\leq i\leq i_x}T_{i}$ such that
\begin{itemize}
\item[(*)] for each $1<j\leq n$, $N_j$ extends a face $F_j$ of $R_{u_j}$, $x\in N_j$ and $x\not\in R_{u_j}$.
\end{itemize}
Intuitively, $x$ lies in the interior of a face of $R_{u_1}$ with normal vector $e_1$, and for each $1<j\leq n$, $x$ is also on the extension of a face of $R_{u_j}$ since $e_j$ is a passive boundary vector of $x$. By the orthogonality of the construction, $u_j$ is unique for each $1\leq j\leq n$.

Assume toward a contradiction that each $e_j$ crosses $e_l$ for all distinct $1\leq j,l\leq n$. Then we claim that for distinct $1\leq j,l\leq n$, $u_j\neq u_l$. First it is clear that $u_1\neq u_j$ for any $1<j\leq n$. Now assume $1<j<l$ and $u_j=u_l=u$, then in our construction the faces of $R_u$ with normal vector $e_l$ were extended first and the faces of $R_u$ with normal vector $e_j$ were extended later. By the orthogonality of our construction, $e_j$ does not cross $e_l$, a contradiction.

Before continuing the proof we adopt a coordinate system to address the elements of $\sP_{[z]}=\sP_{[x]}$. Suppose for each $1\leq j\leq n$,
$$ R_{u_j}=[a^j_1, b^j_1]\times \cdots \times [a^j_n, b^j_n]. $$
Suppose $x$ has coordinates $(x_1,\dots, x_n)$. Then each $x_j$ is either $a^j_j$ or $b^j_j$. Without loss of generality, assume $x_j=a^j_j$. This is also the equation for the hyperplane $N_j$. 

Since for each $2<j\leq n$, $e_2$ crosses $e_j$, we have $a^2_j\leq a^j_j\leq b^2_j$. Since $e_2$ crosses $e_1$, we have $a^2_1\leq a^1_1\leq b^2_1$. It follows that 
$$x=(a^1_1,a^2_2,\dots, a^n_n)\in R_{u_2}. $$
This is a contradiction to (*).
\end{proof}

Suppose $e_j$ does not cross $e_l$. Then on one side of $N_l$ we have at most $2^{n-1}$ many elements of $\sP$ and on the other side there are at most $2^{n-2}$ many elements of $\sP$. This gives 
$$ |\sP|\leq 2^{n-1}+2^{n-2}=3\cdot 2^{n-2}=\gamma_n. $$
The proof of Theorem~\ref{thm:gamma} is complete.

\section{A Clopen $5$-Regulated Partition of $F(2^{\mathbb{Z}^3})$\label{sec:7}}
In this section we give a clopen $5$-regulated partition of $F(2^{\mathbb{Z}^3})$. This improves the $n=3$ case of Theorem~\ref{thm:gamma} and is optimal in view of the results to be proved in the next section.

\begin{thm}\label{thm:3dim5reg} There is a clopen $5$-regulated partition of $F(2^{\mathbb{Z}^3})$. 
\end{thm}

The rest of this section is devoted to a proof of Theorem~\ref{thm:3dim5reg}. We modify the construction in the proof of Theorem~\ref{thm:gamma} for the $n=3$ case. We use the notation in that proof as much as possible.

First let us quickly review the algorithm used in the previous construction in the context of $n=3$. Let $r_2\gg r_1\gg \ell\gg 0$. Let $S_1\supseteq S_2$ be clopen subsets of $F(2^{\mathbb{Z}^3})$ given by Lemmas~\ref{lem:basicclopen} and \ref{lem:basicclopen2} for $r_1$ and $r_2$, and let $T_0=S_2, T_1, \dots, T_k$ be a partition of $S_1$ so that for any $0\leq i\leq k$ and distinct $u, v\in T_i$, $\rho(u,v)$ is large (roughly $r_2$). Then to each $z\in T_i$ we associate a $3$-dimensional rectangle $R_z$ so that the set of all $R_z$ satisfies the orthogonality property. This means that for any distinct $u,v\in S_1$ which are close to each other, the perpendicular distance of the parallel faces of $R_u$ and $R_v$ are relatively large (larger than $\ell$). For each $z\in T_i$ let 
$$ H_z=R_z\setminus \bigcup_{j<i}\{R_u\colon u\in T_j\}. $$
Then $H_z$ is a $3$-dimensional rectangular polygon. Our algorithm to divide $H_z$ goes as follows. Let $e_1, e_2, e_3$ be the standard basis vectors of $\mathbb{R}^3$. We first extend all faces of $H_z$ with normal vector $e_3$. This divides $H_z$ up into a number of $3$-dimensional rectangular polygons each of which is a cylinder in the direction of $e_3$. To be precise, arbitrarily fixing a coordinate system for $[z]$, each of the resulting $3$-dimensional rectangular polygons is of the form $K\times I$, where $K$ is a $2$-dimensional rectangular polygon and $I$ is an interval. Then the next step of the algorithm is to divide $K\times I$ by extending all faces of $K\times I$ with normal vector $e_2$. This results in a number of $3$-dimensional rectangles. These $3$-dimensional rectangles give rise to a clopen regulated partition of $F(2^{\mathbb{Z}^3})$. 

We analyze this regulated partition in more detail. Let $\mathcal{P}_{[z]}$ denote the regulated partition restricted to $[z]$ (which we identify as a regulated partition of $\mathbb{R}^3$). Arbitrarily fix $x\in [z]$. Let $\mathcal{P}$ be a locally regulated partition about $x$ induced by $\mathcal{P}_{[z]}$.  We again let $\alpha(x)$ be the number of active boundary vectors of $x$ and $\pi(x)$ be the number of passive boundary vectors of $x$. If $\alpha(x)+\pi(x)\leq 2$, then $|\mathcal{P}|\leq 4$. Suppose $\alpha(x)+\pi(x)=3$. If $\alpha(x)=3$ then the algorithm gives that $|\mathcal{P}|=4$. If $\alpha(x)=2$ and $\pi(x)=1$, then the algorithm gives that $|\mathcal{P}|=4$ or $|\mathcal{P}|=5$. Thus the only case that could give $|\mathcal{P}|=6$ is when $\alpha(x)=1$ and $\pi(x)=2$. In this case Lemma~\ref{lem:crossing} guarantees that there are distinct $1\leq i,j\leq 3$ such that $e_i$ does not cross $e_j$. Let $k\in \{1,2,3\}\setminus\{i,j\}$. We see that the only case in which $|\mathcal{P}|=6$ is when $e_k$ crosses both $e_i$ and $e_j$, $e_j$ crosses both $e_i$ and $e_k$, $e_i$ crosses $e_k$ but does not cross $e_j$. We claim that $i=2$ and $\{j,k\}=\{1,3\}$. To see this, let $F_1$ be the union of all faces $F$ of some elements of $\sP$ so that $x\in F$ and $F$ has $e_1$ as the normal vector; similarly define $F_2$ and $F_3$. Since $e_1$ is the only active boundary vector for $x$, $x$ must lie in the interior of $F_1$ and $F_1$ is a part of some face of some $H_z$.  This implies that $e_1$ crosses both $e_2$ and $e_3$. We next argue that $e_3$ crosses $e_1$. Assume not. Then $F_3$ only occurs on only one side of $F_1$. By our assumption, $e_2$ crosses $e_1$.  Now if $e_3$ does not cross $e_2$, that is, $F_3$ does not occur on both sides of $F_2$, then one can see that $|\sP|\leq 5$. Otherwise, if $e_3$ crosses $e_2$, that is, $F_3$ occurs on both sides of $F_2$,  then on the side of $F_1$ in which $F_3$ occurs, $F_2$ does not occur on both sides of $F_3$. This is because, in our construction, $F_3$ was introduced first by the algorithm, and by orthogonality, $F_2$ cannot occur on both sides of $F_3$ while $F_2$ occurs on both sides of $F_1$. We have thus argued that $e_3$ crosses $e_1$. By a similar argument $e_3$ also crosses $e_2$. The proof of the claim is complete.

In case $|\sP|=6$ we call $x$ {\em troublesome}. 
 Figure~\ref{fig:6} illustrates the only two cases of troublesome points.

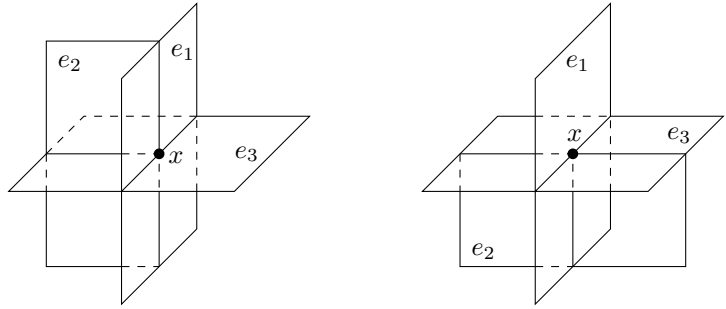
\begin{figure}[h]
\begin{tikzpicture}[scale=0.5]

\draw[] (2,2) --(5,2) -- (7,4) -- (4,4)--(2,2);
\draw[] (4,2)--(4,1)--(2,-1)--(2,2);
\draw[] (2,2)--(2,5)--(4,7)--(4,4);
\draw[] (3,2)--(3,0);
\draw[](3,3)--(3,6)--(0,6)--(0,3)--(2,3);
\draw[] (0,3)--(-1,2)--(2,2);
\draw[] (0,2)--(0,0)--(2,0);
\draw[dashed] (0,3)--(1,4)--(4,4);
\draw[dashed] (2,3)--(3,3)--(3,2);
\draw[dashed] (2,0)--(3,0);
\draw[dashed] (4,2)--(4,4);
\draw[dashed] (0,3)--(0,2);
\node[left] at (1.2,5.4) {$e_2$};
\node[left] at (4.2,5.7) {$e_1$};
\node[left] at (5.9,3) {$e_3$};
\node[left] at (3.9, 2.9) {$x$};
\filldraw[fill=black, very thick] (3,3) circle (0.1);

\draw[] (13,2)--(16,2)--(18,4)--(15,4)--(13,2) --(13,5)--(15,7)--(15,4);
\draw[] (13,2)--(13,-1)--(15,1)--(15,2);
\draw[] (13,2)--(10,2)--(12,4)--(13,4);
\draw[] (13,3)--(11,3);
\draw[] (14,3)--(17,3)--(17,0)--(14,0)--(14,2);
\draw[] (11,2)--(11,0)--(13,0);
\draw[dashed] (13,4)--(15,4)--(15,2);
\draw[dashed] (13,3)--(14,3) --(14,2);
\draw[dashed] (13,0)--(14,0);
\draw[dashed] (11,3)--(11,2);
\node[left] at (12.2,0.4) {$e_2$};
\node[left] at (14.7,5.4) {$e_1$};
\node[left] at (17.4,3.5) {$e_3$};
\node[left] at (14.5, 3.5) {$x$};
\filldraw[fill=black, very thick] (14,3) circle (0.1);

\end{tikzpicture}
\caption{Two cases of troublesome points. \label{fig:6}}
\end{figure}

We then modify $\mathcal{P}_{[z]}$ and define a new regulated partition $\mathcal{Q}_{[z]}$. First note that the set of all troublesome points for $\mathcal{P}_{[z]}$, for all $z\in F(2^{\mathbb{Z}^3})$, is clopen. Now, we modify $\mathcal{P}_{[z]}$ in two steps. In the first step, we find all troublesome points $x$, and around each of them, introduce a $3$-dimensional rectangle $L_x$ into $\mathcal{Q}_{[z]}$. One side face of $L_x$ is a part of the plane with normal vector $e_j$ where $e_2$ does not cross $e_j$. We choose the other side faces of $L_x$ so that their perpendicular distances to all existing parallel faces in the neighborhood are at least $\epsilon\ell$ where $\epsilon>0$ is a constant, and at the same time the side lengths of $L_x$ are between $\epsilon\ell$ (inclusive) and $\ell$. This is feasible since there is a constant $C$ so that there are at most $C$ many elements of $\mathcal{P}_{[z]}$ within a distance of $6r_1$ of $x$, and we may choose $\epsilon=1/(2C)$. Figure~\ref{fig:7} illustrates the introduction of $L_x$ for the two cases of troublesome points.

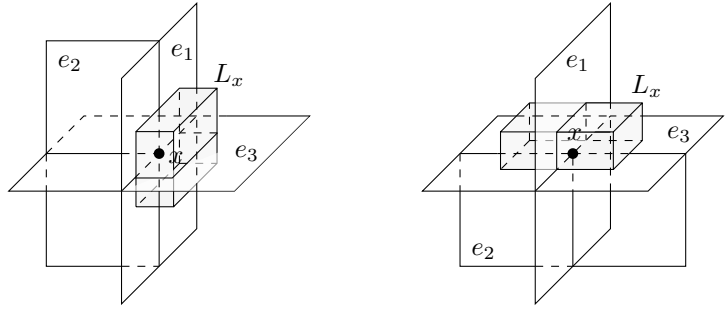
\begin{figure}[h]
\begin{tikzpicture}[scale=0.5]


\movapp{1}{3}{2}{3}{2.2}{-1.4};
\draw[] (2,2) --(5,2) -- (7,4) -- (4.8,4);
\draw[dashed] (4.8,4)--(4,4)--(2.35,2.35);
\draw[] (2,2)--(2.35,2.35);
\fill[white, opacity=0.6] (2,2)--(5,2)--(7,4)--(2,2);
\draw[] (2.35,2.35)--(3.35,2.35)--(4.55,3.55);
\draw[dashed] (4.55,3.55)--(3.55,3.55);
\draw[] (4,2)--(4,1)--(2,-1)--(2,2);
\draw[] (2,2)--(2,5)--(4,7)--(4,4.8);
\draw[dashed] (4,4.8)--(4,4);
\draw[] (3,2)--(3,0);
\draw[dashed] (3,3)--(3,4);
\draw[](3,4)--(3,6)--(0,6)--(0,3)--(2,3);
\draw[] (0,3)--(-1,2)--(2,2);
\draw[] (0,2)--(0,0)--(2,0);
\draw[dashed] (0,3)--(1,4)--(4,4);
\draw[dashed] (2,3)--(3,3)--(3,2);
\draw[dashed] (2,0)--(3,0);
\draw[dashed] (4,2)--(4,4);
\draw[dashed] (0,3)--(0,2);
\node[left] at (1.2,5.4) {$e_2$};
\node[left] at (4.2,5.7) {$e_1$};
\node[left] at (5.9,3) {$e_3$};
\node[left] at (3.9, 2.9) {$x$};
\filldraw[fill=black, very thick] (3,3) circle (0.1);
\node[right] at (4.2,5.1) {$L_x$};

\movapp{3}{2}{1}{12.5}{3}{-0.9};
\fill[white, opacity=0.6] (13,2)--(13,5)--(15,7)--(15,4)--(13,2);
\draw[] (13.57,2.57)--(13.57,3.57)--(15,3.57);
\draw[] (13.57,3.57)--(14.35,4.35)--(15,4.35);
\draw[dashed] (14.35, 4.35)--(14.35,3.35);
\draw[] (13,2)--(16,2)--(18,4)--(16,4);
\draw[dashed] (16,4)--(15,4)--(13.57,2.57);
\draw[] (13.57,2.57)--(13,2);
\draw[] (13,2) --(13,5)--(15,7)--(15,4.4);
\draw[dashed] (15,4.4)--(15,4);
\draw[] (13,2)--(13,-1)--(15,1)--(15,2);
\draw[] (13,2)--(10,2)--(12,4)--(12.5,4);
\draw[dashed] (12.5,4)--(13,4);
\draw[] (12,3)--(11,3);
\draw[dashed] (14,3)--(15.5,3);
\draw[] (15.5,3)--(17,3)--(17,0)--(14,0)--(14,2);
\draw[] (11,2)--(11,0)--(13,0);
\draw[dashed] (13,4)--(15,4)--(15,2);
\draw[dashed] (12,3)--(14,3) --(14,2);
\draw[dashed] (13,0)--(14,0);
\draw[dashed] (11,3)--(11,2);
\node[left] at (12.2,0.4) {$e_2$};
\node[left] at (14.7,5.4) {$e_1$};
\node[left] at (17.4,3.5) {$e_3$};
\node[left] at (14.5, 3.5) {$x$};
\filldraw[fill=black, very thick] (14,3) circle (0.1);
\node[right] at (15.3,4.8) {$L_x$};

\end{tikzpicture}
\caption{Introducing $L_x$ around each troublesome point $x$. \label{fig:7}}
\end{figure}

Note that $L_x$ intersects exactly two of the six $3$-dimensional rectangles in $\mathcal{P}_{[z]}$ which contain $x$. Denote these two elements of $\mathcal{P}_{[z]}$ as $P_{x,1}$ and $P_{x,2}$. Let $Q_{x,1}$ be the closure of $P_{x,1}\setminus L_x$ and let $Q_{x,2}$ be the closure of $P_{x,2}\setminus L_x$. Then $Q_{x,1}\subseteq P_{x,1}$ and $Q_{x,2}\subseteq P_{x,2}$ are $3$-dimensional rectangular polygons. Now the second step of the construction of $\mathcal{Q}_{[z]}$ is to perform a division algorithm in $Q_{x,1}$ and $Q_{x,2}$ which is similar to the division algorithm we have used for $\mathcal{P}_{[z]}$ before, except that we choose the normal vectors of the face extensions differently in $Q_{x,1}$ and $Q_{x,2}$. To be precise, the algorithm goes as follows. In $Q_{x,1}$ first extend the face of $L_x$ with normal vector $e_1$ across $Q_{x,1}$, and then in one of the resulting regions which is not yet a $3$-dimensional rectangle, extend the face of $L_x$ with normal vector $e_2$. In $Q_{x, 2}$ first extend the face of $L_x$ with normal vector $e_3$, and then in one of the resulting regions which is not yet a $3$-dimentional rectangle, extend the face of $L_x$ with normal vector $e_2$. The end result will be our regulated partition $\mathcal{Q}_{[z]}$. 

To see that $\mathcal{Q}_{[z]}$ has the required properties, we only need to analyze what happens to the troublesome point $x$ and the two regions $Q_{x,1}$ and $Q_{x,2}$ in the construction of $\mathcal{Q}_{[z]}$ from $\mathcal{P}_{[z]}$. For this let $\mathcal{Q}$ be the locally regulated partition about $x$ induced by $\mathcal{Q}_{[z]}$. Then it is clear that $|\mathcal{Q}|=5$. An exhaustive check of all the newly introduced boundary points of $\mathcal{Q}_{[z]}$ within $Q_{x,1}$ and $Q_{x,2}$ shows that the algorithm does not introduce new troublesome points for $\mathcal{Q}_{[z]}$. 

Thus we have obtained a clopen $5$-regulated partition of $F(2^{\mathbb{Z}^3})$. The proof of Theorem~\ref{thm:3dim5reg} is complete.

\section{Borel Regulated Partitions of $F(2^{\Z^n})$\label{sec:8}}

Recall from Definition~\ref{gammareg} the notion of a $\gamma$-regulated 
partition of $\R^n$, and its extension in Definition~\ref{gammaregfzn} to $\Z^n$ and in Definition~\ref{def:fzn} to 
the equivalence relation $F(2^{\Z^n})$.  Recall also Definitions~\ref{znrn1} and \ref{znrn2}
which allow us to identify rectangular partitions of $\Z^n$ with regulated partitions of  $\R^n$. 

The notion of a $\gamma$-regulated partition of $\fzn$ in Definition~\ref{def:fzn} further
generalizes in the obvious way to free actions of $\Z^n$ on $0$-dimensional Polish spaces. 
In the following we introduce a notion of nondegeneration for regulated partitions of such actions.

\begin{defn} Let $E$ be the orbit equivalence relation of a free action of $\Z^n$ on a $0$-dimensional Polish space $X$. 
\begin{enumerate}
\item[(i)] A {\em regulated partition} of $E$ is a partition $\sP$ of $X$ such that each element of $\sP$ is of the form $R\cdot x$, where $R$ is an $n$-dimensional rectangle in $\Z^n$. 
\item[(ii)] A regulated partition $\sP$ is {\em nondegenerate} if for every $x\in X$ and every $1\leq i\leq n$, there is $y\in [x]$ such that $e_i\cdot y$ and $y$ belong to the same element of $\sP$.
\end{enumerate}
\end{defn}

In practice, we always work with regulated partitions where the rectangles in the partition have positive side lengths. Such regulated partitions are all nondegenerate. 






For nondegenerate regulated partitions we can define the continuous and the Borel regulation numbers
of the action. 
\begin{defn}
Let $E$ be the orbit equivalence relation of a continuous, free action of $\Z^n$ on a $0$-dimensional
Polish space. The {\em continuous regulation number} $\gamma_c(E)$ of $E$ is the least $\gamma$
such that there is a clopen nondegenerate $\gamma$-regulated partition of $E$. Similarly, the {\em Borel 
regulation number} $\gamma_B(E)$ of $E$ is the least $\gamma$ such that there is a Borel 
nondegenerate $\gamma$-regulated partition of $E$. 
\end{defn}

Our main interest is in the spaces $\fzn$ themselves, but the arguments apply 
to more general case of  $0$-dimensional Polish spaces on which there is a free continuous action of $\Z^n$. The main question 
concerning regulated partitions is the following.

\begin{ques}
What are the values of $\gamma_c(\fzn)$ and $\gamma_B(\fzn)$? 
\end{ques}

For notational convenience we let $\gamma_c(n)$ abbreviate $\gamma_c(\fzn)$
and likewise for $\gamma_B(n)$. 
In view of our results proved in the earlier part of the paper, in particular Lemma~\ref{lem:minimalregulated} and Theorems~\ref{thm:gamma} and \ref{thm:3dim5reg}, 
we have 
$$ \gamma_c(1)=\gamma_B(1)=2, $$
$$ \gamma_c(2)=\gamma_B(2)=3, $$
$$ 4\leq \gamma_B(3)\leq \gamma_c(3)\leq 5, $$
and for all $n> 3$,
$$
n+1 \leq \gamma_B(n) \leq \gamma_c(n) \leq 3\cdot 2^{n-2}.
$$
Because of this last result, the $n+1$-regulated partitions of $\Z^n$-actions are called {\em minimal} regulated partitions.

Our goal in this section  is to improve these estimates about $\gamma_c(n)$ and $\gamma_B(n)$ for $n\geq 3$. 
The next theorem is our main theorem of this section.






\begin{thm} \label{thm:hrp}
Let $n\geq 3$ be an integer. Let $E$ be the orbit equivalence relation induced by a free Borel action of $\Z^n$ 
on a $0$-dimensional Polish space. Then there does not exist a Borel, nondegenerate, 
minimal regulated partition of $E$. 
\end{thm}

This theorem shows a stark contrast between the two-dimensional case and the cases of higher dimensions. 
Recall that for dimensions $n\leq 2$, there exist clopen nondegenerate minimal regulated 
partitions of $F(2^{\Z^n})$. However, for dimensions $n\geq 3$, not only are there no 
clopen nondengerate minimal regulated partitions for $F(2^{\Z^n})$, but there are not even any Borel ones. 
That is, for $n>2$ we have $\gamma_c(n)\geq \gamma_B(n)> n+1$. 
Combining this with Theorem~\ref{thm:gamma}, we have the following result.

\begin{thm} \label{mainineq}
For any $n\geq 3$ we have 
\[
n+2 \leq \gamma_B(n)\leq \gamma_c(n)\leq 3\cdot 2^{n-2}.
\]
\end{thm}

For $n=3$, 
we obtain the exact values of $\gamma_c(3)$ and $\gamma_B(3)$ as follows.

\begin{cor}
$\gamma_B(3)=\gamma_c(3)=5$.
\end{cor}

In the rest of this section we prove Theorem~\ref{thm:hrp}.
In the argument to follow, we will consider the case $X=F(2^{\Z^n})$ and $E$ the equivalence relation
generated by the shift action of $\Z^n$ on $F(2^{\Z^n})$. The more general case stated in Theorem~\ref{thm:hrp}
does not differ in any significant way. So, our goal is to show $\gamma_B(n)\geq n+2$. 

We will give the proof of Theorem~\ref{thm:hrp} for the case $n=3$ in Subsection~\ref{sec:7.2}, 
and then for the general case in Subsection~\ref{sec:7.3}. 
Before giving these proofs, we need to clarify the notation we use.

\subsection{Boundaries of rectangles in $\Z^n$} 
When we discuss a regulated partition $\sP$ of $F(2^{\Z^n})$, we have been focusing on each orbit $[x]$, regarding the restriction of $\sP$ on $[x]$ as a regulated partition of $\Z^n$, and furthermore identifying this regulated partition of $\Z^n$ with a regulated partition of $\R^n$ via Definition~\ref{znrn1}. In the proofs to follow, we sometimes consider the restriction of $\sP$ on $[x]$ only and do not further identify it with a regulated partition of $\R^n$. 

To avoid confusion and for complete clarity in our further discussions, we introduce the following notation. For any subset $A\subseteq F(2^{\Z^n})$, we define
$$ \partial_{\Z}A=\left\{y\in F(2^{\Z^n})\,\colon\, y\in A\mbox{ and there is $1\leq i\leq n$ such that $e_i\cdot y\not\in A$}\right\} $$
and
$$ \Int_{\Z}(A)=A\setminus \partial_{\Z}A $$
and call them the {\em $\Z$-boundary} and the {\em $\Z$-interior} of $A$, respectively. 
For a regulated partition $\sP$ of $F(2^{\Z^n})$, define
$$ \partial_{\Z}\sP=\bigcup\{\partial_{\Z}P\,\colon\, P\in\sP\} $$
and
$$ \Int_{\Z}(\sP)=F(2^{\Z^n})\setminus \partial_{\Z}\sP=\bigcup\{\Int_{\Z}(P)\,\colon\, P\in\sP\} $$
and call them the {\em $\Z$-boundary} and the {\em $\Z$-interior} of $\sP$, respectively. 
The concepts and notation apply to subsets of $\Z^n$ and regulated partitions of $\Z^n$ as well. 

We state the following easy lemma without proof.

\begin{lem}\label{lem:zr2}  Let $P, Q$ be $n$-dimensional rectangles in $\Z^n$ and let $x\in \Z^n$. Let $C_x$, $C_P$ and $C_Q$ be the
$n$-dimensional rectangles in $\R^n$ given by Definition~\ref{znrn1}. Then the following hold:
\begin{enumerate}
\item[(1)] $x\in P$ iff $C_x\subseteq C_P$;
\item[(2)] $x\in \Int_{\Z}(P)$ iff $C_x\subseteq\Int(C_P)$;  
\item[(3)] $P\subseteq \Int_{\Z}(Q)$ iff $C_P\subseteq \Int(C_Q)$.
\end{enumerate}
\end{lem}

We will also use the following notation. Let $\mathsf{S}=\{\pm e_i\,\colon\, 1\leq i\leq n\}$ be the set of standard generators of $\Z^n$. For any subset $A\subseteq F(2^{\Z^n})$ and $v\in \mathsf{S}$, define
$$ \partial_vA=\{x\in A\,\colon\, v\cdot x\not\in A\} $$
and
$$ \Int_v(A)=\{x\in A\,\colon\, v\cdot x\in A\}. $$
Thus $A=\partial_vA\cup \Int_v(A)$ for all $v\in \mathsf{S}$. For $A\subseteq F(2^{\Z^n})$ and $u,v\in\mathsf{S}$, let
$$ \partial_{u\setminus v}A=\partial_uA\setminus\partial_vA=\partial_uA\cap\Int_v(A). $$
If $\sP$ is a regulated partition of $F(2^{\Z^n})$ and $u,v\in\mathsf{S}$, then let
$$ \partial_v\sP=\bigcup\{\partial_vP\,\colon\, P\in\sP\}, $$
$$ \Int_v\sP=\bigcup\{\Int_vP\,\colon\, P\in\sP\}=F(2^{\Z^n})\setminus \partial_v\sP, $$
and
$$ \partial_{u\setminus v}\sP=\bigcup\{\partial_{u\setminus v}P\,\colon\, P\in\sP\}. $$
We state without proof the following easy property of any nondegenerate regulated partition.

\begin{lem}\label{lem:nondeg} Let $\sP$ be a nondegenerate regulated partition of $F(2^{\Z^n})$. Then for any $u, v\in\mathsf{S}$ such that $u\neq \pm v$, $\partial_{u\setminus v}\sP\neq\varnothing$.
\end{lem}






\subsection{Proof of Theorem~\ref{thm:hrp} for the $n=3$ case\label{sec:7.2}} 

In our proof we will use the notation introduced in Definition~\ref{znrn1}. That is, for a dimension $k\leq 3$ and $x\in \Z^k$, we will speak of a $k$-dimensional rectangle $C_x$ in $\R^k$. Similarly, if $P$ is a $k$-dimensional rectangle in $\Z^k$, then we will speak of a $k$-dimensional rectangle $C_P$ in $\R^k$.

We will also use the notions of a maximal surface for a regulated partition from 
Definition~\ref{def:maximalsurface}, the notion of a virtual maximal surface from Definition~\ref{def:vms}, and the properties of these notions given by Lemmas~\ref{lem:maxsurface0}, \ref{lem:vms0} and \ref{lem:maxsurface1}. Of course, these are notions in $\R^3$ as they were defined in Section~\ref{sec:4}, but we reinterpret them in the context of $\Z^3$ and $F(2^{\Z^3})$ via Definition~\ref{znrn1}. More explicitly, if $\sP$ is a regulated partition of $F(2^{\Z^3})$ and $[x]$ is a particular orbit, then $\sP$ restricted to $[x]$ gives rise to $\sQ=\{C_P\,\colon\, P\in\sP, P\subseteq [x]\}$, a regulated partition of $\R^3$ (for the coordinate system, chose $x$ as orgin). Note that all 
boundary points of $\sQ$ have coordinates in $\Z+\frac{1}{2}$. If $M$ is a maximal surface of $\sQ$ on some level $d+\frac{1}{2}$ for an integer $d$, then 
$$ (M\cap \Z^2)\times \{d\} \mbox{ and } (M\cap \Z^2)\times \{d+1\} $$
are the actual subsets of $\partial_{\Z}\sP$ restricted to $[x]$ each of which is a surface in $\Z^3$ given by a union of boundaries of $3$-dimensional rectangles in $\sP$. We thus consider $S=M\cap \Z^2$ to be the reinterpretation of the maximal surface of $M$. Now, these maximal surfaces occurs in pairs which are of perpendicular distance $1$ to each other. Namely, $S$ appears at both level $d$ and $d+1$, where $S\times\{d\}$ is the union of the upper faces of some $P\in\sP$ and $S\times\{d+1\}$ is the union of the lower faces of some $P\in\sP$. In this case, we still call $S$ a {\em maximal surface} but call $d$ the {\em lower level} and $d+1$ the {\em upper level} (of the same maximal surface $S$). By Lemma~\ref{lem:maxsurface0}, every maximal surface for a regulated partition of $\R^3$ is a (2-dimensional) rectangle. It follows that every maximal surface for a regulated partition of $\Z^3$ or $F(2^{\Z^3})$ is also a rectangle. 

Similarly, if $\sP$ is a regulated partition of $\Z^3$ or $F(2^{\Z^3})$, $[x]$ is a particular orbit, and $M$ is a virtual maximal surface from $d_1+\frac{1}{2}$ to $d_2+\frac{1}{2}$ for the regulated partition $\sQ=\{C_P\colon P\in \sP, P\subseteq[x]\}$ of $\R^3$, then $S=M\cap \Z^2$ is a {\em virtual maximal surface} from $d_1+1$ to $d_2$ for $\sP$. Here the defining property correspondent to Definition~\ref{def:vms} is 
$$ \partial_{\Z}S\times ([d_1+1, d_2]\cap \Z)\subseteq \partial_{\Z}\sP. $$
This property is nontrivial even if $d_1+1=d_2$.

We are now ready to prove Theorem~\ref{thm:hrp} for the case $n=3$. 

Assume that $\sP$ is a Borel nondegenerate $4$-regulated partition for $F(2^{\Z^3})$. 
Consider the Cohen forcing $\mathbb{P}$ for adding a generic element $x_G \in F(2^{\Z^3})$, where $G$ is a $\mathbb{P}$-generic filter. So, conditions 
$p \in \mathbb{P}$ are finite partial functions $p\,\colon R\to \{ 0,1\}$ where 
$R$ is a $3$-dimensional rectangle in $\Z^3$. Also, $q \leq_{\mathbb{P}} p$ iff $q$ extends $p$. We will get a contradiction
by considering the partition $\sP$ restricted to a generic class $[x_G]$, which is denoted $\sP_{[x_G]}$. By arbitrarily fixing an origin, we regard $\sP_{[x_G]}$ as a regulated partition of $\Z^3$.

Consider a maximal surface $S$ of $\sP_{[x_G]}$ on some lower level $d\in\Z$. 
We claim that $S$ is a {\em finite} rectangle in $\Z^2$. 
In fact, by Lemma~\ref{lem:maxsurface0}, $C_S$ is convex, and hence $S$ is a (possibly infinite) rectangle in $\Z^2$. Toward a contradiction, assume $S$ contains an infinite line. Without loss of generality, and for definiteness, 
assume $a, b\in \Z$ are such that
\begin{itemize}
\item[($\star$)] $[a, \infty)\times\{b\}\subseteq S$ and $[a, \infty)\times\{b\}\times \{d\}\subseteq \partial_{e_3}\sP$.
\end{itemize}
Since $\sP$ is nondegenerate, there is some $(i,j,k)\not\in \partial_{e_3}\sP_{[x_G]}$, that is, there is $P\in\sP_{[x_G]}$ such that $(i,j,k), (i,j,k+1)\in P$. Let $p\in G$ so that $p$ forces ($\star$) and $(i,j,k)\not\in \partial_{e_3} \sP_{[\dot{x}_G]}$. We note that the set
$$\begin{array}{rl} E=\{q\leq_{\mathbb{P}} p\,\colon & \!\!\!\!\mbox{there is a translation $\pi\,\colon\, \Z^3\to \Z^3$ such that } \\
&\!\!\!\!\mbox{$\dom(p)\cap \dom(\pi(p))=\varnothing$, $\pi(i,j,k)\in [a, \infty)\times\{b\}\times\{d\}$, } \\
& \!\!\!\!\mbox{and $q\leq_{\mathbb{P}} \pi(p)$}\} 
\end{array}$$
is dense below $p$. To see this, suppose $r\leq_{\mathbb{P}} p$. Since $\dom(r)$ is finite, there is a translation $\pi\,\colon\, \Z^3\to \Z^3$ so that $\dom(r)\cap \pi(\dom(r))=\varnothing$ and $\pi(i,j,k)\in[a, \infty)\times\{b\}\times\{d\}$. Thus $r\cup \pi(r)\in\mathbb{P}$. Let $q=r\cup\pi(r)$. Then $q\leq_{\mathbb{P}}r$ and $q\in E$. This shows that $E$ is dense below $p$. Now since $G$ is $\mathbb{P}$-generic and $p\in G$, there is $q\in G\cap E$. 
But then $q$ forces $[a, \infty)\times\{b\}\times\{d\}\subseteq\partial_{e_3}\sP_{[\dot{x}_G]}$ since $q\leq_{\mathbb{P}} p$, and $q$ forces $\pi(i,j,k)\not\in \partial_{e_3}\sP_{[\dot{x}_G]}$ since $q\leq_{\mathbb{P}} \pi(p)$. Thus
$\pi(i,j,k)\in  [a,\infty)\times\{b\}\times\{d\}\subseteq\partial_{e_3}\sP_{[x_G]}$ and $\pi(i,j,k)\not\in \partial_{e_3}\sP_{[x_G]}$ at the same time, a contradiction. This shows that $S$ is a finite rectangle in $\Z^2$.



Next we analyze the virtual maximal surfaces in $\sP_{[x_G]}$. A density argument as in the previous paragraph shows that for any virtual maximal surface $S$ for $\sP_{[x_G]}$ from some $d_1+1$ to $d_2$ with $d_1<d_2$, there is a largest finite such $d_2$. Now suppose $S_0$ is a maximal surface of $\sP_{[x_G]}$ on (upper or lower) level $\delta_0$.  Then by Lemma~\ref{lem:vms0} and the above observation, there is a finite $\delta_1>\delta_0$ such that $S_0$ is a virtual maximal surface from $\delta_0+1$ to $\delta_1$, where $\delta_1$ is the largest such. By Lemma~\ref{lem:maxsurface1}, there is $\delta_2>\delta_1$ and a virtual maximal surface $S_1$ from $\delta_1+1$ to $\delta_2$ such that $S_0\subsetneq S_1$. Moreover, there is a largest finite such $\delta_2$. By repeating this construction and argument, we obtain an infinite increasing sequence
$$ \delta_0<\delta_1<\delta_2<\cdots <\delta_k<\cdots $$
and virtual maximal surfaces $S_k$ from $\delta_k+1$ to $\delta_{k+1}$, where each $\delta_{k+1}$ is the largest such, with $S_k\subsetneq S_{k+1}$ for all $k\geq 0$. A density argument as in the previous paragraph shows that for any $k\geq 0$ and any side $L$ of $S_k$, there is $K>k$ such that $L\subseteq \Int_{\Z}(S_K)$. It thus follows that for any $k\geq 0$, there is $K>k$ such that $S_k\subseteq \Int_{\Z}(S_K)$. Therefore, we get that for any $D>0$, there is a virtual maximal surface $S$ of $\sP_{[x_G]}$ all of whose side lengths are at least $D$.

Let $\mathsf{S}=\{\pm e_1, \pm e_2, \pm e_3\}$ be the set of standard generators of $\Z^3$. We define a condition $p_0 \in\mathbb{P}$ so that $\dom(p_0)=[-s,s]^3\cap \Z^3$ for some positive integer $s$, and $p_0$ forces that
\begin{itemize}
\item[($\filledstar$)] for any $u\in \mathsf{S}$, $v\in\mathsf{S}\setminus\{\pm u\}$, and $w\in\mathsf{S}\setminus\{\pm u, \pm v\}$, there is $a\in [-s,s]\cap \Z$ such that $aw\in\partial_{u\setminus v}\sP_{[\dot{x}_G]}$.
\end{itemize}
To build this condition $p_0$, consider 
$$ \mathsf{T}=\{(u,v,w)\in\mathsf{S}^3\,\colon\, v\neq \pm u \mbox{ and } w\neq \pm u, \pm v\}. $$
Fix a triple $(u,v,w)\in\mathsf{T}$. By Lemma~\ref{lem:nondeg}, there is $y_{u,v}\in [x_G]$ such that $y_{u,v}\in \partial_{u\setminus v}\sP_{[x_G]}$, and therefore there is $q_{u,v}\in \mathbb{P}$ which forces this. Since $\dom(q_{u,v})$ is finite, there is $a_{u, v, w}\in\Z$ such that a translate of $q_{u,v}$, $\pi_{u,v,w}(q_{u,v})$, forces that $a_{u,v,w}w\in\partial_{u\setminus v}\sP_{[\dot{x}_G]}$. In fact, we may find these $a_{u,v,w}, q_{u,v}, \pi_{u,v,w}$ such that for all distinct triples $(u,v,w), (u',v',w')\in \mathsf{T}$, 
$$\dom(\pi_{u,v,w}(q_{u,v}))\cap \dom(\pi_{u',v',w'}(q_{u',v'}))=\varnothing.$$ Now let $s$ be large enough such that $a_{u,v,w}\in [-s, s]$ and $\dom(\pi_{u,v,w}(q_{u,v}))\subseteq [-s,s]^3$ for all triples $(u,v,w)\in \mathsf{T}$, and let $p_0$ have domain $[-s,s]^3\cap \Z^3$ and extend all $\pi_{u,v,w}(q_{u,v})$ for triples $(u,v,w)\in\mathsf{T}$. Such $p_0$ is as required. 

We give an informal description of some motivation behind the construction of this $p_0$. Consider the case where $u=+e_3$. Corresponding to the four cases of $v\in\{\pm e_1, \pm e_2\}$, $p_0$ forces the existence of four points which are all on the upper faces of some rectangles in $\sP_{[x_G]}$. For each of the four directions $v\in \{\pm e_1, \pm e_2\}$, $p_0$ forces the existence of a point in $\partial_{+e_3\setminus v}\sP_{[x_G]}$. For example, for $v=+e_1$, $p_0$ forces the existence of a point $(0,a,0)\in \partial_{+e_3\setminus +e_1}\sP_{[x_G]}$, i.e., there is some $P\in\sP_{[x_G]}$ such that $(0,a,0), (1,a,0)\in P$ and $(0,a,1)\not\in P$. Such a design is to forbid a virtual maximal surface to occur from $d_1$ to $d_2$ with $d_1\leq 0\leq d_2$ and with either $[-s,s]\times \{0\}$ or $\{0\}\times [-s,s]$ as a part of its boundary. 


Next we describe a partial tiling $\sT_2$ of the $2$-dimensional rectangle 
$$R=\left[-s-(2s+1)^3, s+(2s+1)^3\right]^2\cap \Z^2$$ by tiles of dimensions $(2s+1)\times (2s+1)$. Each tile is a translate of $[-s,s]^2\cap\Z^2$, and hence we describe its position by only specifying the position of its center point. Let $A=(2s+1)^2$ and $B=2s$. We place a copy of the tile at each of the following positions:
$$ \big(\alpha(2s+1), \alpha+\beta A\big),\ \big(-\alpha-\beta A, \alpha(2s+1)\big),\ 
\big(-\alpha(2s+1), -\alpha-\beta A\big),\ \big(\alpha+\beta A, -\alpha(2s+1)\big) $$
for 
$$ (\alpha, \beta)\in \big([0, A-1]\times [0, B]\big)\cup \big\{ (A, B)\big\}. $$
Figure~\ref{fig:nrb} illustrates the partial tiling $\sT_2$ for $s=1$. The dotted square on the outside is
$$ R_0=\left[-(2s+1)^3, (2s+1)^3\right]^2\cap \Z^2.$$ 
We use dotted lines to further divide $R_0$ into four quadrants I, II, III and IV, which are illustrated in the figure. The construction in each quadrant can be rotated about the origin by $90^\circ$ to obtain the construction for the next quadrant.

\begin{figure}[ht]
\centering
\begin{tikzpicture}[scale=0.03]
\pgfmathsetmacro{\s}{3};

\foreach \j in {0,...,2}
{
\foreach \i in {0,...,8}
{\draw (9*\i, 3*\i+27*\j) rectangle (9*\i+9, 3*\i+9+27*\j);
\draw (-3*\i-27*\j+9, 9*\i) rectangle (-3*\i-27*\j, 9*\i+9);
\draw (-9*\i+9, -3*\i-27*\j+9) rectangle (-9*\i, -3*\i-27*\j);
\draw (3*\i+27*\j, -9*\i+9) rectangle (3*\i+27*\j+9, -9*\i);}
}

\draw (81,81) rectangle (90,90);
\draw (-72,81) rectangle (-81,90);
\draw (-72,-72) rectangle (-81,-81);
\draw (81,-72) rectangle (90,-81);
\draw[dotted] (-77.5,-77.5) rectangle (84.5,84.5);
\draw[dotted] (-77.5,4.5) -- (84.5,4.5);
\draw[dotted] (4.5,-77.5) -- (4.5,84.5);

\node at (80, 68) {I};
\node at (-59, 70) {II};
\node at (-65, -59) {III};
\node at (67, -60) {IV};



\end{tikzpicture}
\caption{The partial tiling $\sT_2$; also a top view of $\sT_3$.} \label{fig:nrb}
\end{figure}
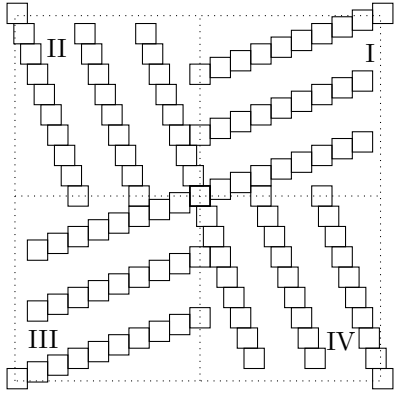

We note the following properties of the partial tiling $\sT_2$, all of which can be verified by a straightforward computation and observation. 
\begin{enumerate}
\item[(a)] Intersecting tiles coincide, so distinct tiles are disjoint from each other. 
\item[(b)] For any $t\in \left[-(2s+1)^3, (2s+1)^3\right]\cap \Z$, there is a tile whose center point has $e_1$-coordinate $t$ and there is a tile whose center point has $e_2$-coordinate $t$. 
\item[(c)] The partial tiling $\sT_2$ contains four corner tiles whose center points form a square of side length $2(2s+1)^3$.
\item[(d)] If two copies of the partial tiling $\sT_2$ are placed side by side (in either a left-right fashion or an top-down fashion), with the corner tiles coinciding with each other, then in the resulting partial tiling the distinct tiles are still disjoint from each other. Consequently, the partial tiling $\sT_2$ can be thus repeatedly applied to partially tile any two-dimensional rectangle in $\Z^2$ whose side lengths are of the form $2s+m\cdot 2(2s+1)^3$ for an integer $m$.
\end{enumerate}


Now we consider a $3$-dimensional partial tiling $\sT_3$ by $3$-dimensional tiles of dimensions $(2s+1)\times (2s+1)\times (2s+1)$. For a $2$-dimensional rectangle $R$ of side lengths $2(2s+1)^3+2s$ in $\Z^3$, say
$$ R=\left[-s-(2s+1)^3, s+(2s+1)^3\right]\times \left[-s-(2s+1)^3, s+(2s+1)^3\right]\times \{0\}, $$
we obtain $\sT_3$ by placing a $3$-dimensional tile at $(i,j,0)$ iff $(i,j)$ is the center point of a tile in $\sT_2$. Thus Figure~\ref{fig:nrb} can be regarded as the top view of $\sT_3$. Properties (a)--(d) for $\sT_2$ continue to hold for $\sT_3$. We observe the following additional property.
\begin{enumerate}
\item[(e)] There is a way to place a copy of $\sT_3$ on each face of a $3$-dimensional rectangle of side lengths $2(2s+1)^3+2s$, with the corner tiles coinciding with each other, so that in the resulting partial tiling the distinct tiles are still disjoint from each other. Consequently, the partial tiling $\sT_3$ can be thus repeatedly applied to partially tile any $3$-dimensional rectangle in $\Z^3$ whose side lengths are of the form $2s+m\cdot 2(2s+1)^3$.
\end{enumerate}
The algorithm to place the copies of $\sT_3$ on such a $3$-dimensional rectangle is to ``roll down" the tiling patterns from one face to an adjacent face. Figure~\ref{fig:rolldown} illustrates the result of this algorithm. Figure~\ref{fig:rolldown} only shows the front view; the back view and the side views are not shown.

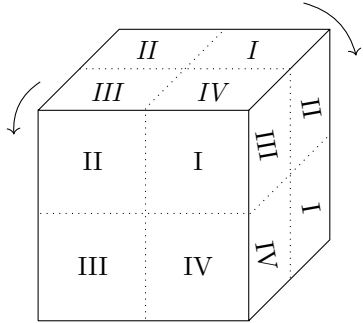
\begin{figure}[h]
\begin{tikzpicture}[scale=0.7]
\oppapp{4}{4}{4}{0}{0}{0};
\draw[dotted] (0.5,-1.5) -- (0.5,2.5) -- (2, 4);
\draw[dotted] (-1.5, 0.5) -- (2.5,0.5) -- (4,2);
\draw[dotted] (-0.75, 3.25) -- (3.25,3.25) -- (3.25,-0.75);
\node at (-0.5,-0.5) {III};
\node at (1.5,-0.5) {IV};
\node at (-0.5, 1.5) {II};
\node at (1.5, 1.5) {I};
\node at (-0.25, 2.8) {{\it III}};
\node at (0.5, 3.6) {{\it II}}; 
\node at (1.75, 2.8) {{\it IV}};
\node at (2.5, 3.6) {{\it I}};
\node at (2.8, -0.3) {\rotatebox{280}{IV}};
\node at (2.8,1.8) {\rotatebox{280}{III}};
\node at (3.65, 0.6) {\rotatebox{280}{I}};
\node at (3.65, 2.5) {\rotatebox{280}{II}};
\draw[->] (3.5,4.5) to[bend left] (4.5,3.5);
\draw[->] (-1.5,3) to[bend right] (-2,2);

\end{tikzpicture}
\caption{The ``roll down" algorithm to partially tile the faces of a $3$-dimensional rectangle. \label{fig:rolldown}}
\end{figure}

We are finally ready to construct a condition $q\in\mathbb{P}$ to derive a contradiction. 

Let $D=4(2s+1)^3$. Since there is a virtual maximal surface of $\sP_{[x_G]}$ all of whose side lengths are at least $D$, we obtain a condition $p\in\mathbb{P}$ with $\dom(p)=[-t, t]^3\cap \Z^3$ for some positive integer $t$ such that $p$ forces the existence of a virtual maximal surface $S$ of $\sP_{[\dot{x}_G]}$ from some $d_1<0$ to $d_2=0$,  where $d_2$ is the largest such, so that the side lengths of $S$ are at least $D$, and $S\times [d_1+1, d_2] \subseteq [-t,t]^3$. Without loss of generality, we may assume $t$ is a postive multiple of $D=4(2s+1)^3$. 

Let 
$$ E=[-2s-2t, 2t+2s]^3\cap \Z^3 $$
and
$$ E_0=[-2t, 2t]^3\cap\Z^3. $$
Partially tile the faces of $E$ by repeatedly applying $\sT_3$ as indicated above, and at the position of each tile place a copy (translate) of the condition $p_0$. Let $q\in\mathbb{P}$ be the extension of $p$ by this construction.


Let $x_G$ be a generic element extending $q$. Then $\sP_{[x_G]}$ contains a virtual maximal surface $S_0\subseteq [-t, t]^2$ from $d_1<0$ to $d_2=0$, where $d_2$ is the largest such, so that $S_0$ has side lengths at least $D$. By applying Lemma~\ref{lem:maxsurface1} repeatedly, we obtain an increasing sequence
$$ \delta_0=0<\delta_1<\cdots< \delta_k<\cdots $$
and virtual maximal surfaces $S_k$ from $\delta_k+1$ to $\delta_{k+1}$, where $\delta_{k+1}$ is the largest such, with $S_k\subsetneq S_{k+1}$ for all $k\geq 0$. Let $K$ be the unique number such that $\delta_K< 2t\leq \delta_{K+1}$.

We consider two cases. 

Case 1: $S_K\subseteq [-2t,2t]^2$. Denote a maximal line segment of $\partial_{\Z}S_K$ by $\ell$. Since $S_K$ is a virtual maximal surface from $\delta_K+1$ to $\delta_{K+1}$, there is $v\in\{\pm e_1, \pm e_2\}$ such that 
$$ \ell\times \{2t\}\subseteq \partial_v\sP_{[x_G]}. $$ 
On the other hand, since $S_0\subseteq S_K$, the length of $\ell$ is at least $D=4(2s+1)^3$. Consider $w\in \{\pm e_1, \pm e_2\}$ such that $w\neq \pm v$. Then $w$ is parallel to $\ell$. 
Since the length of $\ell$ is at least $4(2s+1)^3$, by our construction of $\sT_3$, $\ell\times\{2t\}$ contains a line segment between $C-sw$ and $C+sw$, where $C$ is on the top of $E_0$ and is the center position of a copy of $p_0$ in the construction of $\sT_3$. By the property of $p_0$, there must be some $a\in [-s, s]\cap \Z$ such that $C+aw\in \partial_{+e_3\setminus v}\sP_{[x_G]}$. However,  $C+aw\in \ell$, which is a contradiction.

Case 2: $S_K\not\subseteq [-2t, 2t]^2$.  In this case there is a unique $k<K$ such that $S_k\subseteq [-2t, 2t]^2$ but $S_{k+1}\not\subseteq [-2t, 2t]^2$. We have $\delta_{k+1}\leq \delta_K<2t$. Since $S_0\subseteq S_k\subsetneq S_{k+1}$, the side lengths of $S_k$ and $S_{k+1}$ are all at least $D$. It follows that
$$ (S_{k+1}\times \{\delta_{k+1}+1\})\cap \partial_{\Z}E_0 $$
contains a horizontal line segment of length at least $D$. We denote such a line segment by $\ell$ and let $w\in \{\pm e_1, \pm e_2\}$ be parallel to $\ell$. Let $u\in \{\pm e_1, \pm e_2\}$ be such that $u\neq \pm w$. By Lemma~\ref{lem:maxsurface1} (iic), $\ell\times\{\delta_{k+1}+1\}$ lies on the upper level of a maximal surface. Thus we have
$$ \ell\times\{\delta_{k+1}+1\}\subseteq \partial_{-e_3}\sP_{[x_G]}. $$
Since the length of $\ell$ is at least $D=4(2s+1)^3$, by our construction of $\sT_3$, $\ell\times\{\delta_{k+1}+1\}$ contains a line segment between $C-sw$ and $C+sw$, where $C$ is on a vertical side of $E_0$ and is the center position of a copy of $p_0$ in the construction of $\sT_3$. By the property of $p_0$, there must be some $a\in [-s, s]\cap \Z$ such that $C+aw\in \partial_{u\setminus -e_3}\sP_{[x_G]}$. Again, since $C+aw\in \ell$, this is a contradiction.


This completes the proof of Theorem~\ref{thm:hrp} in dimension $3$.

\subsection{Proof of Theorem~\ref{thm:hrp} for higher dimensions\label{sec:7.3}} This entire subsection is devoted to a proof of Theorem~\ref{thm:hrp} for the case $n>3$. Let $n>3$ be fixed and assume toward a contradiction that $\sP$ is a Borel, nondegenerate, minimal regulated partition of $F(2^{\Z^n})$. 

Let $\mathbb{P}$ be the Cohen forcing for adding a generic element $x_G\in F(2^{\Z^n})$, where $G$ is a $\mathbb{P}$-generic filter. So, conditions $p\in \mathbb{P}$ are finite partial functions $p\colon R\to \{0,1\}$ where $R$ is an $n$-dimensional rectangle in $\Z^n$; for $p, q\in \mathbb{P}$, $q\leq_{\mathbb{P}}p$ iff $q$ extends $p$. 

We still denote the generic class $\mathbb{Z}^n\cdot x_G$ by $[x_G]$. But we also consider an action $\alpha$ of $\mathbb{Z}^3$ on $F(2^{\Z^n})$ in the generic extension which is defined by
$$ (a_1, a_2, a_3)\cdot_\alpha y=(a_1, a_2, a_3, 0, \dots, 0)\cdot y. $$
It is clear that the action $\alpha$ is free and Borel. For any $y\in F(2^{\Z^n})$, let $[y]_\alpha=\mathbb{Z}^3\cdot_\alpha y$. 

Now we define a regulated partition $\sQ$ of $F(2^{\Z^n})$ for the action $\alpha$ by
$$ Q\in \sQ\iff \exists y\in Q\ \exists P\in \sP\ Q=P\cap [y]_\alpha. $$

It is clear that $\sQ$ is a Borel regulated partition of $F(2^{\Z^n})$ for the action $\alpha$. We verify that it is minimal. For this, consider $\sQ$ on an aribtrary $[y]_\alpha$. Then $\sQ$ restricted to $[y]_\alpha$ gives rise to a regulated partition $\mathcal{W}=\{C_Q\,\colon\, Q\in\sQ, Q\subseteq [y]_\alpha\}$ of $\R^3$. Meanwhile, note that $\sP$ restricted on $[y]$ also gives rise to a regulated partition $\mathcal{V}=\{C_P\,\colon\, P\in \sP, P\subseteq [y]\}$ of $\R^n$. In particular, note that $y$ is not on the boundary of any $V\in\mathcal{V}$. Let $\mbox{\rm proj}$ be the projection of $\R^n$ onto its first three coordinates. Then 
$$ \mathcal{W}=\left\{\mbox{\rm proj}(V)\,\colon\, V\in \mathcal{V}, V\cap [y]\neq\varnothing\right\}. $$
Now let $z$ be an arbitrary element of $\R^3$, and let $W_1, \dots, W_k$ enumerate all elements $W\in \mathcal{W}$ with $z\in W$. For each $W_j$, there is a unique $V_j\in \mathcal{V}$ with $V_j\cap [y]\neq\varnothing$ and $\mbox{\rm proj}(V_j)=W_j$. This implies that $\nu_{\mathcal{W}}(z)=\nu_{\mathcal{V}}(z)$. On the other hand, $z$ is on the boundary of $W_j$ iff $z$ is on the boundary of $V_j$, and if $z$ is on a face $F$ of $V_j$, then the normal vector of $F$ can only be $\pm e_i$ for $i=1,2,3$. This implies that $\beta_{\mathcal{W}}(z)=\beta_{\mathcal{V}}(z)$. Since $\sP$ is minimal, we have that $\mathcal{V}$ is minimal and $\nu_{\mathcal{V}}(z)=\beta_{\mathcal{V}}(z)+1$. It follows that $\nu_{\mathcal{W}}(z)=\beta_{\mathcal{W}}(z)+1$, and therefore $\mathcal{W}$ is minimal. Thus $\sQ$ is minimal.

Now we focus on the class $[x_G]_\alpha$ and derive a contradiction by repeating the proof of the $n=3$ case. By examining the proof one sees that it suffices to verify the nondegeneracy hypothesis for $\sQ$ restricted to $[x_G]_\alpha$, which we now do. Suppose $1\leq i\leq 3$. By the nondegeneracy of $\sP$, there is $z\in [x_G]$ such that $e_i\cdot z$ and $z$ belong to the same element of $\sP$. Let $p\in G$ so that $p$ forces this statement; in fact, we get $g\in \Z^n$ such that $p$ forces that $(g+e_i)\cdot \dot{x}_G$ and $g\cdot\dot{x}_G$ belong to the same element of $\sP$. Equivalently, $p$ forces that $g\cdot \dot{x}_G\not\in \partial_{e_i}\sP$. Let $\theta$ be any translation of $\Z^n$ so that for any $3<j\leq n$, $\pi_j(\pi(g))=0$. Then $\theta(p)$ forces that there is $z\in [\dot{x}_G]_\alpha$ such that $z\not\in\partial_{e_i}\sP$. It follows that $\theta(p)$ forces that there is $z\in [\dot{x}_G]_\alpha$ such that $z\not\in \partial_{e_i}\sQ$. This argument gives the density of the set of elements $r\in \mathbb{P}$ such that $r$ forces that there is $z\in [\dot{x}_G]_\alpha$ with $z\not\in \partial_{e_i}\sQ$. Hence we get that there is $z\in [x_G]_\alpha$ such that $z\not\in \partial_{e_i}\sQ$, as needed.

\section{More Remarks on Higher Dimensions\label{sec:9}}

In our proof of Theorem~\ref{thm:hrp} for dimensions $n>3$, we used a direct lift-up of the proof from
the $n=3$ case. In particular, we did not need to establish the two key lemmas, 
Lemmas~\ref{lem:maxsurface0} and \ref{lem:maxsurface1} for $n>3$. However,
it is of interest that these two combinatorial lemmas about minimal regulated partitions
do hold in all dimensions $n$. They also can be used to give an alternate proof of 
Theorem~\ref{thm:hrp} which we will briefly sketch later in this section. Since these lemmas are of possible independent interest, we 
give their proofs in general.

Definition~\ref{def:maximalsurface} of a maximal surface and Definition~\ref{def:vms} of a virtual maximal surface 
extend to locally regulated partitions of $\R^n$. 
As in Definition~\ref{def:maximalsurface} we define the notion of {\em{level}} as there using now 
$\pi_n$ instead of $\pi_3$. We then define $\sP_{d,-}$ and $\sP_{d,+}$ and then 
define the notion of a {\em maximal surface} as there. The notion of a rectangle $M$ in $\R^{n-1}$ 
being a {\em virtual maximal surface }from $d_1$ to $d_2$ is also defined as in Definition~\ref{def:vms}, namely
$\partial M\times [d_1,d_2]\subseteq \bigcup\partial\sP=\bigcup_{P\in\sP} \partial P$.
Similarly, the notion of $\sP$ respecting $M$ from $d_1$ to $d_2$ is 
defined as in Definition~\ref{def:vms}.

We first extend Lemma~\ref{lem:maxsurface0} concerning maximal surfaces. 
We use the following lemma.

\begin{lem}\label{lem:kdimbasic}
Let $k\geq 1$. Let $M$ be a rectangular polyhedron in $\R^k$, i.e., $M$ is homeomorphic to a $k$-dimensional rectangle and it is the union of finitely many $k$-dimensional rectangles. Suppose that for every $x \in \partial M$ and for all sufficiently small $k$-dimensional rectangles $R$ centered at $x$, 
$R\cap M$ is a $k$-dimensional rectangle. Then $M$ is a $k$-dimensional rectangle. 
\end{lem}

\begin{proof}
The assumption implies that for every point $x \in M$ (not just those in $\partial M$) and every small enough rectangle $R$
centered about $x$, $R\cap M$ is a rectangle. It follows that for every 2-dimensional
plane $Q$ which is a translation of the span of $\{e_i,e_j\}$ for some distinct $e_i,e_j$,
we have that $Q\cap M$ also has the property that for all $x\in Q\cap \partial M$ and all 
small enough rectangles $R$ centered at $x$, $Q\cap M\cap R$ is a 2-dimensional rectangle. 
So, $Q\cap M$ is a 2-dimensional rectangular polyhedron with the property that for all 
$x \in Q\cap \partial M$ and all small enough $R$ about $x$, $Q\cap M\cap R$ is a 2-dimensional rectangle. 
This easily imples that $Q\cap M$ is convex and hence a 2-dimensional rectangle.

Now we show by induction on $k$ that if $M$ is a $k$-dimensional rectangular
polyhedron  with the property that  $Q\cap M$ is a 2-dimensional rectangle for any 2-dimensional plane $Q$ as above,
 then $M$ is a $k$-dimensional rectangle. By the induction hypothesis for $k-1$, we have that 
for every hyperplane $H$ perpendicular to some $e_i$, $H\cap M$ is a $(k-1)$-dimensional rectangle. Consider a hyperplane $H$ perpendicular to $e_k$ 
which intersects $M$. Let $R=H\cap M$. So $R$ is a $(k-1)$-dimensional rectangle. 
If $M$ is not a $k$-dimensional rectangle, then there must be two points $x_1, x_2\in R$ such that the line $L_1$ through $x_1$ parallel to $e_k$
and the line $L_2$ through $x_2$ parallel to $e_k$ have different intersections with $M$.  Let $y_0=x_1, y_1, \dots, y_\ell=x_2$ be a sequence of points in $R$ where for each $j\leq \ell-1$, $y_j$ and $y_{j+1}$ differ at exactly one coordinate. There must be some $j\leq \ell-1$ such that the line through $y_j$ parallel to $e_k$ and the line through $y_{j+1}$ parallel to $e_k$ have different intersections with $M$. In other words, there are $z_1, z_2\in R$ and $i_0\in \{1,\dots, k-1\}$ such that $\pi_{i_0}(z_1)\neq \pi_{i_0}(z_2)$, $\pi_i(z_1)=\pi_i(z_2)$ for all $i \in \{ 1,\dots,k-1\}\setminus\{i_0\}$, and the line through $z_1$ parallel to $e_k$
and the line through $z_2$ parallel to $e_k$ have different intersections with $M$.
But then the 2-dimensional plane $Q$ parallel to the span of $\{ e_{i_0},e_n\}$ 
through $z_1, z_2$ has an intersection with $M$ which is not a 2-dimensional rectangle, a contradiction.
\end{proof}

\begin{thm} \label{thm:maxsurface0}
Let $n\geq 3$. Let $\sP$ be a minimal regulated partition of $\R^n$. 
Then every maximal surface for $\sP$ is an $(n-1)$-dimensional rectangle. 
\end{thm}

\begin{proof}
Let $M$ be a maximal surface for $\sP$, say at level $d$. 
Clearly $M$ is an $(n-1)$-dimensional rectangular polyhedron. By Lemma~\ref{lem:kdimbasic}, it suffices to show 
that for every $x\in \partial M$ and every small enough $(n-1)$-dimensional rectangle $R$ about $x$, 
$R\cap M$ is an $(n-1)$-dimensional rectangle. For this, it suffices to show that every small enough $n$-dimensional rectangle
$R$ about $(x,d)$ has the property that $R\cap (M\times\{d\})$ is an $(n-1)$-dimensional rectangle. 

Consider a sufficiently
small $n$-dimensional rectangle $R$ about $(x,d)$. Then $ \sP_R=\{P\cap R\,\colon P\in \sP\}$
is a minimal locally regulated partition. Recall Theorem~\ref{thm:minimalrep} which 
gives an analysis of any minimal locally regulated partition.
Using the notation of that theorem, we obtain distinct $i_1,\dots,i_k\in\{1,\dots, n\}$ and a labelled full simple tree $T$ of $k+1$ levels 
such that $\sP_R$ is given by $\sP(i_1,\dots,i_k;T)$. Since $x \in \partial M$ and $M$ is a maximal surface, there is $1\leq m\leq k$ such that $n=i_m$. Then from the definition of  $\sP(i_1,\dots,i_k,T)$ we have that at 
step $m$ in the construction of  $\sP(i_1,\dots,i_k,T)$, letting $H_n$ be the hyperplane perpendicular
to $e_{i_m}=e_n$, $H_n \cap M$ does not cross any of the previous hyperplanes through
$(x,d)$, namely the hyperplanes perpendicular to $e_{i_1},\dots,e_{i_{m-1}}$. It follows that 
$H_n \cap M$ is an $(n-1)$-dimensional rectangle. 
\end{proof}

We now have the following generalization of Lemma~\ref{lem:maxsurface1}. 

\begin{thm}\label{thm:maxsurface1} 
Let $R=S\times[a_n,b_n]$ be an $n$-dimensional rectangle in $\R^n$ and 
let $\sP$ be a minimal locally regulated partition of $R$. Let $M_1$ be a virtual
maximal surface from $d_1$ to $d_2$ with $a_n\leq d_1<d_2<b_n$ and $d_2$ the largest such. 
Then at least one of the following holds:
\begin{enumerate}
\item[(i)] there is a unique maximal surface $M_2$ on level $d_2$ such that $M_1\subsetneq M_2$;
\item[(ii)] there is $d_3>d_2$ and a virtual maximal surface $M_2$ from $d_2$ to $d_3$ such that 
\begin{itemize}
\item[(iia)] $M_1\subsetneq M_2$, 
\item[(iib)] $\pi_i(M_1) \neq \pi_i(M_2)$ for exactly one $i\in \{1,\cdots,n-1\}$, 
\item[(iic)] for any $\vec a=(a_1,\dots a_{n-1})\in M_2\setminus M_1$, 
$\vec a$ lies on a maximal surface $N$ on the level $d_2$ so that $\Int(N)\cap \Int(M_1)\neq\varnothing$.
\end{itemize}
\end{enumerate} 
\end{thm}


\begin{proof}
Let $M_1=[a_1,b_1]\times \cdots \times [a_{n-1},b_{n-1}]$ and $d_2>d_1$ be as in the statement 
and $P\in \sP$ which does not respect 
$M_1 \times [d_1, b_n]$ and which has a face $F$ with $\pi_n(F)=\{ d_2\}$. 
Let $N$ be the maximal surface of $\sP$ such that $N\times \{ d_2\}$ contains $F$. 
Say $N=[a'_1,b'_1]\times \cdots \times [a'_{n-1},b'_{n-1}]$. 

Let $A=\{1,\dots, n-1\}$ and $F'=\pi_{A}(F)$, i.e., $F=F'\times\{d_2\}$. Then $F'$ is an $(n-1)$-dimensional rectangle and 
$\Int(F') \cap \partial M_1 \neq \emptyset$ by assumption. Since $N\times \{ d_2\}$ extends $F$ which does not respect $M_1 \times [d_1,b_n]$, there is a coordinate $1\leq i\leq n-1$ such that either $a'_i<a_i<b'_i$, or $a'_i<b_i<b'_i$, or both. For definiteness, and without loss of generality, we assume  $a'_{n-1}< a_{n-1}< b'_{n-1}$.
We may also assume that there is a coordinate, say $1\leq m\leq n-1$, such that $[a'_m,b'_m]$ does not contain 
$[a_m,b_m]$, as otherwise $N$ properly contains $M_1$ which implies the result. 
We have either $a_m<a'_m<b_m$, or $a_m<b'_m<b_m$, or both. Without loss of generality assume $a_{m}< b'_{m}<b_{m}$.



First assume $m \neq n-1$. 
Let $W$ be the part of the hyperplane $\pi_{n-1}(x)=a_{n-1}$ which is below 
the $\pi_n(x)=d_2$ hyperplane, i.e., $W=\{x\in R\,\colon \pi_{n-1}(x)=a_{n-1}, \pi_n(x)\leq d_2\}$. Let $F$ be as above, and $G$ the face of $P$ 
on the hyperplane $\pi_{m}(x)=b'_{m}$. Let $z$ be a point on the 
intersection of $W$, $F$, and $G$, which easily exists. In three dimensions,
the point $z$ is illustrated as the point $x$ in Figure~\ref{fig:4}. 
Let $R$ be a small $n$-dimensional rectangle in $\R^n$ about the point $z$. Because $\sP$ is 
an $(n+1)$-regulated partition of $\R^n$, the analysis of Theorem~\ref{thm:minimalrep}
applies. Consider the ordering $<$ of the three surfaces $W$, $F$, $G$ in 
Theorem~\ref{thm:minimalrep}. Since $F$ divides $W$ we must have $F<W$. 
Since $W$ divides $G$, we have $W<G$. 
So, $F<W<G$ and thus $F<G$. But then $F$ must divide $G$ which however it does not 
as the maximal surface $M_2$ has only points $x$ with $\pi_{m}(x)\leq \pi_{m}(z)$. 
So, this case cannot occur.

Assume now that $m=n-1$. Let $G$ be the hyperplane $\pi_{n-1}(x)=b'_{n-1}$
and let $F$ again be the hyperplane $\pi_n(x)=d_2$ (there is a slight abuse of the notation here, as $F$ was a face of $P$, but here we let it denote the hyperplane on which the face lies).
Suppose first that there is an $\ell\in \{ 1,\dots,n\}\sm \{ n-1,n\}$
such that $\pi_\ell(N) \neq \pi_\ell(M_1)$. Since $\ell \neq m= n-1$, we are 
in the case $\pi_\ell(N) \supseteq \pi_\ell(M_1)$ and thus $\pi_\ell(N)\supsetneq
\pi_\ell(M_1)$. Without loss of generality, we may assume $a'_\ell<b_\ell<b'_\ell$. 
Now let $W$ again be the part of the hyperplane $\pi_\ell(x)=b_\ell$ which is below the hyperplane $F$.
Fix a point 
$z\in F\cap G\cap W$. Again consider a sufficiently small $n$-dimensional rectangle $R$ in $\mathbb{R}^n$ about $z$, and apply Theorem~\ref{thm:minimalrep}.
It now follows that $F$ divides $W$, so $F<W$ as in Theorem~\ref{thm:minimalrep}.
We also have $W$ divides $G$ as $a_{n-1}<b'_{n-1}<b_{n-1}$ and $m=n-1$. So, $W<G$ and consequently
$F<G$, that is $F$ divides $G$,  which is a contradiction as before.

The remaining case is when $\pi_i(N)=\pi_i(M_1)$ for all $i \neq n-1$.
Suppose $N'$ is another maximal surface at level $d_2$ with $\Int(N') \cap \partial M_1\neq \emptyset$.
The previous proof shows that there is at most one $j \in \{ 1,\dots,n-1\}$ such
that $\pi_j(N')\neq \pi_j(M_1)$. If such a $j$ exists and $j \neq n-1$, then 
both $\Int(N \cap N')\neq \emptyset$ and $\Int(N\setminus N')\neq\emptyset$, and we have a contradiction to the maximality of $N'$. So, for any such $N'$ we have 
$\pi_j(N')=\pi_j(M_1)$ for all $j \neq n-1$. Let $N''$ be the union 
of all the maximal surfaces $N'$ at level $d_2$ such that $\Int(N')\cap \partial M_1 \neq \emptyset$. 
We have shown that $\pi_j(N'')=\pi_j(M_1)$ for all $j \neq n-1$. Let $M_2=M_1\cup N''$. Then $M_2$ is an $(n-1)$-dimensional rectangle. By the definition of $M_2$, $\sP$ respects $M_2$ from $d_2$ to some $d_3>d_2$. Thus $M_2$ is a virtual maximal surface satisfying (iia)--(iic). 
\end{proof}

Using Theorems~\ref{thm:maxsurface0} and \ref{thm:maxsurface1} we now prove the following result 
about minimal regulated partitions of $\R^n$, which may be of independent interest. 
Afterwards, we use this result to briefly sketch 
another proof of Theorem~\ref{thm:hrp}. 


\begin{lem} \label{lem:segrespects}
Let $n\geq 3$ and $b>d>0$. Let $R$ be an $n$-dimensional rectangle in $\R^n$
with side lengths at least $b$.  Let $\sP$ be a minimal locally regulated partition of $R$ 
with discreteness constant $d$. Suppose $\Int(R)\cap \partial \sP\neq\emptyset$. 
Then there is a line segment $s \subseteq R$ parallel to some $e_j$ 
such that $s$ is of length $\sqrt{bd}/(n-1)$ and $s \subseteq \partial \sP$. 
In fact, for some $e_i\neq e_j$ we have that $s \subseteq \partial_{e_i} \sP$. 
\end{lem}

\begin{proof} Let $c=\sqrt{bd}/(n-1)$. Without loss of generality 
we may assume that all $P\in \sP$ have side lengths less than $c$, since otherwise the lemma holds trivially.  
Let $M_1$ be a maximal surface at level $d_1$, where $\min \pi_n(R)<d_1<\min \pi_n(R)+c$.
Then $\sP$ respects $M_1$ from $d_1$ to some $d_2>d_1$, where $d_2$ is the largest such. If $d_2-d_1\geq c$ or if $M_1$ has a side length at least $c$
then we are done, so we assume otherwise. 
By Theorem~\ref{thm:maxsurface1}
there is a virtual maximal surface $M_2$, which is an $(n-1)$-dimensional rectangle  properly containing $M_1$, such that $\sP$ respects $M_2\times
[d_2,d_3]$ for some $d_3 > d_2$. Again, if $d_3-d_2\geq c$ then we are done. 
Note that at least one of the $n-1$ side lengths in the $e_1,\dots,e_{n-1}$ directions
has gotten strictly larger in going from $M_1$ to $M_2$. By the definition of the discreteness constant of $\sP$, the increase of the side length is at least $d$. If any of the side lengths has increased to be at least
$c$ then we are done, so we assume not. We repeat this argument until we find a virtual maximal surface $M_k$ at level $d_k$ where $d_k<\max\pi_n(R)$ is the largest such. We may assume that for all $1\leq j\leq k-1$, $d_{j+1}-d_j<c$, since otherwise we are done. This gives that 
$$ k\geq \left\lfloor\displaystyle\frac{b}{c}\right\rfloor\geq \frac{b}{c}-1. $$
Since there are $n-1$ side lengths of each $M_j$,
after $k$ many steps some $e_i$ side length of $M_k$ must have length at least 
$$ \left\lceil\displaystyle\frac{k}{n-1}\right\rceil\cdot d\geq \frac{k}{n-1}d\geq \left(\frac{b}{c}-1\right)\frac{d}{n-1}=\sqrt{bd}-\frac{d}{n-1}\geq \frac{\sqrt{bd}}{n-1}=c.$$
The proof of the lemma is complete.
\end{proof}

We next sketch an alternate proof of Theorem~\ref{thm:hrp} using Lemma~\ref{lem:segrespects}.
We show that the assumed minimal regulated partition $\sP$ of $F(2^{\Z^n})$ when 
restricted to the equivalence class of a Cohen generic real $x \in 2^{\Z^n}$ violates Lemma~\ref{lem:segrespects}
(considering the corresponding regulated partition $\sP'=\sC_\sP$ of $\R^n$ as in Definition~\ref{znrn2}).  
This will prove Theorem~\ref{thm:hrp}.

Suppose $\sP$ is a Borel, nondegenerate, minimal regulated partition of $\fzn$. 
Let $\bP$ be the Cohen forcing for adding an element of $2^{\Z^n}$, so conditions $p$ are 
elements of $2^R$ where $R=\dom(p)$ is a finite rectangle in $\Z^n$.
Let $x$ be a generic for $\bP$ which we identify with an element of $2^{\Z^n}$. 
Let $\sP_x$ be the restriction of the partition $\sP$ to 
the equivalence class of $x$, and let $\sP'_x=\mathcal{C}_{\sP_x}$ be the corresponding regulated partition of $\R^n$.

We begin by constructing an initial condition $p_0$. By the absoluteness of the nondegeneracy condition, we have that $\sP_x$ is nondegenerate, that is, for each direction $e_i$ there is a shift $s_i\cdot x$ such that 
$s_i\cdot x$ and $(s_i+e_i)\cdot x$ lie in the same  element of $\sP_x$.  
We may then easily get a $p_0$ with domain a rectangle in $\Z^n$ centered around $\vec 0$ of side lengths $d_0$ say, 
such that for all  directions $e_j$ and $e_i\neq e_j$ we have that $p_0$ forces 
``for some $a_j \in \Z$, $a_j e_j$ and $a_je_j+e_i$ are in $\dom(p_0)$ 
and belong to the same element of $\sP_{\dot{x}}$.''

Using $p_0$ we will construct a condition $q\in \bP$ which we will then extend to a generic $y\in F(2^{\Z^n})$.
We will have that $\dom(q)$ is a large $n$-dimensional rectangle $R$ with side lengths $b$ such that $q$ forces 
that there are no line segments $s \subseteq R$  of length $\sqrt{b}/(n-1)$ 
and directions $e_i$ not parallel to $s$ such that $s\subseteq \partial_{e_i} \sP'_{\dot{x}}$. 
Noting that the discreteness constant for $\sP'_x$ is $1$, this will
contradict Lemma~\ref{lem:segrespects} and complete the proof of Theorem~\ref{thm:hrp}.

Let $b$ be a large multiple of $d_0$ with $$b> \frac{\sqrt{b}}{n-1} \gg 10 n C^2 \cdot d_0^2.$$ Here $C$ is a constant depending
only on $d_0$ and $n$. We take $C$ to satisfy $C>2 d_0^n$.


To define the condition $q \in \bP$, we begin with an $n$-dimensional rectangle $Q=[0,b]^n$ in $\Z^n$ with side lengths $b$. 
For each integer $m$ with $1 \leq m \leq b$ which is a multiple of $2Cd_0$, consider
the hyperplane $H_m$ in $\Z^n$ given by the equation 
$$x_1+x_2+\cdots+x_n=m. $$
For each of the hyperplanes $H_m$ we do the following construction. Let 
$$m'=1+\left(\displaystyle\frac{m}{2Cd_0}\!\!\! \mod n\right). $$ 
For each $k \in \{1,\dots,n\}$ we define an auxiliary $T_k \in 2^{\Z^n}$ by specifying $T_k(c_1,\dots, c_n)$ for any given $\vec{c}=(c_1,\dots, c_n)\in \mathbb{Z}^n$. 
Let $(v_{\ell})_{\ell < d_0^{n-1}}$ enumerate the vectors $v$ in 
$$\Z e_1\oplus \cdots \oplus \Z e_{k-1}\oplus \Z e_{k+1} \oplus \cdots \oplus \Z e_n$$ 
where all the coefficients of $v$ are in $[0, d_0)$. 
Let $j= \lfloor c_k/d_0\rfloor \mod d_0^{n-1}$ and set $$T_k(c_1,\dots, c_n)= p_0( \vec c+v_{j}\!\! \mod (d_0,\dots,d_0) ).$$
For the given $m$, let $Q_m$ be the intersection of $T_{m'}$ with the 
$C d_0$ neighborhood of $H_m \cap Q$. Finally, let $q$ be be any extension to $\dom(Q)$ of 
$\bigcup_m Q_m$. Because the values of $m$ we are considering are spaced by at least 
$2C d_0$, we have that $Q_{m_1}\cap Q_{m_2}=\emptyset$ if $m_1\neq m_2$. Thus, $q$ is well defined. 
See Figure~\ref{fig:qcond} for an illustration of the construction of $q$.

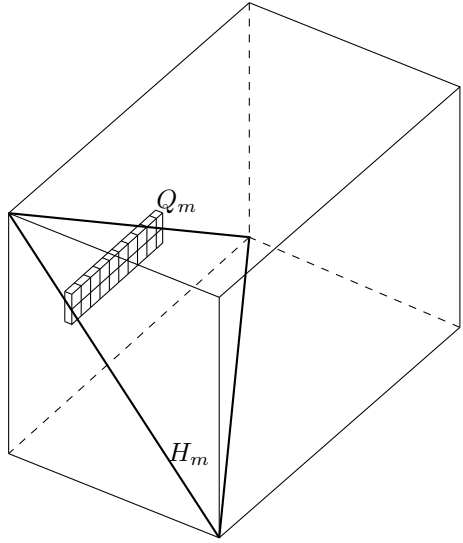
\begin{figure}[htbp]
\begin{tikzpicture}[scale=0.04]
\draw (0,0) to (70,-28) to (150,42);
\draw[dashed] (150,42) to (80,72);
\draw[dashed] (80,72) to (0,0);
\draw (0,0) to (0,80) to (70,52) to (70,-28);
\draw (70,52) to (150,122) to (150,42);
\draw (0, 80) to (80,150);
\draw[dashed] (80,150)  to (80,72);
\draw (150,122) to (80,150);

\draw[thick] (0,80) to (70,-28);
\draw[thick] (70,-28) to (80,72) to (0,80);

\draw (21,48) to (20+0.39*80, 47+0.39*72);
\draw (21,58-5) to (20+0.39*80, 47+0.39*72+10-5);

\draw (21,58-5) to (21, 38+5) to (20+0.39*80, 47+0.39*72-10+5) to (20+0.39*80, 47+0.39*72+10-5);

\foreach \i in {1,...,9}
{\draw (21+\i* 80*0.39/10, 38+\i* 72*0.39/10+5) to (21+\i* 80*0.39/10, 38+\i* 72*0.39/10+20-5);}

\draw (21,58-5) to (21- 70/30, 58+28/30-5);
\draw (21,38+5) to (21- 70/30, 38+28/30+5);
\draw (20+0.39*80, 47+0.39*72+10-5) to (20+0.39*80-70/30, 47+0.39*72+10+28/30-5);

\draw (21- 70/30, 58+28/30-5) to (21- 70/30, 38+28/30+5);

\draw (21- 70/30, 58+28/30-5) to (20+0.39*80-70/30, 47+0.39*72+10+28/30-5);

\foreach \i in {1,...,9}
{\draw (21+\i* 80*0.39/10, 38+\i* 72*0.39/10+20-5) to (21+\i* 80*0.39/10- 70/30,38+\i* 72*0.39/10+20+28/30-5);}

\node at (60,0) {$H_m$};
\node at (56,84) {$Q_m$};

\end{tikzpicture}
\caption{Construction of the condition $q$. One of the hyperplanes $H_m$ is shown 
along with a part of $Q_m$, which consists of copies of $p_0$.} \label{fig:qcond}
\end{figure}


Let $y \in 2^{\Z^n}$ be a generic element extending $q$. Then $y \in F(2^{\Z^n})$. Let $\sP$ be the 
Borel  minimal regulated partition of $F(2^{\Z^n})$ restricted to the equivalence class of $y$, which we identify with an 
 minimal regulated partition of $\Z^n$.  Let $\sP'$ be the corresponding
 minimal regulated partition of $\R^n$. Applying Lemma~\ref{lem:segrespects} to $\sP'$  and $\dom(q)$ (which
we can view as an $n$-dimensional rectangle in $\R^n$) we get a line segment $s'$ in $\dom(q)$ parallel to one of the
coordinate axes, say $e_j$, and an $e_i\neq e_j$ with $s'\subseteq\partial_{e_i} \sP'$ 
and with the length of $s'$ at least 
$\sqrt{b}/(n-1)$.

Let $s \subseteq \bigcup_{P\in \sP} \partial P$ be a segment in $\Z^n$ 
parallel to $s'$ and of length at least $|s'|-1$ and such that $|\pi_i(s)-\pi_i(s')|\leq 1/2$. 
So, for any point $u$ on $s$, the points $u$ and $u+e_i$ lie in different 
elements of $\sP$. 
The segment $s$ intersects at least 
$$\left(\frac{\sqrt{b}}{n-1}-1\right) \frac{1}{2Cd_0}>5nCd_0$$ many of the hyperplanes $H_m$. 
Thus there is an $m$ with 
$$ \frac{m}{2Cd_0} \equiv j-1 \!\! \mod n \mbox{ and } m \notin [0,n] \cup [b-n,n] $$
such that 
$$s \cap H_m \neq \emptyset \mbox{ and } [\pi_j(z)-Cd_0, \pi_j(z)+Cd_0] \subseteq \pi_j(s)$$ 
where $z$ is the point of intersection of $s$ with $H_m$. There is an $a < d_0^n\leq \frac{1}{2} Cd_0$
such that $z+a e_j$, which lies on $s$, is at the center of a copy of $p_0$ in 
$Q_m$ (note that $z+a e_j$ is within distance $Cd_0$ of $H_m$). 
This implies that there is a $c <d_0$ so that $z+(a+c)e_j$ and $z+(a+c)e_j+e_i$ 
lie in the same rectangle of $\sP$. Since $a+c<C d_0$, $z+(a+c)e_j$ lies on $s$ which contradicts 
the above property of $s$.

\end{document}